\documentclass{article}

\usepackage{lmodern}

\usepackage[utf8]{inputenc} 
\usepackage[T1]{fontenc}    
\usepackage{lmodern}

\usepackage{hyperref}       
\usepackage{url}            
\usepackage{booktabs}       
\usepackage{amsfonts}       
\usepackage{nicefrac}       
\usepackage{microtype}      

\usepackage{subfig}

\usepackage{verbatim}
\usepackage{amsmath}\allowdisplaybreaks
\usepackage{amsthm, amssymb}
\DeclareMathOperator*{\argmax}{argmax}
\DeclareMathOperator*{\argmin}{argmin}

\usepackage{ifthen}
\usepackage{color, xcolor}

\usepackage{hyperref}
\hypersetup{colorlinks=true, linkcolor=blue, citecolor=blue, anchorcolor=blue, urlcolor=blue}  
\usepackage{graphicx}

\newcommand{\compilehidecomments}{false}

\ifthenelse{ \equal{\compilehidecomments}{true} }{%
    \newcommand{\wei}[1]{}
    \newcommand{\haoyu}[1]{}
    \newcommand{\weizhong}[1]{}
}{
    \newcommand{\wei}[1]{{\color{blue!50!black}  [\text{Wei:} #1]}}
    \newcommand{\weizhong}[1]{{\color{red!60!black} [\text{Weizhong:} #1]}}
    \newcommand{\haoyu}[1]{{\color{brown!60!black} [\text{Haoyu:} #1]}}
}

\newcommand{\compilefullversion}{true}
\ifthenelse{\equal{\compilefullversion}{false}}{%
	\newcommand{\OnlyInFull}[1]{}
	\newcommand{\OnlyInShort}[1]{#1}
}{%
	\newcommand{\OnlyInFull}[1]{#1}%
	\newcommand{\OnlyInShort}[1]{}%
}%

\newtheorem{lemma}{Lemma}
\newtheorem{corollary}{Corollary}
\newtheorem{proposition}{Proposition}

\newtheorem{definition}{Definition}

\newcommand{\E}{\mathbb{E}}

\newcommand{\R}{\mathbb{R}}

\newcommand{\be}{\boldsymbol{e}}

\newcommand{\bu}{\boldsymbol{u}}

\newcommand{\bv}{\boldsymbol{v}}

\newcommand{\bx}{\boldsymbol{x}}
\newcommand{\by}{\boldsymbol{y}}
\newcommand{\bz}{\boldsymbol{z}}

\newcommand{\cA}{\mathcal{A}}
\newcommand{\cB}{\mathcal{B}}

\newcommand{\cD}{\mathcal{D}}

\newcommand{\cN}{\mathcal{N}}
\newcommand{\cR}{\mathcal{R}}

\newcommand{\cP}{{\mathcal{P}}}

\newcommand{\LB}{{\it LB}}

\mathchardef\mhyphen="2D
\newcommand{\GradRIS}{{\sf Grad\mhyphen RIS}}
\newcommand{\ProxGradRIS}{{\sf ProxGrad\mhyphen RIS}}
\newcommand{\ProxGradRISHEU}{{\sf ProxGrad\mhyphen RISHEU}}
\newcommand{\ProxGradOrg}{{\sf ProxGrad\mhyphen Org}}
\newcommand{\ProxGradOrgHEU}{{\sf ProxGrad\mhyphen OrgHEU}}

\newcommand{\UpperGradRIS}{{\sf UpperGrad\mhyphen RIS}}
\newcommand{\UpperGradRISHEU}{{\sf UpperGrad\mhyphen RISHEU}}
\newcommand{\GreedyRIS}{{\sf Greedy\mhyphen RIS}}
\newcommand{\Sampling}{{\sf Sampling}}

\usepackage{algorithm}
\usepackage{algorithmicx}
\usepackage{algpseudocode}
\algrenewcommand\algorithmicrequire{\textbf{Input:}}
\algrenewcommand\algorithmicensure{\textbf{Output:}}

\usepackage{thm-restate}
\usepackage{fullpage}

\title{Gradient Method for Continuous Influence Maximization with 
    Budget-Saving Considerations}

\author{Wei Chen\\
Microsoft Research\\
\texttt{weic@microsoft.com}
\and 
Weizhong Zhang\\
IIIS, Tsinghua University\\
\texttt{zwz15@mails.tsinghua.edu.cn}
\and
Haoyu Zhao\\
IIIS, Tsinghua University\\
\texttt{zhaohy16@mails.tsinghua.edu.cn}
}
\date{}

\begin{document}

\maketitle

\sloppy


\begin{abstract}
Continuous influence maximization (CIM) generalizes the original influence maximization by
	incorporating general marketing strategies: a marketing strategy mix is a vector 
	$\bx = (x_1,\dots,x_d)$ such that for each node $v$ in a social network, 
	$v$ could be activated as a seed of diffusion with probability $h_v(\bx)$, 
	where $h_v$ is a strategy activation function satisfying DR-submodularity.
CIM is the task of selecting a strategy mix $\bx$ with constraint $\sum_i x_i \le k$ where
	$k$ is a budget constraint, such that the total number of activated nodes after the
	diffusion process, called influence spread and denoted as $g(\bx)$, is maximized.
In this paper, we extend CIM to consider budget saving, that is, each strategy mix $\bx$ has 
	a cost $c(\bx)$ where $c$ is a convex cost function, and we want to maximize the balanced sum
	$g(\bx) + \lambda(k - c(\bx))$ where $\lambda$ is a balance parameter, subject to
	the constraint of $c(\bx) \le k$.
We denote this problem as CIM-BS.
The objective function of CIM-BS is neither monotone, nor DR-submodular or concave, and thus
	neither the greedy algorithm nor the standard result on gradient method could be directly
	applied.
Our key innovation is the combination of the gradient method with reverse influence sampling
	to design algorithms that solve CIM-BS: 
	 For the general case, we give an algorithm that achieves $\left(\frac{1}{2}-\varepsilon\right)$-approximation, and for the case of independent strategy activations, we present an algorithm that achieves $\left(1-\frac{1}{e}-\varepsilon\right)$ approximation.
\end{abstract}

\section{Introduction}
Influence maximization is the task of selecting a small number of seed nodes
	in a social network such that the influence spread from the seeds when following
	an influence diffusion model is maximized.
It models the viral marketing scenario and has been extensively studied
	(cf. \cite{kempe03journal,chen2013information,LiFWT18}).
Continuous influence maximization (CIM) generalizes the original influence maximization by
incorporating general marketing strategies: a marketing strategy mix is a vector 
$\bx = (x_1,\dots,x_d)$ such that for each node $v$ in a social network, 
$v$ could be activated as a seed of diffusion with probability $h_v(\bx)$, 
where $h_v$ is a strategy activation function satisfying monotonicity and DR-submodularity.
CIM is the task of selecting a strategy mix $\bx$ with constraint $\sum_i x_i \le k$ where
$k$ is a budget constraint, such that the total number of activated nodes after the
diffusion process, called influence spread and denoted as $g(\bx)$, is maximized.
CIM is proposed in~\cite{kempe03journal}, and recently followed up by
	a few studies~\cite{YangMPH16,WYC18}.
	
In this paper, we extend CIM to consider budget saving: 
	each strategy mix $\bx$ has 
	a cost $c(\bx)$ where $c$ is a convex cost function, and
	we want to maximize the balanced sum
	$g(\bx) + \lambda(k - c(\bx))$ where $\lambda$ is a balance parameter, subject to
	the constraint of $c(\bx) \le k$.
We denote this problem as CIM-BS.
The objective reflects the realistic consideration of balancing between increasing
	influence spread and saving marketing budget.
In general we have $g(\bx)$ monotone (increasing) 
	and DR-submodular (diminishing return property
	formally defined in Section~\ref{sec:model}), but $\lambda(k - c(\bx))$ is
	concave and likely to be monotonically decreasing, and thus
	the objective function $g(\bx) + \lambda(k - c(\bx))$ is neither monotone, nor DR-submodular or concave, and thus
neither the greedy algorithm nor the standard result on gradient method could be directly
applied for a theoretical guarantee.

In this paper, we apply the gradient method 
	\cite{nesterov2013introductory,parikh2014proximal}
	to solve CIM-BS with theoretical approximation
	guarantees.
This is the only case we know of that the gradient method is applied to influence maximization
	with a theoretical guarantee while the greedy method cannot.
The gradient method may be applied to the original objective function
	$g(\bx) + \lambda(k - c(\bx))$, but $g(\bx)$ is a complicated 
	combinatorial function and its exact gradient is infeasible to compute.
We could use stochastic gradient instead, but it results in large variance and very
	slow convergence.
Instead, we integrate the gradient method with the reverse influence sampling (RIS)
{	approach~\cite{BorgsBrautbarChayesLucier,tang14,tang15}, which is the main 
	technical innovation in our paper.
RIS is proposed for improving the efficiency of the influence maximization task,
	but when integrating with the gradient method, it brings two additional benefits:
	(a) it allows the efficient computation of the exact gradient of an estimator
	function of $g(\bx)$, which avoids slow convergence caused by the large variance
	in the stochastic gradient,
	and (b) for a class of independent strategy activation functions $h_v$ where
		each strategy dimension independently attempts to activate node $v$,
		the new objective function is in the form of coverage functions
	\cite{KarimiLH017}, which allows a tight concave upper bound function and leads to
	a better approximation ratio.

For the general case, we apply the proximal gradient method originally designed for
	concave functions to work with RIS and achieve an approximation of
	$\left(\frac{1}{2}-\varepsilon\right)$ (Theorem~\ref{thm-proxgrad-ris}).
This requires an adaptation of the proximal gradient method for the functions of 
	the form $f_1(\bx) + f_2(\bx)$ where $f_1$ is non-negative, monotone and DR-submodular
	and $f_2$ is non-negative and concave, and the result of this adaption
	(Theorem~\ref{thm-proximal-grad}) may be of independent interest.
For the independent strategy activation case, we apply the projected subgradient
	method on a tight concave upper bound of the objective function and
	achieve a $\left(1-\frac{1}{e}-\varepsilon\right)$ approximation (Theorem~\ref{thm-uppergrad-ris}).
We test our algorithms on a real-world dataset and validate its effectiveness comparing
	with other algorithms.

In summary, our contributions include:
	(a) we propose the study of CIM-BS problem to balance influence spread with budget
	saving; and
	(b) we integrate the gradient method with reverse influence sampling and provide
	two algorithms with theoretical approximation guarantees, on the objective function
	that is neither monotone, nor DR-submodular or concave.
Our study is one of the first studies that introduce the gradient method to influence
	maximization, and hopefully it will enrich the scope of the influence maximization 
	research.
	
\OnlyInFull{
For clarity, the detailed proofs and full experiment
results are moved to the appendix.
}
\OnlyInShort{
Due to the space constraint, detailed proofs and full experiment
	results are moved to the full technical report~\cite{CZZ19}.
}

\noindent{\bf Related works.\ \ }
Influence maximization was first proposed by Kempe et al. \cite{kempe03journal} as a discrete
	optimization problem, and has been extensively studied since 
	(cf. \cite{chen2013information,LiFWT18}).
CIM is also proposed in \cite{kempe03journal}.
Yang et al.~\cite{YangMPH16} propose heuristic algorithms to solve CIM more efficiently,
while Wu et al.~\cite{WYC18} consider discrete version of CIM and apply RIS to solve
	it efficiently.
Profit maximization in \cite{LL12,TTY16} introduces linear cost with no budget constraint
	to influence maximization.
Our CIM-BS problem is new and more general than both CIM and profit maximization
	studied before.
The RIS approach is originally proposed in~\cite{BorgsBrautbarChayesLucier}, and
	is further improved in~\cite{tang14,tang15,NguyenTD16}.
	
Recently, a number of studies have applied gradient methods to DR-submodular maximization:
	Bian et al.~\cite{BianMB017} apply the Frank-Wolfe algorithm to achieve 
	$1-1/e$ approximation for down-closed sets;
	Hassani et al.~\cite{HassaniSK17} apply stochastic projected gradient descent
		to achieve $1/2$ approximation;
	Karimi et al.~\cite{karimi2017stochastic} achieve
	$1-1/e$ approximation for coverage functions, which we adopt for the
	independent strategy activation case;
	Mokhtari et al.~\cite{DBLP:conf/aistats/MokhtariHK18} apply more complicated
	conditional gradient method to achieve $1-1/e$ approximation.
Our study is not a simple adoption of such methods to CIM-BS, because our
	objective function is not DR-submodular, and gradient computation cannot be treated
	as an oracle ---- we have to provide exact gradient computation and 
	an end-to-end integration with the RIS approach.

\section{Preliminary and Model} \label{sec:model}
In this paper, we focus on the triggering model for influence maximization problem. We use a directed graph $G=(V,E)$ to represent a social network, where $V$ is the set of nodes 
	representing individuals, and $E$ is the set of directed edges with edge $(u,v)$ representing that $u$ could directly influence $v$.
Let $n=|V|$ and $m=|E|$.
In the diffusion process, 
	each node is either active or inactive, and a node will stay active if it is activated. 
In the triggering model, every node $v\in V$ has a distribution $D_v$ on the
	subsets of $v$'s in-neighbors $N^-(v) = \{u|(u,v)\in E\}$. 
Before the diffusion starts, each node $v\in V$ samples a triggering set $T_v \subseteq N^{-}(v)$ from the distribution $D_v$, denoted $T_v\sim D_v$. 
At time $t=0$, the nodes in a pre-determined {\em seed set} $S$ are activated. 
For any time $t = 1, 2, \ldots$, the node $v$ is activated if at least one of nodes in its triggering set $T_v$ is activated at time $t-1$. 
The whole propagation stops when no new node is activated in a step.
An important quantity is the {\em influence spread} of the seed set $S$, denoted as $\sigma(S)$, 
	which is defined as the expected number of the final activated nodes with seed set $S$. 
The classical {\em influence maximization} 
	problem is to maximize $\sigma(S)$ such that $|S| \le k$
	for some given budget $k$.

A generalization of the classical influence maximization problem is the 
	{\em continuous influence maximization (CIM)} problem with general marketing strategies~\cite{kempe03journal}. 
A mix of marketing strategies is represented by a $d$-dimensional vector 
	$\bx =(x_1,x_2,\ldots, x_d) \in\R^d_{+}$, where $\R_{+}$ is the set of non-negative real numbers. 
In the general case, we consider strategy mix $\bx$ in a general convex set 
	$\cD \subseteq \R^d_{+}$, but most commonly we consider $\cD = \R^d_{+}$ or
	$\cD$ has an upper bound in each dimension, e.g. $\cD = [0,1]^d$.
Given the strategy $\bx$, each node $v\in V$ is independently activated as a seed with 
	probability $h_v(\bx)$, 
	where $h_v:\R^d_{+} \rightarrow [0,1]$ is referred to as a 
	{\em strategy activation function}.
Once a set of seeds $S$ is activated by a marketing strategy mix $\bx$, 
	the influence propagates from seeds in $S$ following the triggering model.
Then we define the influence spread of strategy mix $\bx$, $g(\bx)$, as the expected number of nodes
	activated by $\bx$, and formally,
\begin{align}
    &g(\bx) = \E_S[\sigma(S)]\nonumber\\
    =& \sum_{S\subseteq V}\sigma(S)\cdot \prod_{u\in S}h_u(\bx)\cdot \prod_{v\notin S}(1-h_v(\bx)).
\end{align}
The above formula means that we enumerate through all possible seed sets $S$, and due to independent seed activation by $\bx$
	the probability of $S$ being the seed set is $\prod_{u\in S}h_u(\bx)\cdot \prod_{v\notin S}(1-h_v(\bx))$ and its
	influence spread is $\sigma(S)$.

In many situations, each strategy dimension in $\bx$ activates each node independently.
That is, for each node $v$ and each strategy $j\in [d]$, there is a function $q_{v,j}$
	such that strategy $j$ with amount $x_j$ 
	activates node $v$ with probability $q_{v,j}(x_j)$. 
Then we have $h_v(\bx) = 1 - \prod_{j\in [d]} (1-q_{v,j}(x_j))$.
We call this case {\em independent strategy activation}.
Independent strategy activation models many scenarios
such as personalized marketing and event marketing~\cite{WYC18}.

In this paper, we focus on an extension of the continuous influence maximization problem --- 
	{\em continuous influence maximization with budget saving (CIM-BS)}. 
We have a total budge $k$, and for every strategy mix $\bx$, there is a cost $c(\bx)$. 
We do not want the cost to exceed the budget, and we want to maximize the budget balanced
	influence spread: a combination of the expected influence spread and the remaining budget. More formally, we have the following definition.

\begin{definition}[Continuous Influence Maximization with Budget Saving] 
The {\em continuous influence maximization with budget saving (CIM-BS)} is the problem of
	given (a) a social network $G=(V,E)$ and the triggering model $\{D_v\}_{v\in V}$ on $G$, 
	(b) strategy activation functions $\{h_v\}_{v\in V}$, 
	(c) cost function $c$, total budget $k$, and a balance parameter $\lambda \ge 0$, 
	finding a strategy mix $\bx^* \in \R^d_{+}$ to maximize its balanced sum of influence spread
	and budget savings, 
	i.e., find $\bx^*$ such that
    $\bx^* \in \argmax_{\bx\in \cD, c(\bx) \le k} \left(g(\bx)+\lambda(k-c(\bx))\right)$.
\end{definition}
Note that when $\lambda =0$ and $c(\bx)=\sum_{i\in [d]} x_i$, 
	the CIM-BS problem falls back to the CIM problem defined in~\cite{kempe03journal}, and also appears
	in~\cite{YangMPH16,WYC18}.
When $c(\bx)$ is a linear function and the budget constraint $c(\bx)\le k$ is dropped, the problem
	resembles the profit maximization studied in~\cite{LL12,TTY16}.
However, the general version of the problem as defined here with $\lambda > 0$, a general
	cost function $c(\bx)$ and constraint $c(\bx) \le k$ together is new.
Henceforth, let $s(\bx) = \lambda(k-c(\bx))$.
Intuitively, $s(\bx)$ is the budget-saving part of the objective. 
The overall objective of $g(\bx) + s(\bx)$ is trying to find the balance between maximizing
	influence and saving budget.
We call $g(\bx) + s(\bx)$ the {\em budget-balanced influence spread}.
For convenience, we denote $\cP = \{\bx\mid \bx\in \cD, c(\bx)\le k \}$.
	
We say that a vector function $f:\cD \rightarrow \R$ is DR-submodular if
	for any $\bx,\by \in \cD$ with $\bx \le \by$ (coordinate-wise), for any unit vector $\be_i$ with the $i$-th dimension $1$ and
	all other dimensions $0$, and for any $\delta > 0$, we have 
	$f(\bx + \delta\cdot \be_i) - b(\bx)\ge f(\by + \delta\cdot \be_i) - f(\by)$.
DR-submodularity characterizes the diminishing marginal return on function $f$ as vector $\bx$ increases, hence the name.
We also say that $f$ is monotone if for any $\bx,\by \in \cD$ with $\bx \le \by$, $f(\bx) \le f(\by)$; 
	$f$ is convex if for any $\bx,\by \in \cD$, any $\lambda \in [0,1]$, 
	$f(\lambda \bx + (1-\lambda) \by) \le \lambda f(\bx) + (1-\lambda) f(\by)$;
	$f$ is $L$-Lipschitz if for any $\bx,\by \in \cD$, $|f(\bx) - f(\by)| \le L \cdot ||\bx - \by||_2$, where $||\cdot ||_2$ is the vector 2-norm;
	$f$ is $\beta$-smooth if it has gradients everywhere and for any $\bx,\by \in \cD$,
	$||\nabla f(\bx)-\nabla f(\by)||_2 \le \beta ||\bx-\by||_2$.
Note that when the gradients exist, the $L$-Lipschitz condition is equivalent to
	$||\nabla f(\bx)||_2 \le L$ for all $\bx\in \cD$.
%
%
	
In this paper, we assume that the strategy activation function $h_v$ is monotone and DR-submodular, which implies that the influence spread function $g$ is also monotone
and DR-submodular, same as assumed
	in~\cite{kempe03journal,WYC18}.
It is reasonable in that, with more marketing effort, the probability of seed activation would increase (monotonicity) but 
	the marginal effect may be decreasing.
For the case of independent strategy activation ($h_v(\bx) = 1 - \prod_{j\in [d]} (1-q_{v,j}(x_j))$), we assume $q_{v,j}$ is 
	non-decreasing and concave, which implies that $h_v$ is monotone and DR-submodular~\cite{WYC18}.
In term of the cost function $c$, we assume that $c$ is convex and $L_c$-Lipschitz.
The most common function is the simple summation
(or 1-norm of $\bx$): $c(\bx) = \sum_{i\in [d]} x_i$, but 
	more general convex functions are also common in the economics literature
	(e.g. \cite{mankiw2014principles}), for example 
	$c(x) = ||\bx||_2$.

An important remark is now in order.
When $h_v$'s are monotone and DR-submodular and $c$ is convex,
	$g$ is monotone and DR-submodular and $s$ is concave, and as a result
	$g+s$ may be neither monotone nor DR-submodular.
This means the greedy hill-climbing algorithm of \cite{kempe03journal,WYC18} 
	no longer has theoretical approximation guarantee for the CIM-BS problem.
This motivates us to apply the gradient method to solve CIM-BS.

\section{Gradient Method with Reverse Influence Sampling}
\label{sec:gradientAlgo}

Gradient method has been applied to many continuous optimization problems.
For our CIM-BS problem, a natural option is to apply the gradient method directly
	on the objective function $g+s$.
However, the influence spread function $g$ is a complicated combinatorial function, such
	that its gradient $\nabla g$ is too complex to compute 
	in practice.
We could apply stochastic gradient on $g$ (see \OnlyInFull{Appendix~\ref{app:original}}\OnlyInShort{Appendix D in
	\cite{CZZ19}})
	but it has very large variance due to the significant amount
	of randomness from both strategies activating seeds and influence
	propagation from seeds, which leads to very slow convergence of the method.
Instead, in this section, we propose a novel integration of the gradient method with
	the reverse influence sampling (RIS) 
	approach~\cite{BorgsBrautbarChayesLucier} 
	for CIM-BS. 
The key insight is that RIS allows the efficient computation of the exact gradient
	of an alternative objective function $\hat{g}_{\cR}+s$ while maintaining an
	approximation guarantee of $1/2-\varepsilon$.
Moreover, when independent strategy activation is satisfied by the model, the
	alternative objective enables a tight concave upper bound, which leads to
	a $1-1/e-\varepsilon$ approximation.
	
In Section~\ref{sec-properties-rrset}, we first 
	review existing results on RIS with the continuous domain.
Then in Sections~\ref{sec:algoFramework} and~\ref{sec-gradient-method-algorithm}, 
	we present the gradient method, its integration with RIS, and its theoretical analysis,
	which are our main technical contribution.

\subsection{Properties of the Reverse Reachable Sets}\label{sec-properties-rrset}

The central concept in the RIS approach is the {\em reverse reachable set}, as defined below.

\begin{definition}[Reverse Reachable Set]
	Under the triggering model,
a {\em reverse reachable (RR) set} with a root node $v$, denoted $R_v$, is the
random set of nodes that $v$ reaches in one reverse propagation: sample all
triggering sets $T_u,u\in V$, such that edges $\{(w,u)|u\in V,w\in T_u\}$
together with nodes $V$ form a live-edge graph, and $R_v$ is the set of
nodes that can reach $v$ (or $v$ can reach reversely) in this live-edge
graph. An RR set $R$ without specifying a root is one with root $v$ selected
uniformly at random from $V$ .
\end{definition}

An RR set $R_v$  includes nodes that would activate $v$
	in one sample propagation.
Then, the key insight is that for a collection of RR sets, if some node $u$ appears in
	many of these RR sets, it means $u$ is likely to activate many nodes, and thus has
	high influence.
Technically, RR sets connect with the influence spread of a seed set $S$ with the
	following equation~\cite{BorgsBrautbarChayesLucier,tang15}:
	$\sigma(S) = \E_R[\Pr\{S \cap R \ne \emptyset \}]$.
For CIM-BS, we have the following connection as given in~\cite{WYC18}:
\begin{lemma}[\cite{WYC18}]
	For any strategy $\bx\in\cP$, we have
	$g(\bx)         = n\cdot \E_{R}\left[1-\prod_{u\in R}(1-h_u(\bx))\right]$.
\end{lemma}
Intuitively, the above lemma means that a node $u\in R$ would activate $R$'s root if $u$ itself is activated, which happens with probability $h_u(x)$, and
	thus strategy mix successfully activate $R$'s root with probability 
	$1-\prod_{u\in R}(1-h_u(\bx))$.
We can generate $\theta$ independent RR-sets $\cR = \{R_1,\dots,R_{\theta}\}$, and take the average among them as defined below:
\begin{equation} \label{eq:hatg}
\hat g_{\cR}(\bx) = \frac{n}{\theta}\sum_{R\in\cR}\left(1-\prod_{v\in R}(1-h_v(\bx))\right).
\end{equation}
We can see that $\hat{g}_{\cR}(\bx)$ is an unbiased estimator of $g(\bx)$. If $\theta$ is large enough, then $\hat{g}_{\cR}(\bx)$ should be close to $g(\bx)$ at every $\bx\in\cP$.
We also have the following lemma from \cite{WYC18}.
\begin{lemma}[\cite{WYC18}]
	\label{lem-drsubmodular}
    If $h_v$ is monotone and DR-submodular for all $v\in V$, then 
    $g$ and $\hat g_{\cR}$ are also monotone and DR-submodular.
\end{lemma}

\subsection{Algorithmic Framework Integrating Gradient Method with RIS}
\label{sec:algoFramework}

We first introduce the general algorithmic framework that integrates any gradient
	algorithm with the RIS approach.
We assume a generic gradient algorithm $\cA$ that takes a set of RR sets
	$\cR = \{R_1,\dots,R_{\theta}\}$, an objective function 
	$\hat g_{\cR}+s$, and an additive error $\varepsilon$ as input,
	and return a solution $\hat \bx$ that guarantees
$(\hat g_{\cR}+s)(\hat \bx) \ge \alpha\cdot \max_{\by\in\cP}(\hat g_{\cR}+s)(\by) - \varepsilon$,
in time $T = \text{poly}(\frac{1}{\varepsilon})$, where $\alpha$ is some
	constant approximation ratio.  
We call such an $\cA$ an {\em $(\alpha,\epsilon)$-approximate gradient algorithm}.

\begin{algorithm}[!t]
	\caption{$\GradRIS$: Gradient-RIS Meta-Algorithm for CIM-BS.}
	\label{alg-meta}
	\begin{algorithmic}[1]
		\Require Directed graph $G$, triggering model $\{D_v\}_{v\in V}$, domain $\cP$,
		strategy activation functions $\{h_v\}_{v\in V}$, cost function $c$, budget $k$, 
		balance parameter $\lambda$,
		Lipschitz constants $L_1,L_2$ for the functions $g+s$ and $\hat g_{\cR}+s$, approximation parameter $\varepsilon$, confidence parameter $\ell$, gradient algorithm $\cA$ with  the approximation guarantee $\alpha$
		\Ensure A strategy mix $\bx$
		\State $\cR,\LB \leftarrow \Sampling(\text{all parameters received})$
		\State $\bx\leftarrow \cA(\cR, \hat g_{\cR}+s, \varepsilon \LB)$
		/* $\hat g_{\cR}$ defined in Eq.~\eqref{eq:hatg},
		$s(\bx) = \lambda(k - c(\bx))$ */
		\State\Return $\bx$  
	\end{algorithmic}
\end{algorithm}
\begin{algorithm}[!t]
	\caption{$\Sampling$ Procedure.}
	\label{alg-sampling}
	\begin{algorithmic}[1]
		\Require Same as in Algorithm~\ref{alg-meta}
		\Ensure The RR-sets $\cR$ and an estimated lower bound $\LB$
		\State $\LB \leftarrow  1,\cR_0 \leftarrow \phi,\theta_0 = 0,\varepsilon'\leftarrow\sqrt{2}\varepsilon/3$
		\For{$i=1,2,\dots,\lfloor\log_2 (n+\lambda k)-1\rfloor$}
		\State $x_i \leftarrow (n+\lambda k)/2^i$
		\State \label{alg:thetai}
		$\theta_i \leftarrow \left\lceil n\left(2+\frac{2}{3}\varepsilon'\right)\cdot  \right.$
		\State $\quad\quad \left. \frac{\ln\cN(\cP,\frac{\varepsilon/3}{L_2} x_i)+\ell\ln n+\ln 2 + \ln\log_2{(n+\lambda k)}}{\varepsilon'^2 x_i} \right\rceil$
		\State Generate $\theta_i-\theta_{i-1}$ independent RR-sets $\cR'$ \label{line:generateRRset1}
		\State $\cR_i\leftarrow \cR'\cup\cR_{i-1}$
		\State \label{alg:callGradient}
		$\by_i \leftarrow \cA(\cR_i, \hat g_{\cR_i}+s, \varepsilon x_i/3)$
		\If{$(\hat g_{\cR_i} + s)(\by_i) \ge (1+\varepsilon'+\varepsilon/3)\cdot x_i$}
		\State $\LB \leftarrow \frac{(\hat g_{\cR_i} + s)(\by_i)}{1+\varepsilon'+\varepsilon/3}$; $\theta \leftarrow \theta_i$
		\State \textbf{break}
		\EndIf
		\EndFor
		\State 
		$\theta^{(1)} = \frac{8n\cdot \ln\left(4n^\ell\right)}{\LB\cdot (\alpha-\varepsilon/3)^2\varepsilon^2/9}$
		\State 
		$\theta^{(2)} = \frac{2\alpha'\cdot n\cdot\ln\left(4n^\ell\cN(\cP,\frac{\varepsilon/3}{L_1+L_2}\LB)\right)}{(\varepsilon/3 - \frac{1}{4}(\alpha-\varepsilon/3)^2 \varepsilon/3)^2\LB}$ \label{alg:theta12} 
		\State $\tilde\theta \leftarrow \max\{\theta^{(1)},\theta^{(2)}\}$.
		\State Generate $\tilde\theta$ independent RR-sets $R_1,\dots,R_{\tilde\theta}$, $\cR\leftarrow\{R_1,\dots,R_{\tilde\theta}\}$ 
		\State\Return($\cR,\LB$)
	\end{algorithmic}
\end{algorithm}

Algorithm~\ref{alg-meta} gives the meta-algorithm.
It first calls the $\Sampling$ procedure to sample enough RR sets $\cR$, together with
	an estimated lower bound $\LB$ of the optimal solution.
Then it calls the gradient algorithm $\cA$ with $\cR$, using $\hat g_{\cR}+s$ as
	the objective function and $\varepsilon \LB$ as the additive error.

The $\Sampling$ procedure is to sample enough RR sets for the theoretical guarantee.
We adapt the sampling procedure of the IMM algorithm~\cite{tang15}, as shown
	in Algorithm~\ref{alg-sampling}.
We use the IMM sampling procedure mainly because of its clarity in analysis and theoretical
	guarantee, while other sampling procedures (e.g. ~\cite{NguyenTD16}) could be
	adapted too.
The main structure of the sampling procedure is the same as in IMM, where we repeatedly
	halving the guess $x_i$ of the lower bound $\LB$ of the optimal budget-balanced influence
	spread to find a good lower bound estimate, and then use $\LB$ to estimate the final
	number of RR sets needed 
	and regenerate these RR sets (the regeneration is the workaround 1 proposed
	in~\cite{Chen18} to fix a bug in the original IMM).
There are two important differences worth to mention.
First, in line~\ref{alg:callGradient}, we call the gradient algorithm $\cA$ to find
	an approximate solution $\by_i$, which replaces the original greedy algorithm in IMM.
Second, and more importantly, the original IMM algorithm works on a finite solution
	space --- at most $\binom{n}{k}$ feasible seed sets of size $k$, 
	and $\binom{n}{k}$ is used to bound the number of RR sets needed.
However, in CIM-BS, we are working
	on an infinite solution space, and thus we cannot directly have such a bound.
To tackle this problem, we utilize the concept of $\varepsilon$-net and covering number
	to turn the infinite solution space into a finite space:
%
%
%
\begin{definition}[$\varepsilon$-Net and Covering Number]
    A finite set $N$ is called an $\varepsilon$-net for $\cP$ if for every $\bx\in\cP$, there exists $\pi(\bx)\in N$ such that $||\bx - \pi(\bx)||_2\le\varepsilon$. The smallest cardinality of an $\varepsilon$-net for $\cP$ is called the covering number:
    $\cN(\cP,\varepsilon) = \inf\{|N|:N \text{ is an $\varepsilon$-net of $\cP$}\}$.
\end{definition}

As a concrete example, suppose we have 1-norm or 2-norm cost function
	$c(\bx) = ||\bx||_1 \text{or } ||\bx||_2$.
With budget $k$, we know that $\cP$ is bounded by the ball $\mathbb B_1(k)$
	or $\mathbb B_2(k)$ with radius $k$.
Then, As shown in \cite{van2014probability}, 
	the covering number satisfies
	$\cN(\cP,\varepsilon) \le \cN(\mathbb B_1(k),\varepsilon) \le \cN(\mathbb B_2(k),\varepsilon) \le \left(3k/\varepsilon\right)^d$. 

Besides the $\varepsilon$-net, we also need to have the upper bounds 
	$L_1$ and $L_2$ on the Lipschitz constants of functions $g+s$ and $\hat g_{\cR}+s$. 
We defer the discussion on $L_1$ and $L_2$ to the next subsection.
Covering number and $L_1, L_2$ together are used to 
	bound the number of RR sets needed, as used in
	lines~\ref{alg:thetai} and~\ref{alg:theta12} of Algorithm~\ref{alg-sampling}.
We denote Algorithms~\ref{alg-meta} and~\ref{alg-sampling} together as
	$\GradRIS$, and we show that $\GradRIS$ achieves the following
	approximation guarantee:
\begin{restatable}{theorem}{thmcombining}\label{thm-combining-gradient-rrset}
	For any $\varepsilon, \ell, \alpha > 0$, 
	for any $(\alpha,\varepsilon/3)$-approximate 
	gradient algorithm $\cA$ for $\hat{g}_\cR+s$, 
	with probability at least $1-\frac{1}{n^\ell}$, $\GradRIS$ outputs a solution $\bx$ that is an $(\alpha-\varepsilon)$-approximation of the optimal solution $\text{OPT}_{g+s}$ of CIM-BS, i.e.
    $(g+s)(\bx) \ge (\alpha-\varepsilon)\text{OPT}_{g+s}$.
\end{restatable}

%
The proof of the above theorem follows the proof structure of IMM~\cite{tang15},
	where the number of the RR sets needed is carefully adapted to accommodate
	the covering number of the $\varepsilon$-net, and the Lipschitz constants of the
	objective functions.

\subsection{Gradient Algorithms}\label{sec-gradient-method-algorithm}
In this subsection, we will show two gradient algorithms that 
	approximately maximize the function $\hat g_{\cR}(\bx)+s(\bx)$.
They are the instantiations of the generic algorithm $\cA$ in 
	Section~\ref{sec:algoFramework}: the first one works on the general model 
	and uses proximal gradient to achieve $(\frac{1}{2},\varepsilon)$-approximation, while the second one
	works on the special case of independent strategy activation and 
	uses gradient on a concave upper bound to achieve $(1-\frac{1}{e},\varepsilon)$-approximation.


\noindent{\bf General Case: $\ProxGradRIS$.\ \ }
We first consider the general case where the strategy activation functions $h_v$'s  
	are monotone, DR-submodular, $L_h$-Lipschitz and $\beta_h$-smooth, and
	the cost function $c$ is convex and $L_c$-Lipschitz.
In this case, we have that $g$ and $\hat{g}_{\cR}$ are monotone and DR-submodular
	(Lemma~\ref{lem-drsubmodular}), and
	budget-saving function $s$ is concave.
To solve this problem, we adapt the (stochastic) proximal gradient algorithm
	\cite{parikh2014proximal,nitanda2014stochastic} to provide a
	$\frac{1}{2}$-approximate solution to the following problem:
	given a convex set $\cP$, a $\beta$-smooth, non-negative, monotone, and DR-submodular function $f_1(\bx)$ on $\cP$ and a non-negative and concave function $f_2(\bx)$ 
	on $\cP$, find a solution in $\cP$ maximizing $f_1(\bx)+f_2(\bx)$.
The original proximal gradient is for the case when both $f_1$ and $f_2$ are concave,
	and we adapt it to the case when $f_1$ is monotone and DR-submodular to provide
	an approximate solution.
The reason we use proximal gradient is that our budget-saving function $s$
	may not be smooth (e.g. when the cost function is the 2-norm function).
We present the general solution first, since it may be of independent interest.
The following is the iteration procedure for the stochastic proximal gradient algorithm.
\begin{equation}\label{equ-proximal-gradient}
\left\{
\begin{array}{l}
    \bx^{(t+1)} = \text{prox}_{-\eta_t f_2}(\bx^{(t)} + \eta_t \bv^{(t)}),\\ 
    \quad\quad\text{where } \E[\bv^{(t)}] = \nabla f_1(\bx^{(t)}),\\
    \text{prox}_\phi(\bx) := \argmin_{\by\in\cP} (\phi(\by) + \frac{1}{2}||\bx-\by||^2_2),\\
    \quad\quad\text{for any convex function } \phi,
\end{array}
\right.
\end{equation}
where $\eta_t \le \frac{1}{\beta}$ is the step size and $\bv^{(t)}$ is the stochastic gradient at $\bx^{(t)}$ .
We use $\Delta$ to denote an upper bound of the diameter of $\cP$, i.e. $\Delta \ge \max_{\bx,\by\in \cP} ||\bx - \by||_2$.
The following is the main result for the above stochastic proximal gradient algorithm,
	with its proof adapted from the original proof.
%
\begin{restatable}{theorem}{thmproximalconverge}\label{thm-proximal-grad}
    Suppose that $\cP$ is a convex set, function $f_1(\bx)$ is
    $\beta$-smooth, non-negative, monotone, and DR-submodular on $\cP$, $f_2(\bx)$ is non-negative and concave on $\cP$. Let $\bx^*$ be the point that maximizes $f_1(\bx)+f_2(\bx)$. 
    Suppose that for some $\sigma > 0$, the stochastic gradient $\bv^{(t)}$ satisfies $\E||\bv^{(t)}-\nabla f_1(\bx^{(t)})||_2^2 \le \sigma^2$ for all $t$,
    then for all $T > 0$, 
     if we set $\eta_t = \eta = 1/(\beta + \frac{\sigma}{\Delta}\sqrt{2T})$,
     and iterate as shown in (\ref{equ-proximal-gradient}) starting from $\bx^{(0)}\in\cP$,
     we have
\begin{align*}
	&\E\left[\max_{t=0,1,2,\dots,T} (f_1+f_2)(\bx^{(t)})\right]\\
	&\ \ge \frac{1}{2}(f_1+f_2)(\bx^*) - \frac{\beta \Delta^2}{4T} - \frac{\sigma \Delta}{\sqrt{2T}}.
\end{align*}
%
%
\end{restatable}




Note that if we use exact gradient instead of the stochastic gradient, we simply set
	$\sigma=0$ in the above theorem.
To apply the proximal gradient algorithm and Theorem~\ref{thm-proximal-grad} to
	maximize $\hat g_{\cR}+s$, we compute the exact gradient of $\hat g_{\cR}$ and
	also derive the Lipschitz and smoothness constants, as shown below.
For RR set sequence $\cR=\{R_1, \ldots, R_\theta \}$, let
	$\nu^{(1)}(\cR) = \sum_{R\in \cR} |R|/\theta$ be the average RR set size in $\cR$,
	$\nu^{(2)}(\cR) = \sum_{R\in \cR} |R|^2/\theta$ be the average squared size in $\cR$, and $\nu^{(3)}(\cR) = \sum_{R\in \cR} |R|^3/\theta$ be the average cubed size.
\begin{lemma} \label{lem:generalGradient}
If functions $h_v(\bx)$'s are $L_h$-Lipschitz, then function $\hat g_{\cR}(\bx)$ 
	is $(\nu^{(1)}(\cR)nL_h)$-Lipschitz, and function $g(\bx)$ is $(n^2L_h)$-Lipschitz. 
If functions $h_v(\bx)$'s are $\beta_h$-smooth, then function $\hat g_{\cR}(\bx)$ is
	$(\nu^{(1)}(\cR) n \beta_h + \nu^{(2)}(\cR) n L_h^2)$-smooth.
The gradient of function $\hat g_{\cR}(\bx)$ is
\begin{equation} \label{eq:gradienthv}
\nabla \hat g_{\cR}(\bx) = \frac{n}{\theta}\sum_{R\in\cR,v'\in R}\nabla h_{v'}(\bx)\prod_{v\in R,v\neq v'}(1-h_v(\bx)).
\end{equation}
\end{lemma}
With Lemma~\ref{lem:generalGradient} and Theorem~\ref{thm-proximal-grad}, we can conclude
 the gradient algorithm $\cA$ working with $\GradRIS$ with the following settings:
 (a) we use the proximal gradient iteration given in Eq.~\eqref{equ-proximal-gradient},
 with stochastic gradient $\bv^{(t)}$ replaced with the exact gradient
 $\nabla \hat g_{\cR}(\bx^{(t)})$ as given in Eq.~\eqref{eq:gradienthv};
 (b) we set step size $\eta_t = 1/(\nu^{(1)}(\cR) n \beta_h + \nu^{(2)}(\cR) n L_h^2)$
	when calling the algorithm with $\cR$;
(c) we set number of steps $T=3(\nu^{(1)}(\cR) n \beta_h + \nu^{(2)}(\cR) n L_h^2)\cdot \Delta^2/4\varepsilon$; 
(d) we use $L_1 = L_2 = n^2 L_h+\lambda L_c$ (since $n \ge \nu^{(1)}(\cR)$) as parameters in $\GradRIS$.
We refer to the full algorithm with the above setting as $\ProxGradRIS$.
The following theorem summarizes the approximation guarantee of $\ProxGradRIS$.
%
%
\begin{restatable}{theorem}{thmproxgradris}\label{thm-proxgrad-ris}
    For any $\varepsilon, \ell > 0$, 
	with probability at least $1-\frac{1}{n^\ell}$, $\ProxGradRIS$ outputs a solution $\bx$ that is a $(\frac{1}{2}-\varepsilon)$-approximation of the optimal solution $\text{OPT}_{g+s}$ of CIM-BS, i.e.
    $(g+s)(\bx) \ge (\frac{1}{2}-\varepsilon)\text{OPT}_{g+s}$.
\end{restatable}

We remark that the actual computation of the
	proximal step $\text{prox}_{-\eta_t f_2}(\cdot)$ in 
	Eq.\eqref{equ-proximal-gradient} depends on domain $\cD$ and cost function $c$.
When $\cD = \R_{+}^d$ and $c$ is 1-norm or 2-norm function, we can derive efficient algorithm
	for the proximal step, as summarized below.
%
%
\begin{restatable}{lemma}{lemproximalonenorm}\label{lem-proximal-onenorm-twonorm}
    When $c(\bx) = ||\bx||_1$ and $\cD = \R_{+}^d$, the proximal step can be done in time $O(d\log d)$. When $c(\bx) = ||\bx||_2$ and $\cD = \R_{+}^d$, the proximal step can be done in $O(d)$.
\end{restatable}

\noindent{\bf Independent Strategy Activation Case: $\UpperGradRIS$.\ \ }
Next, we introduce an $\left(1-\frac{1}{e}\right)$-approximation for maximizing $\hat g_{\cR}+s$ under the case of independent strategy activation. 
Recall that in the independent strategy activation case, each function $h_v(\bx) = 1 - \prod_{j\in [d]} (1-q_{v,j}(x_j))$, where $q_{v,j}(x_j)$ is monotone and concave in $x_j$. 
In this case, we can write $\hat g(\bx)$ into the following form.
\begin{align*}
&\hat g_{\cR}(\bx) =\frac{n}{\theta}\sum_{R\in\cR}\left(1- \prod_{v\in R}(1-h_v(\bx))\right)\\
=& \frac{n}{\theta}\sum_{R\in\cR}\left(1- \prod_{v\in R}(\prod_{j\in [d]}(1-q_{v,j}(\bx))\right).
\end{align*}
The above form makes $\hat g_{\cR}(\bx)$ belong to coverage functions, which
	has the following concave upper and lower bounds (\cite{karimi2017stochastic}):
\begin{proposition}[\cite{karimi2017stochastic}]\label{prop-concave-upperbound}
    For any $\bx\in [0,1]^l$, let $\alpha(\bx) = 1 - \prod_{i=1}^l(1-x_i)$ and $\beta(\bx) = \min\{1,\sum_{i=1}^l x_i\}$, we have $(1-1/e)\beta(\bx) \le \alpha(\bx) \le \beta(\bx)$.
\end{proposition}

By Proposition~\ref{prop-concave-upperbound}, we can optimize $\bar g_{\cR} + s$ where $\bar g_{\cR}(\bx)$ is
	the upper bound of $\hat g_{\cR}(\bx)$ defined as: 
\begin{equation}
    \bar g_{\cR}(\bx) := \frac{n}{\theta}\sum_{R\in\cR}\min\{1,\sum_{j\in [d], v\in R}q_{v,j}(\bx)\}.
\end{equation}

Since the function $g_{\cR}(\bx)$ is non-smooth, we use the projected subgradient method to maximize the function $\bar g_{\cR}+s$~\cite{nesterov2013introductory}, as summarized by the following lemma.

\begin{restatable}{lemma}{lemconcaveupperbound}\label{lem-concave-upperbound}
    In the case of independent strategy activation, suppose that $\bar g_{\cR}+s$ is $L_{\bar g+s}$-Lipschitz. If we use projected subgradient descent to optimize the function $(\bar g_{\cR}+s)(\bx)$ with step size $\eta_t = \frac{\Delta}{L_{\bar g+s}\sqrt{t}}$ and let $\by$ denote the output where $(\bar g_{\cR}+s)(\by) \ge \max_{\bx\in\cP}(\bar g_{\cR}+s)(\bx) - \varepsilon$. Then
    $(\hat g_{\cR}+s)(\by) \ge \left(1-\frac{1}{e}\right)\max_{\bx\in\cP}(\hat g_{\cR}+s)(\bx) - \varepsilon$,
    and $\by$ can be solved in $\frac{(\Delta L_{\bar g+s})^2}{\varepsilon^2}$ iterations.
\end{restatable}

The following lemma presents the Lipschitz constants and subgradients needed in
	Lemma~\ref{lem-concave-upperbound}.
\begin{lemma}\label{lem:independentGradient}
	Suppose that functions $q_{v,j}(x_j)$'s are $L_q$-Lipschitz, then 
	function $g(\bx)$ is $n^2\sqrt{d}L_q$-Lipschitz, and 
	functions $\hat g_{\cR}(\bx)$ and $\bar g_{\cR}(\bx)$ are
	$\nu^{(1)}(\cR) n \sqrt{d}L_q$-Lipschitz.
	The subgradient of the function $\bar g_{\cR}(\bx)$ is
	\begin{align} \label{eq:subgradientbarg}
	\frac{n}{\theta}\sum_{R\in \cR}\left\{
	\begin{aligned}
	0 & \text{, if }\sum_{j\in [d],v\in R} q_{v,j}(x_j)\ge 1, \\
	\sum_{v\in R, j\in [d]}\nabla q_{v,j}(x_j) & \text{, if }\sum_{j\in [d],v\in R} q_{v,j}(x_j)<1.
	\end{aligned}
	\right.
	\end{align}
\end{lemma}
Combining Lemma~\ref{lem:independentGradient} with Lemma~\ref{lem-concave-upperbound},
	we can conclude our subgradient algorithm based on the upper bound function
	$\bar g_{\cR} + s$:
(a) we use the projected subgradient algorithm with the subgradient of $\bar g_{\cR}$
	given in Eq.\eqref{eq:subgradientbarg};
(b) we set step size $\eta_t = \Delta / (\nu^{(1)}(\cR) n \sqrt{d}L_q\sqrt{t})$;
(c) we use $T = 9(\Delta \nu^{(1)}(\cR) n \sqrt{d}L_q + \lambda L_c)^2/\varepsilon^2$ iterations to get $\varepsilon$ accuracy; and
(d) we set $L_1=L_2 = n^2\sqrt{d}L_q + \lambda L_c$ in $\GradRIS$.
We refer to the full algorithm with the above setting as $\UpperGradRIS$.
The following theorem summarizes the approximation guarantee of $\UpperGradRIS$.

\begin{restatable}{theorem}{thmuppergradris}\label{thm-uppergrad-ris}
    For any $\varepsilon, \ell > 0$, with probability at least $1-\frac{1}{n^\ell}$, $\UpperGradRIS$ outputs a solution $\bx$ that is an $(1-1/e-\varepsilon)$-approximation of the optimal solution $\text{OPT}_{g+s}$ of CIM-BS, i.e.
    $(g+s)(\bx) \ge (1-1/e-\varepsilon)\text{OPT}_{g+s}$.
\end{restatable}

\noindent{\bf Total Time Complexity.\ \ }
For the time complexity of $\ProxGradRIS$ and $\UpperGradRIS$, we make the following
	reasonable assumptions:
	(1) the time for sampling a trigger set $T_v\sim D_v$ is proportional to the
		in-degree of $v$; 
	(2) the optimal influence spread $\max_{\bx\in\cP} g(\bx)$ among strategy mixes
		is at least the optimal single node influence spread $\max_{v\in V}\sigma(\{v\})$;
		and
	(3) $\lambda k \le n$, otherwise the budget saving is more important than 
	influencing the entire network, and CIM-BS problem no longer makes much sense.
The following theorem summarizes the time complexity result when 
	 $\cD = \R_{+}^d$ and $c(\bx) = ||\bx||_1$ or $c(\bx) = ||\bx||_2$. 
The more general result is given in \OnlyInShort{\cite{CZZ19}}\OnlyInFull{Appendix~\ref{app:timeComplexity}}.
	Notation $\tilde{O}(\cdot)$ ignores poly-logarithmic factors.
\begin{restatable}{theorem}{thmproximaltimenorm}\label{thm:timesimple}
    Suppose that $\cD = \R_{+}^d$ and $c(\bx) = ||\bx||_1 \text{ or } ||\bx||_2$, 
    $h_v(\bx)$'s are $L_h$-Lipschitz and $\beta_h$-smooth.
    If $\nabla h_v(\bx)$ can be computed in time $T_h$, the expected running time of $\ProxGradRIS$ is bounded by
    $    \tilde O\left(\frac{\beta_h n^2 + L_h^2 n^3}{\varepsilon}\cdot \frac{T_h(m+n)\cdot(d+\ell)}{\varepsilon^2}\right)$.
 Under independent strategy activation, if $q_{v,j}(x_j)$'s are $L_q$-Lipschitz
 and the gradient and function value of $q_{v,l}(x_j)$ can be computed in time $T_q$,
 the expected running time of $\UpperGradRIS$ is bounded by
$ \tilde O\left(\frac{n^4dL^2_q}{\varepsilon^2}\cdot \frac{T_q(m+n)\cdot(d+\ell)}{\varepsilon^2}\right)$.
\end{restatable}
From the time complexity result, we can see that the two gradient algorithms still have
	high-order dependency on the graph size.
This is mainly because we need the conservative bounds on the number of gradient algorithm
	iterations for the theoretical guarantee
	(terms $\frac{\beta_h n^2 + L_h^2 n^3}{\varepsilon}$ and $\frac{n^4dL^2_q}{\varepsilon^2}$).
In our actual algorithms, we already use $\nu^{(1)}(\cR)$ and $\nu^{(2)}(\cR)$ instead of
	$n$ and $n^2$ in the upper bound of the gradient decent steps 
	for $\hat{g}_{\cR} + s$, so our actual performance
	would be reduced by corresponding factors.
For details, please see \OnlyInShort{\cite{CZZ19}}\OnlyInFull{Appendix~\ref{app:timeMoments}} for the results and 
	discussions on using the moments of RR set size in the time complexity bounds.


\OnlyInFull{
\section{Experiments}
\label{app:experiment}

\begin{figure*}[!ht]
    \centering
    \subfloat[$c(\bx)=||\bx||_1$,  $\lambda=5$]{\includegraphics[width=2.6in]{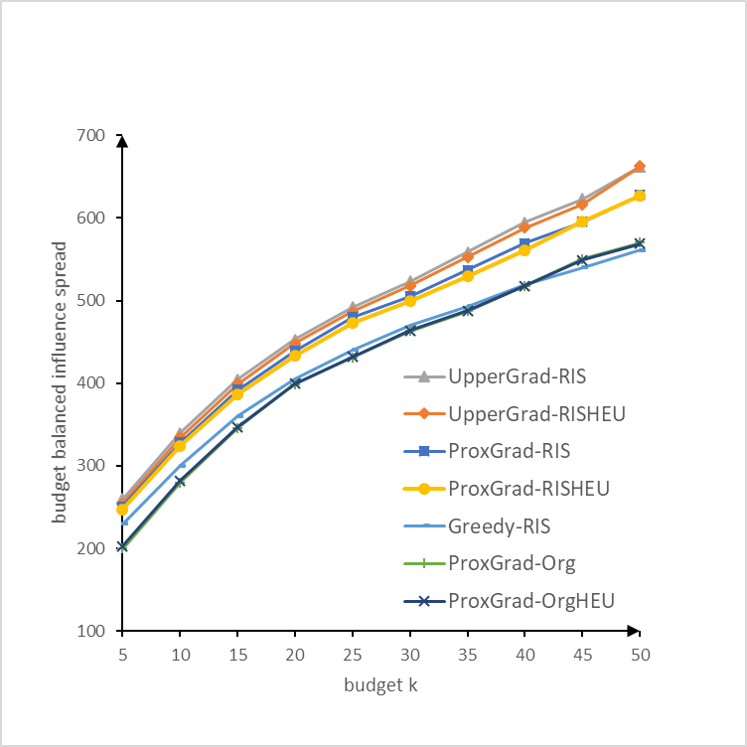}}\quad
    \subfloat[$c(\bx)=||\bx||_1$, $k=50$]{\includegraphics[width=2.6in]{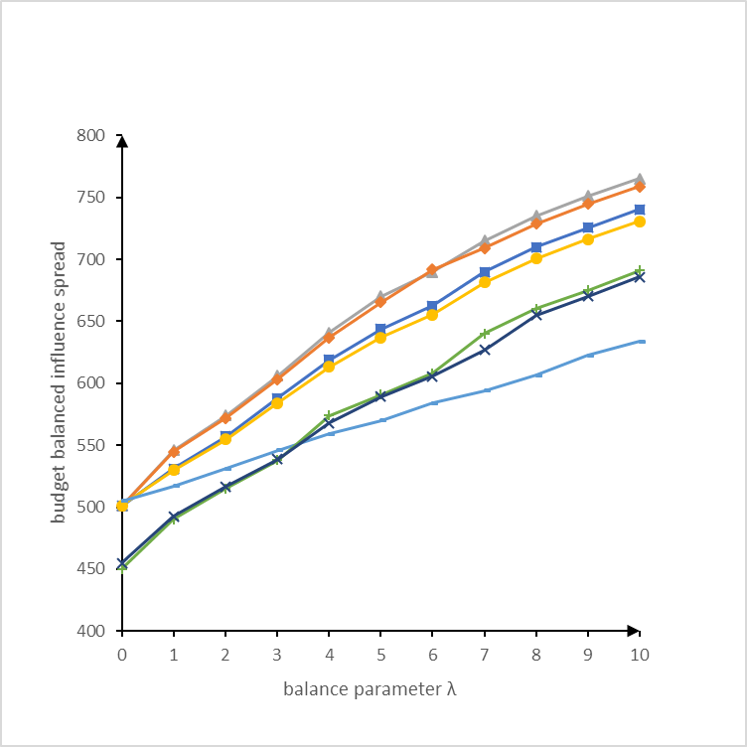}}\quad
    
    \subfloat[$c(\bx)=||\bx||_2$, $\lambda=50$]{\includegraphics[width=2.6in]{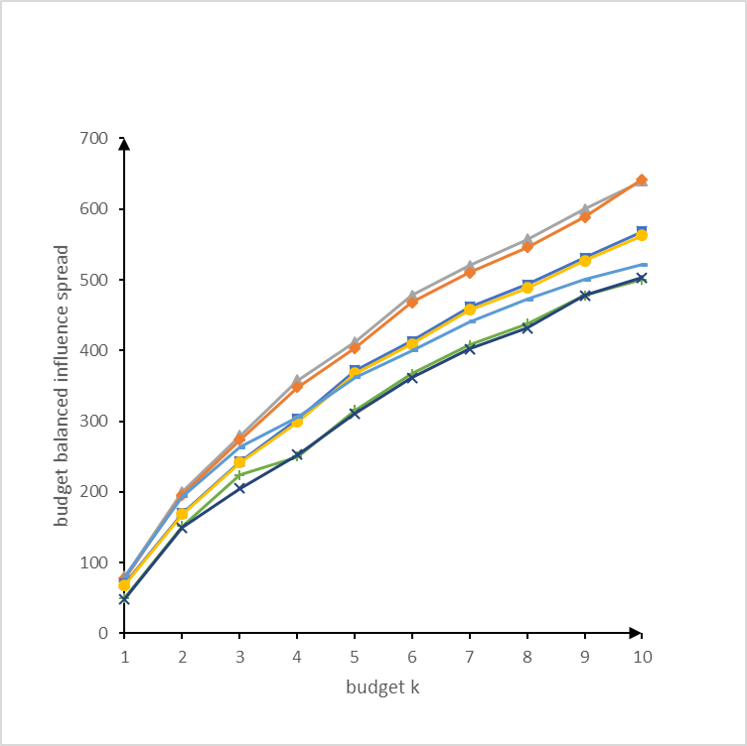}} \quad
    \subfloat[$c(\bx)=||\bx||_2$, $k=5$]{\includegraphics[width=2.6in]{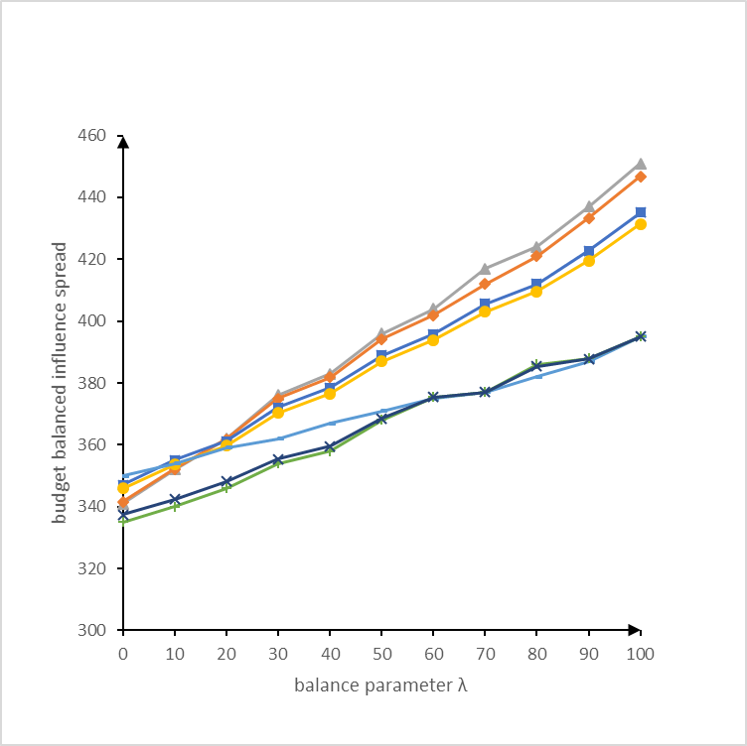}}
    \caption{Budget balanced influence spread results for the personalized marketing scenario on the DM dataset. 
        The legends shown in (a) apply to all other figures.
    } \label{fig:spread}
\end{figure*}

\begin{figure*}[!ht]
    \centering
    \subfloat[$c(\bx)=||\bx||_1$,  $\lambda=5$]{\includegraphics[width=2.6in]{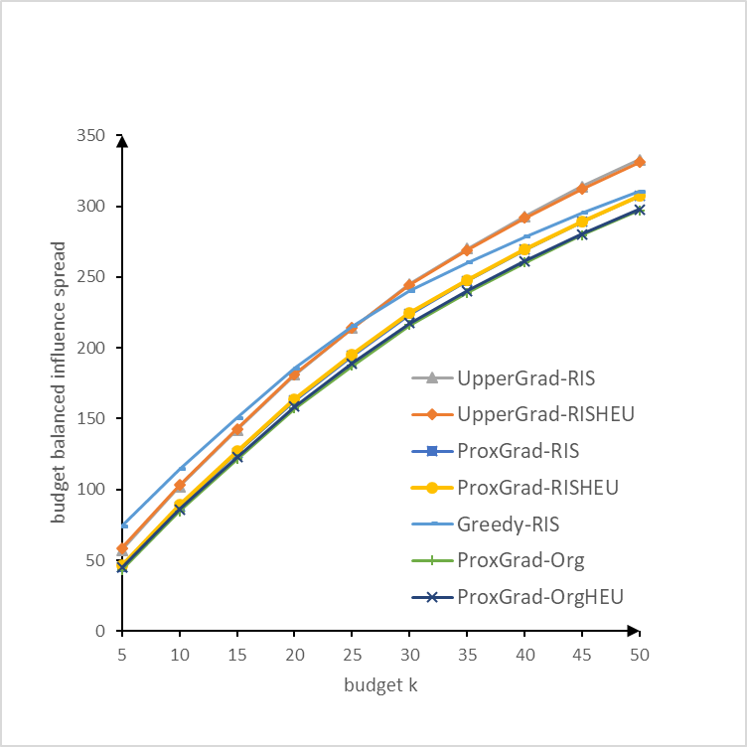}}\quad
    \subfloat[$c(\bx)=||\bx||_1$, $k=50$]{\includegraphics[width=2.6in]{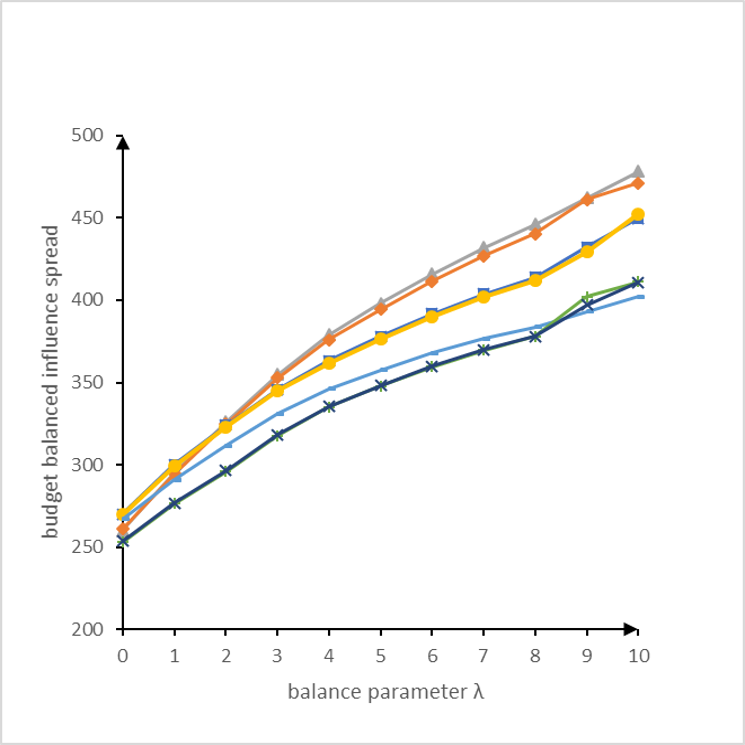}}\quad
    
    \subfloat[$c(\bx)=||\bx||_2$, $\lambda=50$]{\includegraphics[width=2.6in]{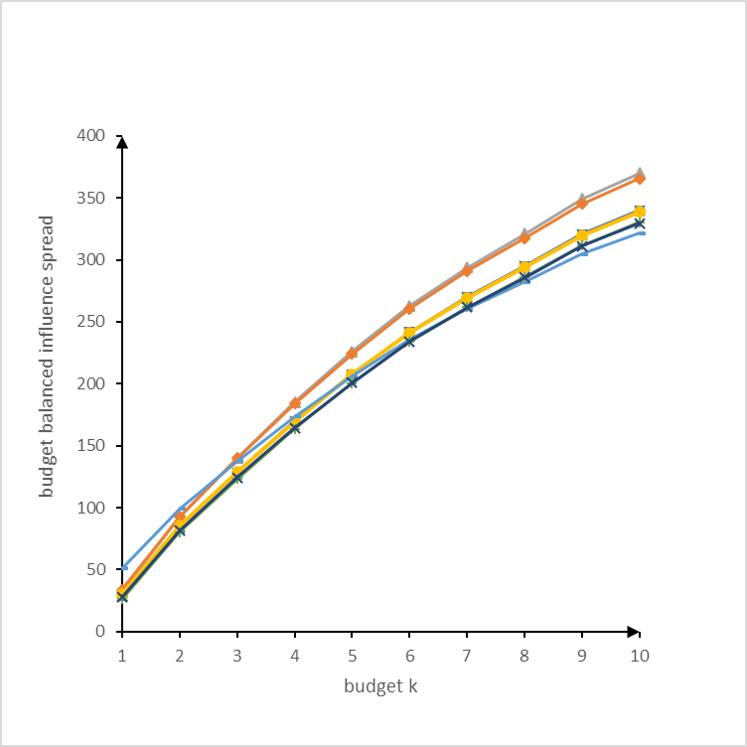}} \quad
    \subfloat[$c(\bx)=||\bx||_2$, $k=5$]{\includegraphics[width=2.6in]{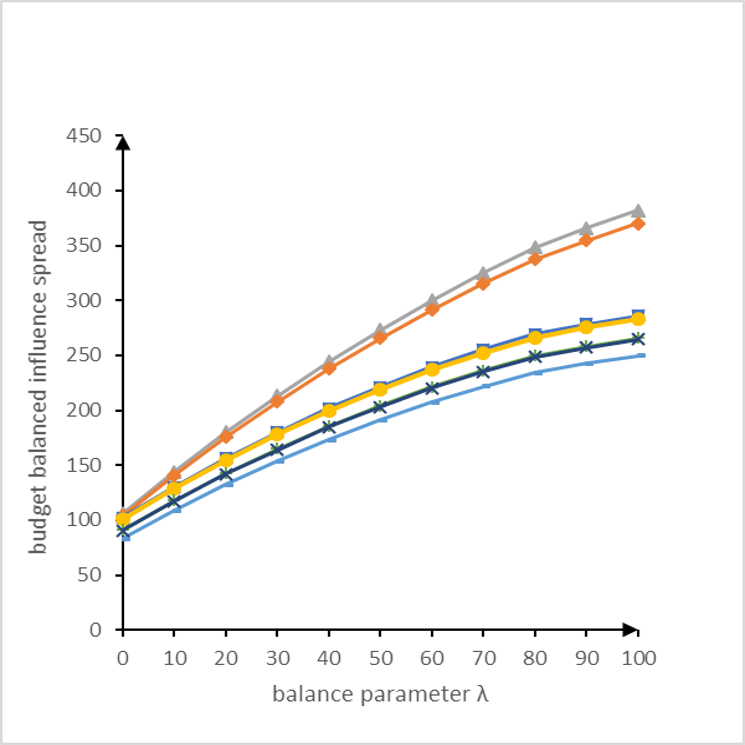}}
    \caption{Budget balanced influence spread results for the segment marketing scenario on the DM dataset. 
        The legends shown in (a) apply to all other figures.
    } \label{fig:eventspread}
\end{figure*}

\noindent{\bf Experiment setup.\ \ }
We test on two network dataset. The first dataset is the DM network, which is a network of data mining researchers
extracted from the ArnetMiner archive (arnetminer.org),
with  $679$ nodes and $3,374$ edges, and edge weights are learned from
a topic affinity model and obtained from the authors \cite{TangSWY09}.
The second dataset is NetHEPT, a popular dataset used in many influence maximization studies (e.g.~\cite{ChenWY09,WCW12,tang15}).
It is an academic collaboration network from the ``High Energy Physics Theory'' section of arXiv from 1991 to 2003, 
where nodes represent the authors and each edge represents one paper co-authored by two nodes. 
After removing duplicated edges, we have $15,233$ nodes and $62,774$ directed edges. 
The influence probabilities on edges are assigned according to the weighted cascade setting~\cite{kempe03journal}: the influence probability of edge $(u,v)$ 
	is $1/d_v$, where $d_v$ is the in-degree of $v$.

Besides our $\ProxGradRIS$ and $\UpperGradRIS$ algorithms, we test two more
algorithms: (a) $\ProxGradOrg$: stochastic proximal gradient algorithm on the
original objective function, and the stochastic gradient computation as well 
as step size and step count settings are given in Appendix~\ref{app:original}.
By Theorem~\ref{thm-proximal-grad}, $\ProxGradOrg$ would achieve $1/2$ approximation
in expectation. 
(b) $\GreedyRIS$: simply replace the gradient algorithm $\cA$ in $\GradRIS$ 
with the greedy algorithm for the objective 
$\hat{g}_{\cR}(\bx) + s(\bx)$ on generated RR sets $\cR$, and 
the greedy algorithm stops either when the budget
is exhausted or the marginal gain is negative. This is similar
to the algorithm in~\cite{WYC18}, but since $\hat{g}_{\cR}(\bx) + s(\bx)$ is
neither monotone nor DR-submodular, $\GreedyRIS$ has no theoretical guarantee and it
    is only a heuristic algorithm for our tests.
For the three gradient-based algorithms $\ProxGradRIS$, $\UpperGradRIS$, and $\ProxGradOrg$,
    we further test their heuristic versions that may lead to faster running time:
    instead of using a conservative number of iteration steps for theoretical guarantees,
    we heuristically terminate the gradient iteration if the difference in the
    objective function values 
    for two consecutive iterations is within a small value of $0.3$ (we will justify the choice of this
    parameter in our tests). 
We put suffix $\sf HEU$ for the three versions of the heuristic gradient termination algorithms.

For parameter settings, we set $\varepsilon=0.3$ and $\ell=1$ for all algorithms.
For $\GreedyRIS$, we set the greedy step size to be $0.1$ on each dimension.
For 1-norm cost function ($c(\bx) = ||\bx||_1$), we test
    (a) vary $k$ from $5$ to $50$ while keeping $\lambda=5$, and
    (b) vary $\lambda$ from $0$ to $10$ while keeping $k=50$.
For 2-norm cost function ($c(\bx) = ||\bx||_2$), we test
(c) varying vary $k$ from $1$ to $10$ while keeping $\lambda=50$, and
(d) vary $\lambda$ from $0$ to $100$ while keeping $k=5$.
The reason we use a smaller budget $k$ for 2-norm cost function is because
    $||\bx||_2 \le ||\bx||_1/\sqrt{d}$, and thus we need a significantly small budget
    for 2-norm in order to have a similar feasible region.
Parameter $\lambda$ is adjusted accordingly so that $\lambda\cdot k$ is at the same
scale as the influence spread, otherwise either influence spread or budget saving is dominant, and the
problem is degenerated.
For functions $h_v(\bx)$, we test two cases: the personalized marketing case and the segment marketing case~\cite{YangMPH16,WYC18}. 
In the personalized marketing case, each node $v$ 
receives a separate discount $x_v\in [0,1]$.
This corresponds to the independent strategy activation case with $d=n$, and
$q_{v,j}(x_j) >0$ only when $j=v$.
We set $q_{v,v}(x_v) = 2x_v - x_v^2$ as in~\cite{YangMPH16,WYC18}. 
In the segment marketing case, we have 10 strategies in total, i.e $d=10$, each strategy 
targets to a disjoint segment of users, and
each user has exactly one corresponding strategy.
Each user is randomly put into one of the 10 user segments with equal probability, 
    and if one segment in the end has less than 50 or larger than 80 users, 
    we regenerate the user segments, so that in the end all segments have sizes within 50 to 80. 
    
The experiments are run on a Ubuntu 17.04 server machine with 2.9GHz and 128GB  memory.
The code is written in C++ and compiled by g++.

\noindent{\bf Experimental results.\ \ }
We first show the results on the DM dataset.
Figure~\ref{fig:spread} shows the influence spread results of the personalized marketing scenario, and Figure~\ref{fig:eventspread} shows the influence results of 
    the segment marketing scenario, both on the DM dataset.
Each data point on an influence spread curve is the average of five solutions found by
    five runs of the same algorithm, and
    the influence spread of each solution is 
    an average of 1000 simulation runs.
In all cases, $\UpperGradRIS$/$\UpperGradRISHEU$ has the best performance, followed by
$\ProxGradRIS$/$\ProxGradRISHEU$, which coincides with our theoretical analysis that
$\UpperGradRIS$ has a better theoretical guarantee.
Both algorithms outperform two baselines in most cases, especially when $\lambda$ is
getting large.
Large $\lambda$ indicates that we need to pay more attention to budget saving, and thus
the result suggests that our algorithm
handles much better in the balance between influence spread and budget saving.
Comparing the heuristic termination version of each gradient-based algorithm with their
    corresponding theory-guided termination, the heuristic versions almost match the theoretical version in influences spread in all cases, showing that the heuristic termination seems to perform well in practice.

\begin{table}[h]
    \centering
    \small
    \caption{Running time results for the personalized marketing scenario on the DM dataset (in seconds).}
    \label{tab:time}
    \begin{tabular}{|l|r|r|}
        \hline
        &    $c(\bx) = ||\bx||_1$, $k=50$ and $\lambda=5$    & $c(\bx) = ||\bx||_2$, $k=5$ and $\lambda=50$ \\
        \hline
        $\ProxGradRIS$    & 33.2    & 27.3 \\ 
        \hline
        $\ProxGradRISHEU$ &    6.5 & 5.5  \\   
        \hline
        $\UpperGradRIS$ &    81.8 & 72.3  \\   
        \hline
        $\UpperGradRISHEU$ &    13.2 & 13.9  \\   
        \hline
        $\GreedyRIS$ &    10.2 & 8.7  \\   
        \hline
        $\ProxGradOrg$ &    1043.9 & 1021.6  \\   
        \hline
        $\ProxGradOrgHEU$ &    243.4 & 187.8  \\   
        \hline
    \end{tabular}
\end{table}

\begin{table}[h]
    \centering
    \small
    \caption{Running time results for the segment marketing scenario on the DM dataset (in seconds).}
    \label{tab:eventtime}
    \begin{tabular}{|l|r|r|}
        \hline
        &    $c(\bx) = ||\bx||_1$, $k=50$ and $\lambda=5$    & $c(\bx) = ||\bx||_2$, $k=5$ and $\lambda=50$ \\
        \hline
        $\ProxGradRIS$    & 0.58    & 0.46 \\ 
        \hline
        $\ProxGradRISHEU$ &    0.12 & 0.11  \\   
        \hline
        $\UpperGradRIS$ &    1.9 & 2.4  \\   
        \hline
        $\UpperGradRISHEU$ &    0.32 & 0.77  \\   
        \hline
        $\GreedyRIS$ &    0.12 & 0.13  \\   
        \hline
        $\ProxGradOrg$ &    28.6 & 19.2  \\   
        \hline
        $\ProxGradOrgHEU$ &    4.8 & 4.3  \\   
        \hline
    \end{tabular}
\end{table}
    
Table~\ref{tab:time} shows the running time of the personalized marketing scenario,
and Table~\ref{tab:eventtime} shows the running time of the segment marketing scenario,
Each running time
number is the average of five runs. 
The result shows that 
$\ProxGradRIS$ and $\UpperGradRIS$ are $9$ to $49$ times faster than $\ProxGradOrg$.
This is mainly due to the high variance in the stochastic gradient for the original
    objective function, as we discussed before.
Moreover, $\ProxGradRIS$ and $\UpperGradRIS$ is slower than $\GreedyRIS$.
This is mainly because our conservative bounds on the number of gradient iterations
    make $\ProxGradRIS$ and $\UpperGradRIS$ slow, while
    $\GreedyRIS$ only use the heuristic greedy approach with step size $0.1$
    without any theoretical guarantee.
Indeed $\GreedyRIS$ is inferior to $\ProxGradRIS$ and $\UpperGradRIS$ in terms
    of the influence spread achieved. 
    
The heuristic termination significantly improves the running time. Comparing against their
    respective theory-guided termination counterparts, we can see that heuristic termination
    in general improves the running time for 5 --- 7 times. 
Comparing against the $\GreedyRIS$ algorithm, we can see that 
    $\ProxGradRISHEU$ is now faster than $\GreedyRIS$ and
    $\UpperGradRISHEU$ is close to $\GreedyRIS$ in running time.
Therefore, this means that our gradient-based algorithms could achieve faster running time
    with heuristic termination while still providing better influence spread quality
    than the greedy heuristic, and if we want a theoretical guarantee, we could use
    more conservative theory-guided termination, which runs a few times slower but
    provides both theoretical guarantee and best empirical performance on influence spread.

\begin{figure*}[!ht]
	\centering
	\subfloat[$c(\bx)=||\bx||_1$,  $\lambda=10$]{\includegraphics[width=2.6in]{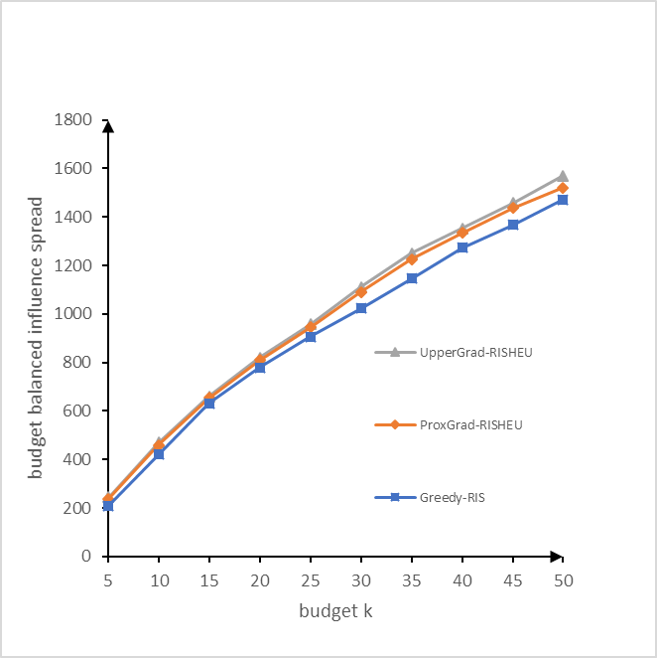}}\quad
	\subfloat[$c(\bx)=||\bx||_1$, $k=50$]{\includegraphics[width=2.6in]{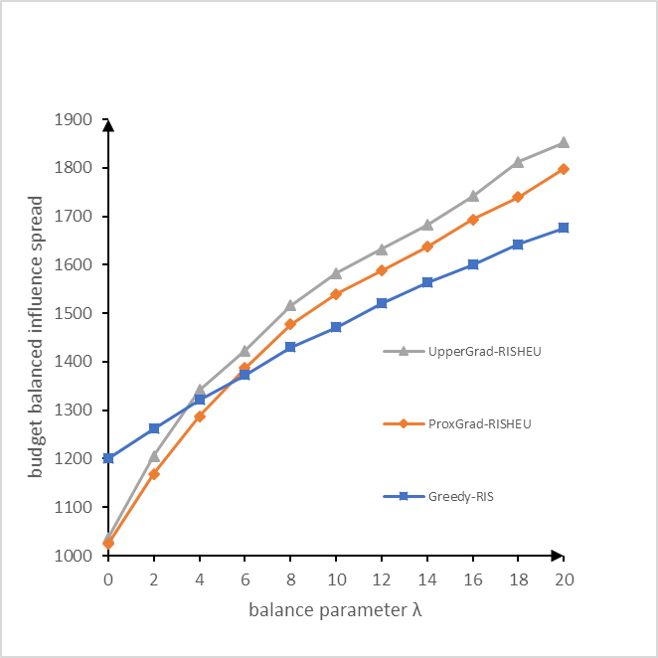}}
	\caption{Budget balanced influence spread results for the personalized marketing scenario on the NetHEPT dataset. 
	} \label{fig:HEPTspread}
\end{figure*}

Next, we test on the larger dataset NetHEPT. On this larger dataset, the gradient algorithms with theoretical guarantee is too slow to run, so we
	only run the heuristic versions of the algorithms and comparing them with the heuristic greedy algorithm. 
Figure~\ref{fig:HEPTspread} show the result of the budget balanced influence spread versus $k$ and $\lambda$ respectively,
	for the case of the personalized marketing scenario.
We use basically the same parameter settings as in the DM dataset, except that we try large $\lambda$ values (e.g. $\lambda=10$ instead of
	$\lambda=5$ as in the DM dataset when varying $k$), because NetHEPT dataset has larger influence spread, and we need a larger
	value of $\lambda$ to balance that.
From the result, we can see that the $\UpperGradRISHEU$ and $\ProxGradRISHEU$ still perform better than the greedy heuristic.
Moreover, the advantage is larger when the balance parameter $\lambda$ is getting large, similar to the results we see on the
	DM dataset.
Table~\ref{tab:HEPTtime} reports the running time of the algorithms on the NetHEPT dataset.
We see that $\UpperGradRISHEU$ and $\ProxGradRISHEU$ are a few times slower than the greedy heuristic.

\begin{table}[h]
	\centering
	\small
	\caption{Running time results for the personalized marketing scenario on the NetHEPT dataset (in seconds).}
	\label{tab:HEPTtime}
	\begin{tabular}{|l|r|}
		\hline
		&    $c(\bx) = ||\bx||_1$, $k=50$ and $\lambda=10$  \\
		\hline
		$\ProxGradRISHEU$ &    964.1    \\   
		\hline
		$\UpperGradRISHEU$ &    1522.7   \\   
		\hline
		$\GreedyRIS$ &    334.0   \\   
		\hline
	\end{tabular}
\end{table}

Besides looking at the budget balanced influence spread $g(\bx)+s(\bx)$ as a whole, we would also like to decompose this overall objective into the influence spread $g(\bx)$ and budget saving $s(\bx)$ and see how each of them behaves, especially when $\lambda$ changes.
Figure~\ref{fig:othertests} (a) shows this test result on the DM dataset with $k=50$ and $c(\bx)=||\bx||_1$, focusing on the $\UpperGradRIS$ and $\GreedyRIS$ algorithms.
The result shows that when $\lambda$ increases, the influence spread objective $g(\bx)$ in general decreases while the budget saving objective $s(\bx)$ increases, 
	indicating that both algorithms lean towards budget saving when more weight is put on budget saving.
Comparing the two algorithms, we clearly see that $\UpperGradRIS$ put much more emphasis on budget saving than $\GreedyRIS$, with budget saving objective $s(\bx)$ more 
	than doubled. 

For the heuristic version of our gradient algorithms, we further verify the stopping criteria parameter $0.3$ that we use. 
To do so, we vary this parameter from $0.1$ to $1$ and see the result comparing to the theory-guided version of the algorithms.
Figure~\ref{fig:othertests} (b) shows this test result on the DM dataset with $k=50$, $\lambda=5$ and $c(\bx)=||\bx||_1$.
The result shows that in general before $0.3$ or $0.4$, the performance of our heuristic algorithms match very closely with the theory-guided versions, but when 
	the parameter increases to $0.5$ or above, the performance of the heuristic algorithms starts to drop significantly.
Therefore, in our main experiments, we set this parameter to $0.3$.
   
   \begin{figure*}[!ht]
   	\centering
   	\subfloat[Decomposition of the objective function]{\includegraphics[width=2.6in]{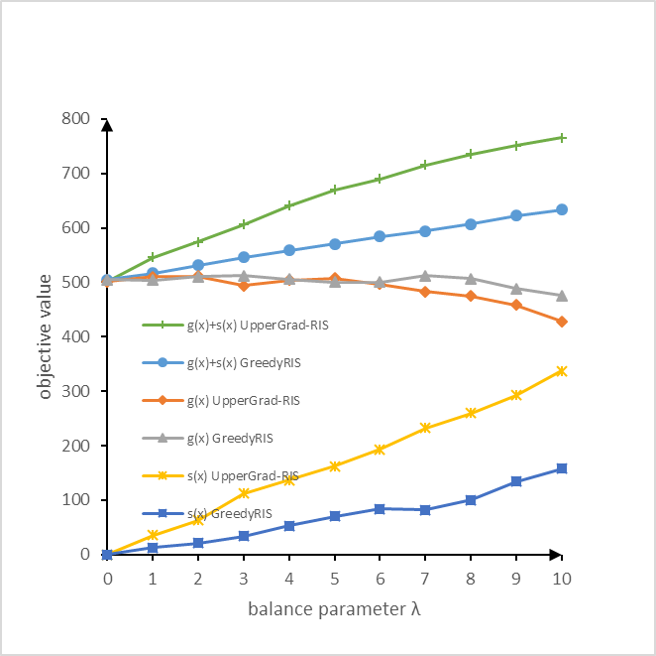}}\quad
   	\subfloat[Stopping criteria parameter tuning]{\includegraphics[width=2.6in]{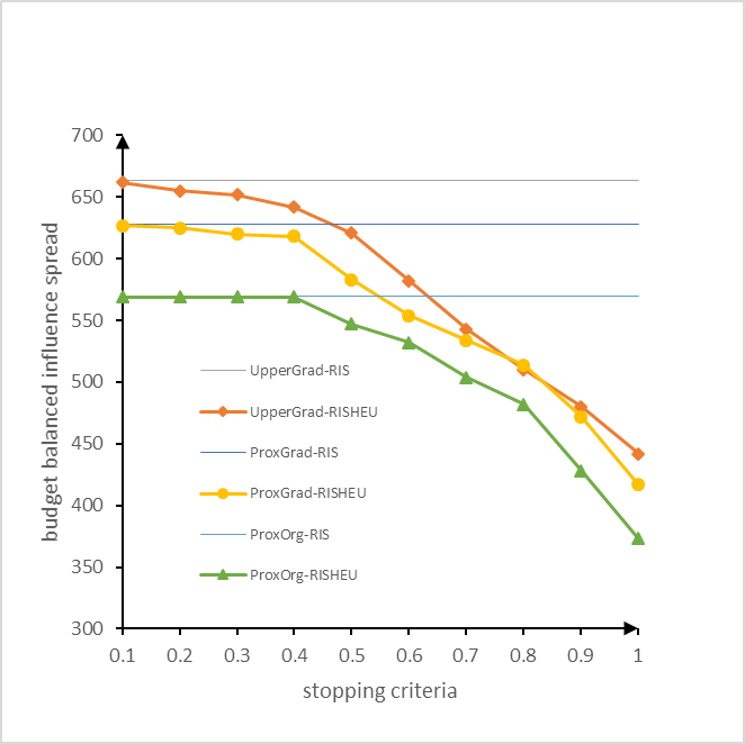}}
   	\caption{Decomposition of the objective function, and stopping criteria parameter tuning for the heuristic version of algorithms, on the DM dataset. 
   		We use $c(\bx)=||\bx||_1$, $k=50$, and for (b) $\lambda=5$.
   	} \label{fig:othertests}
   \end{figure*}

Finally, we collect the statistics for the first three moments of the average RR set sizes,
    which are closely related to the running time of the gradient-based algorithms, 
    as discussed at the end of Section~\ref{sec:gradientAlgo} and shown in
    Theorems~\ref{thm-proximal-time-heu} and~\ref{thm-concave-upperbound-time-heu}.
In particular, by random sampling 10,000 RR sets and taking the average, we obtain
    $\nu^{(1)} = \E[|R|] =7.2$, $\nu^{(2)}= \E[|R|^2]= 62.9$, $\nu^{(3)}= \E[|R|^3] = 501.4$.
Following the remark after Theorems~\ref{thm-proximal-time-heu}, we can see that
    without using these moments in the time complexity bound, we would have relaxed the
    bound for a large factor.
In particular, according to the remark after Theorems~\ref{thm-proximal-time-heu}, 
    the relaxation factors for various components of computations are:
    $n / \nu^{(1)} =94$, $n^2/\nu^{(2)} = 7,330$, $\nu^{(1)}\cdot n / \nu^{(2)} = 78$, and
    $\nu^{(1)} \cdot n^2 / \nu^{(3)} = 6,620$.
This suggests that using moments of mean RR set size would significantly reduce the
    time complexity bound.

}
\OnlyInShort{
\input{aaai_experiments_short.tex}
}

\section{Conclusion and Further Work}
In this paper, we tackle the new problem of continuous influence maximization with
	budget saving (CIM-BS), whose objective function is neither monotone, nor
	DR-submodular or concave.
We use the gradient method to solve CIM-BS, and provide innovative 
	integration with the reverse influence sampling method to achieve theoretical
	approximation guarantees. 
One important direction of future study is to make the gradient method more scalable,
	which requires more detailed study of convergence behavior and properties of the gradient
	method in the influence maximization domain.
Another direction is to investigate if the gradient method can be applied to other
	influence maximization settings such as competitive influence maximization.
Gradient method is a rich and powerful approach that has been already applied to many
	application domains, and thus we hope our work could inspire more studies 
	incorporating the gradient method into the influence maximization research.



\OnlyInShort{
\bibliographystyle{aaai}
}
\bibliography{bibdatabase}

\OnlyInFull{
\newpage

\appendix

\section*{Appendix}

\section{Proof of Theorem \ref{thm-combining-gradient-rrset}}\label{appendix-proof-thm-combining-gradient-rrset}
In this section, we present the detailed proof of Theorem \ref{thm-combining-gradient-rrset}.
The proof follows the structure of the proof of IMM~\cite{tang15}.

 Our analysis is based on a version of Chernoff bound, which is shown as follow. For convenience, we will use the notation $\text{OPT}_g$ to denote the maximum value of $g$ and $\text{OPT}_{g+s}$ to denote the maximum value of $g+s$ in the set $\cP$, and we use $\bx^*_g$ to denote the solution when maximizing the function $g$ in the set $\cP$, i.e. $g(\bx^*_g) = \text{OPT}_g$. We will also use $\bx^*_{g+s}$ to denote the point maximizing $g+s$ in the set $\cP$, i.e. $g(\bx^*_{g+s}) = \text{OPT}_{g+s}$.

\begin{proposition}[Chernoff Bound~\cite{tang15}]\label{prop-chernoff}
    Let $X_1,X_2,\dots,X_t$ be $t$ independent random variable with support $[0,1]$, and let $\E[X_i] = \mu$ for all $i\in [t]$. Let $Y = \sum_{i=1}^tX_i$, we have for any $\gamma > 0$,
    \[\Pr\{Y - t\mu \ge \gamma \cdot t\mu \} \le \exp\left(-\frac{\gamma^2}{2+\frac{2}{3}\gamma}t\mu\right).\]
    For any $0 < \gamma < 1$, we have
    \[\Pr\{Y - t\mu < -\gamma \cdot t\mu \}\le \exp\left(-\frac{\gamma^2}{2}t\mu\right).\]
\end{proposition}

Recall that we want to optimize the function $(g+s)(\bx)$ in the set $\cP$, where $h(\bx) \ge 0$ for all $\bx\in \cP$, and $g(\bx)$ is the influence of the network with strategy $\bx$ and $g(\bx)$, and we use
\[\hat g_{\cR}(\bx) = \frac{n}{\theta}\sum_{R\in\cR}\left(1-\prod_{v\in R}(1-h_v(\bx))\right)\]
to approximate $g(\bx)$.

Then given the Chernoff bound, our proof comes as follow. We first fix the number of generated independent RR-sets $\theta = |\cR|$ and we assume that we have an lower bound $\text{LB}$ for the optimal value $\text{OPT}_{g+s}$. We first show that with the randomness of the generated RR-sets, with high probability, optimizing the function $(\hat g_{\cR}+s)$ will lead to a guaranteed approximation of the function $(g+s)$. (Lemma \ref{lem-direct-chernoff},\ref{lem-approx-ratio},\ref{lem-approx-ratio-with-parameter}) Then we show that with high probability, the Sampling Procedure(Algorithm \ref{alg-sampling}) will return a lower bound $\text{LB}$ for the optimal value $\text{OPT}_{g+s}$.

\begin{lemma}\label{lem-direct-chernoff}
    Given a constant $0 < \alpha' < 1$. For any $\varepsilon > 0$, any $0<\alpha'\varepsilon_1<  \varepsilon/3$, and any $\delta_1,\delta_2>0$. If we have an lower bound  $\text{LB}$ for the optimal value $\text{OPT}_{g+s}$, let
    \[\theta^{(1)} = \frac{2n\cdot \ln\left(\frac{1}{\delta_1}\right)}{ \varepsilon_1^2\cdot \text{LB}}, \quad \theta^{(2)}(\cN) = \frac{2\alpha'\cdot n\cdot\ln\left(\frac{\cN}{\delta_2}\right)}{(\varepsilon/3 - \alpha' \varepsilon_1)^2\text{LB}},\]
    where $\cN$ is a variable.
    Recall that $\bx^*_{g+s} = \argmax_{\bx\in\cP}(g(\bx)+s(\bx))$, then for any fixed $|\cR| = \theta \ge \theta^{(1)}$, we have
    \[\Pr\{\hat g_{\cR}(\bx^*_{g+s}) + h(\bx^*_{g+s}) < (1-\varepsilon_1)\cdot  \text{OPT}_{g+s}\} \le \delta_1.\]
    For any fixed $|\cR| = \theta \ge \theta^{(2)}(\cN)$ and any fixed possible solution $\bx\in\cP$ that satisfies $g(\bx) + h(\bx) < (\alpha' - \varepsilon/3)\text{OPT}_{g+s} - T$ where $T\ge 0$ is a constant, we have
    \[\Pr\left\{\hat g_{\cR}(\bx)+s(\bx) \ge \alpha'(1-\varepsilon_1)\cdot\text{OPT}_{g+s}-T\right\} \le \frac{\delta_2}{\cN}.\]
\end{lemma}

\begin{proof}
    First recall that
    \[\hat g_{\cR}(\bx) = \frac{n}{\theta}\sum_{R\in\cR}\left(1-\prod_{v\in R}(1-h_v(\bx))\right).\]
    Let $X_i^{\cR}(\bx) = 1-\prod_{v\in R_i}(1-h_v(\bx))$ where $\cR = \{R_1,R_2,\dots,R_{\theta}\}$, then $X_i^{\cR}(\bx) \in [0,1]$ and $\hat g_{\cR}(\bx) = \frac{n}{\theta}\sum_{i=1}^{\theta}X_i^{\cR}(\bx)$. We also know that $X_i^{\cR}(\bx)$ are independent. Then from the Chernoff bound(Proposition \ref{prop-chernoff}), we have
    \begin{align*}
        &\Pr\{\hat g_{\cR}(\bx^*_{g+s}) + s(\bx^*_{g+s}) < (1-\varepsilon_1)\cdot  \text{OPT}_{g+s}\}\\
        =& \Pr\left\{\frac{n}{\theta}\sum_{i=1}^{\theta}X_i^{\cR}(\bx^*)+s(\bx^*_{g+s}) < (1-\varepsilon_1)\cdot  (g(\bx^*_{g+s}) + s(\bx^*_{g+s}))\right\} \\
        =& \Pr\left\{\sum_{i=1}^{\theta}X_i^{\cR}(\bx^*_{g+s}) - \frac{\theta}{n}g(\bx^*_{g+s}) < -\varepsilon_1\cdot  \frac{\theta}{n}\text{OPT}_{g+s}\right\} \\
        \le& \exp\left(-\frac{\left(\varepsilon_1\cdot  \frac{\text{OPT}_{g+s}}{ g(\bx^*_{g+s})}\right)^2}{2}\frac{\theta}{n}g(\bx^*_{g+s})\right) \\
        \le& \exp\left(-\frac{\varepsilon_1^2}{2}\frac{\theta}{n}\text{OPT}_{g+s}\right)\\
        \le& \exp\left(-\frac{\varepsilon_1^2}{2}\frac{2n\cdot \ln\left(\frac{1}{\delta_1}\right)}{n\text{LB}\cdot \varepsilon_1^2}\text{OPT}_{g+s}\right) \le \delta_1.
    \end{align*}
    Let $\varepsilon_2 = \varepsilon - \alpha' \varepsilon_1$. Let $\bx\in\cP$ be a point such that $g(\bx) + s(\bx) < (\alpha' - \varepsilon/3)\text{OPT}_{g+s} - T$, and we have
    \begin{align*}
        &\Pr\left\{\hat g_{\cR}(\bx)+s(\bx) \ge \alpha'(1-\varepsilon_1)\cdot\text{OPT}_{g+s}-T\right\} \\
        =& \Pr\left\{\frac{n}{\theta}\sum_{i=1}^{\theta}X_i^{\cR}(\bx) + s(\bx) - g(\bx) - s(\bx) \ge \alpha'(1-\varepsilon_1)\cdot\text{OPT}_{g+s} - g(\bx) - s(\bx)-T\right\} \\
        \le& \Pr\left\{\sum_{i=1}^{\theta}X_i^{\cR}(\bx) - \frac{\theta}{n}g(\bx) \ge \frac{\theta}{n}\left(\alpha'(1-\varepsilon_1)\cdot\text{OPT}_{g+s} - (\alpha' - \varepsilon/3)\text{OPT}_{g+s}\right)\right\} \\
        =& \Pr\left\{\sum_{i=1}^{\theta}X_i^{\cR}(\bx) - \frac{\theta}{n}g(\bx) \ge \frac{\theta}{n}\left(\varepsilon_2\text{OPT}_{g+s}\right)\right\}\\
        =& \Pr\left\{\sum_{i=1}^{\theta}X_i^{\cR}(\bx) - \frac{\theta}{n}g(\bx) \ge \left(\frac{\varepsilon_2\text{OPT}_{g+s}}{g(\bx)}\right)\frac{\theta}{n}g(\bx)\right\}\\
        \le& \exp\left(-\frac{\left(\frac{\varepsilon_2\text{OPT}_{g+s}}{g(\bx)}\right)^2}{2+\frac{2}{3}\frac{\varepsilon_2\text{OPT}_{g+s}}{g(\bx)}}\frac{\theta}{n}g(\bx)\right) \\
        \le& \exp\left(-\frac{\varepsilon_2^2\text{OPT}_{g+s}^2}{2g(\bx)+\frac{2}{3}\varepsilon_2\text{OPT}_{g+s}}\frac{\theta}{n}\right) \\
        \le& \exp\left(-\frac{\varepsilon_2^2\text{OPT}_{g+s}^2}{2(\alpha' - \varepsilon/3)\text{OPT}_{g+s}+\frac{2}{3}\varepsilon_2\text{OPT}_{g+s}}\frac{\theta}{n}\right) \\
        \le& \exp\left(-\frac{(\varepsilon/3 - \alpha' \varepsilon_1)^2\text{OPT}_{g+s}}{2\alpha'}\frac{\theta}{n}\right) \\
        \le& \exp\left(-\frac{(\varepsilon/3 - \alpha' \varepsilon_1)^2\text{OPT}_{g+s}}{2\alpha'}\frac{1}{n}\cdot \frac{2\alpha'\cdot n\cdot\ln\left(\frac{\cN}{\delta_2}\right)}{(\varepsilon/3 - \alpha' \varepsilon_1)^2\text{LB}}\right) \le \frac{\delta_2}{\cN}.
    \end{align*}
\end{proof}

\begin{lemma}\label{lem-approx-ratio}
    Given a constant $0 < \alpha < 1, L_1,L_2$. For any $\varepsilon > 0$, any $\varepsilon_2 > 0, 0 < \varepsilon_3 < \alpha$, any $0<\varepsilon_1< (\alpha-\varepsilon_3) \varepsilon/3$, and any $\delta_1,\delta_2>0$. Suppose that we have an lower bound  $\text{LB}$ for the optimal value $\text{OPT}_{g+s}$, if 
    
    (a) $\Pr\{\hat g_{\cR}(\bx^*_{g+s}) + s(\bx^*_{g+s}) < (1-\varepsilon_1)\cdot  \text{OPT}_{g+s}\} \le \delta_1$; 
    
    (b) for any $\bx\in\cP$ that satisfies $g(\bx) + s(\bx) < (\alpha - \varepsilon_3 - \varepsilon/3)\text{OPT}_{g+s}-L_2\varepsilon_2 \text{LB}$, we have 
    \[\Pr\left\{\hat g_{\cR}(\bx)+s(\bx) \ge ((\alpha-\varepsilon_3)(1-\varepsilon_1)\cdot\text{OPT}_{g+s} - L_2\varepsilon_2\text{LB}\right\} \le \frac{\delta_2}{\cN(\cP, \varepsilon_2\text{LB})};\]
    
    (c) the algorithm $\cA$ will output $\bx_{out}$ such that $(\hat g_{\cR}+s)(\bx_{out}) \ge (\alpha-\varepsilon_3)\cdot \max_{\bx\in\cP}(\hat g_{\cR}+s)(\bx)$;
    
    (d) $(g+s)(\bx)$ is $L_1$-Lipschitz, and $(\hat g_{\cR}+s)(\bx)$ is $L_2$-Lipschitz;
    
    then with probability at least $1-\delta_1-\delta_2$,
    \[(g+s)(\bx_{out}) \ge (\alpha - \varepsilon/3 -\varepsilon_3 - (L_1+L_2)\varepsilon_2)\text{OPT}_{g+s}.\]
\end{lemma}

\begin{proof}
    First we fix an $\varepsilon_2\text{LB}$-net $E$ for the set $\cP$ with number of points $\cN(\cP, \varepsilon_2\text{LB})$. Let $\pi(\bx):\cP\to E$ denote the mapping from $\cP$ to the $\varepsilon_2\text{LB}$-net $E$ such that $||\pi(\bx) - \bx||_2 \le \varepsilon_2\text{LB}$.
    From assumption (a), we know that with probability at least $1-\delta_1$, we have
    \begin{align*}
        (\hat g_{\cR}+s)(\bx_{out}) \ge& (\alpha - \varepsilon_3)\max_{\bx\in\cP}(\hat g_{\cR}+s)(\bx)\\
        \ge& (\alpha - \varepsilon_3)(\hat g_{\cR}+s)(\bx^*_{g+s})\\
        \ge& (\alpha - \varepsilon_3)(1-\varepsilon_1)\cdot  \text{OPT}_{g+s}.
    \end{align*}
    Since $(\hat g_{\cR}+s)(\bx)$ is $L_2$-Lipschitz, we have
    \begin{align*}
        (\hat g_{\cR}+s)(\pi(\bx_{out})) =& (\hat g_{\cR}+s)(\bx_{out}) + ((\hat g_{\cR}+s)(\pi(\bx_{out})) - (\hat g_{\cR}+s)(\bx_{out})) \\
        \ge& (\hat g_{\cR}+s)(\bx_{out}) - |(\hat g_{\cR}+s)(\pi(\bx_{out})) - (\hat g_{\cR}+s)(\bx_{out})| \\
        \ge& (\alpha - \varepsilon_3)(1-\varepsilon_1)\cdot  \text{OPT}_{g+s} - L_2\varepsilon_2\text{LB}.
    \end{align*}
    Then from (b) and the union bound, we know that with probability at least $1-\delta_2$, for every $\bx\in E$, if $g(\bx) + s(\bx) < (\alpha - \varepsilon_3 - \varepsilon/3)\text{OPT}_{g+s}-L_2\varepsilon_2\text{LB}$, then
    \[\hat g_{\cR}(\bx)+s(\bx) < (\alpha-\varepsilon_3)(1-\varepsilon_1)\cdot\text{OPT}_{g+s} - L_2\varepsilon_2\text{LB}.\]
    Then by the union bound, we know that with probability at least $1-\delta_1-\delta_2$,
    \[(g+s)(\pi(\bx_{out})) \ge (\alpha - \varepsilon_3 - \varepsilon/3)\text{OPT}_{g+s}-L_2\varepsilon_2\text{LB}.\]
    Since from (d), $(g+s)(\bx)$ is $L_1$-Lipschitz, then with probability at least $1-\delta_1-\delta_2$, we have
    \[(g+s)(\bx_{out}) \ge (g+s)(\pi(\bx_{out})) - L_1\varepsilon_2\text{LB} \ge (\alpha - \varepsilon_3 - \varepsilon/3-(L_1+L_2)\varepsilon_2)\text{OPT}_{g+s}.\]
\end{proof}

Combining Lemma \ref{lem-direct-chernoff} and Lemma \ref{lem-approx-ratio} together, we have the following lemma,
\begin{lemma}\label{lem-approx-ratio-with-parameter}
    If $(g+s)(\bx)$ is $L_1$-Lipschitz, and $(\hat g_{\cR}+s)(\bx)$ is $L_2$-Lipschitz. Suppose that we have an lower bound $\text{LB}$ for the optimal value $\text{OPT}_{g+s}$, and an oracle $\cA$ that can get an $(\alpha-\varepsilon/3)$-approximation for $\text{OPT}_{g+s}$, where $\alpha-\varepsilon/3 > 0$. Let
    \[\theta^{(1)} = \frac{8n\cdot \ln\left(4n^{\ell}\right)}{\text{LB}\cdot (\alpha-\varepsilon/3)^2\varepsilon^2/9}, \quad \theta^{(2)} = \frac{2\alpha'\cdot n\cdot\ln\left(4n^{\ell}\cN(\cP,\frac{\varepsilon/3}{L_1+L_2}\text{LB})\right)}{(\varepsilon/3 - \frac{1}{4}(\alpha-\varepsilon/3)^2 \varepsilon/3)^2\text{LB}}.\]
    If $|\cR| = \theta \ge \max\{\theta^{(1)},\theta^{(2)}\}$, and $\bx_{out} = \cA(\hat g_{\cR} + h, \varepsilon \text{LB})$, then with probability at least $1-\frac{1}{2n^{\ell}}$,
    \[(g+s)(\bx_{out}) \ge (\alpha - \varepsilon)\text{OPT}_{g+s}.\]
\end{lemma}

\begin{proof}
    The proof of this lemma is a direct combination of Lemma \ref{lem-direct-chernoff} and Lemma \ref{lem-approx-ratio}. We choose the parameters $\delta_1 = \delta_2 = \frac{1}{4n^{\ell}}$ in Lemma \ref{lem-direct-chernoff} and Lemma \ref{lem-approx-ratio}. We choose $\varepsilon_1 = \frac{1}{2}(\alpha-\varepsilon/3)$, $\varepsilon_2 = \frac{\varepsilon/3}{L_1+L_2}$, $\varepsilon_3 = \varepsilon/3$ in Lemma \ref{lem-approx-ratio}, and $\alpha' = \alpha-\varepsilon/3$ in Lemma \ref{lem-direct-chernoff}.
    
    Now since $|\cR| = \theta \ge \max\{\theta^{(1)},\theta^{(2)}\}$, then the assumption of Lemma \ref{lem-direct-chernoff} is satisfied, and then the assumption (a),(b) of Lemma \ref{lem-approx-ratio} is satisfied.
    
    Then based on our assumption on the oracle $\cA$, we know that
    \[(\hat g_{\cR}+s)(\bx_{out}) \ge \alpha \max_{\bx\in\cP}(\hat g_{\cR}+s)(\bx) - \varepsilon/3\cdot\text{LB} \ge (\alpha-\varepsilon/3)\max_{\bx\in\cP}(\hat g_{\cR}+s)(\bx),\]
    then the assumption (c) of Lemma \ref{lem-approx-ratio} is satisfied.
    
    We also assume that $(g+s)(\bx)$ is $L_1$-Lipschitz, and $(\hat g_{\cR}+s)(\bx)$ is $L_2$-Lipschitz, so assumption (d) is also satisfied. Then we know that with probability at least $1-\frac{1}{2n^{\ell}}$,
    \[(g+s)(\bx_{out}) \ge (\alpha - \varepsilon)\text{OPT}_{g+s}.\]
\end{proof}

Now we show that with high probability, the sampling procedure will return a lower bound $\text{LB} \le \text{OPT}_{g+s}$.

\begin{lemma}\label{lem-high-probability-lb}
    For every $i=1,2,\dots,\lfloor\log_2(n+\lambda k)-1\rfloor$, suppose that $g_{\cR_i} + h$ is $L_2$-Lipschitz where $|\cR_i| = \theta_i = \left\lceil\frac{n\cdot\left(2+\frac{2}{3}\varepsilon'\right)\cdot\left(\ln\cN(\cP,\frac{\varepsilon/3}{L_2} x_i)+\ell \ln n+\ln 2 + \ln\log_2{(n+\lambda k)}\right)}{\varepsilon'^2 x_i}\right\rceil$,
    
    (a) if $x_i = \frac{n+k}{2^i} > \text{OPT}_{g+s}$, then with probability at most $\frac{1}{2n^{\ell}\log_2(n+\lambda k)}$, 
    \[(\hat g_{\cR_i}+s)(\by_i) \ge (1+\varepsilon'+\varepsilon/3) x_i;\]
    
    (b) if $x_i = \frac{n+k}{2^i} \le \text{OPT}_{g+s}$, then with probability at most $\frac{1}{2n^{\ell}\log_2(n+\lambda k)}$, 
    \[(\hat g_{\cR_i}+s)(\by_i) \ge (1+\varepsilon'+\varepsilon/3) \text{OPT}_{g+s}.\]
\end{lemma}

\begin{proof}
    Let $E_i$ be the $\frac{\varepsilon/3}{L_2} x_i$-net such that $|E_i| = \cN(\cP, \frac{\varepsilon/3}{L_2} x_i)$ and let $\pi_i(\bx):\cP\to E_i$ denote the mapping such that $||\pi_i(\bx) - \bx||_2 \le \frac{\varepsilon/3}{L_2} x_i$. Since $g_{\cR_i} + h$ is $L_2$-Lipschitz, we only have to prove that
    
    (a') if $x_i = \frac{n+k}{2^i} > \text{OPT}_{g+s}$, then with probability at most $\frac{1}{2n^{\ell}\log_2(n+\lambda k)}$, 
    \[(\hat g_{\cR_i}+s)(\pi_i(\by_i)) \ge (1+\varepsilon') x_i;\]
    
    (b') if $x_i = \frac{n+k}{2^i} \le \text{OPT}_{g+s}$, then with probability at most $\frac{1}{2n^{\ell}\log_2(n+\lambda k)}$, 
    \[(\hat g_{\cR_i}+s)(\pi_i(\by_i)) \ge (1+\varepsilon') \text{OPT}_{g+s}.\]
    
    Let $X_j^{\cR_i}(\bx) = 1-\prod_{v\in R_j}(1-h_v(\bx))$ where $\cR_i = \{R_1,R_2,\dots,R_{\theta_i}\}$, then $X_j^{\cR_i}(\bx) \in [0,1]$ and $\hat g_{\cR_i}(\bx) = \frac{n}{\theta}\sum_{j=1}^{\theta}X_h^{\cR_i}(\bx)$.
    
    We first prove (a'). We first fix a $\by\in\cP$, then if $x_i = \frac{n+k}{2^i} > \text{OPT}_{g+s}$, we have
    \begin{align*}
        &\Pr\left\{(\hat g_{\cR_i}+s)(\by) \ge (1+\varepsilon') x_i\right\}\\
        =& \Pr\left\{\frac{n}{\theta}\sum_{j=1}^{\theta_i}X_j^{\cR_i}(\by) + s(\by) \ge (1+\varepsilon') x_i\right\} \\
        =& \Pr\left\{\frac{n}{\theta}\sum_{j=1}^{\theta_i}X_j^{\cR_i}(\by)-g(\by)\ge (1+\varepsilon') x_i-s(\by)-g(\by)\right\}\\
        \le& \Pr\left\{\sum_{j=1}^{\theta_i}X_j^{\cR_i}(\by)-\frac{\theta_i}{n}g(\by)\ge \frac{\theta_i}{n}\varepsilon' x_i\right\}\\
        \le& \exp\left(-\frac{\left(\frac{\varepsilon' x_i}{g(\by)}\right)^2}{2+\frac{2}{3}\frac{\varepsilon' x_i}{g(\by)}}\frac{\theta_i}{n}g(\by)\right) \\
        =& \exp\left(-\frac{\left(\varepsilon' x_i\right)^2}{2g(\by)+\frac{2}{3}\varepsilon' x_i}\frac{\theta_i}{n}\right) \\
        \le& \exp\left(-\frac{\varepsilon'^2 x_i}{2+\frac{2}{3}\varepsilon'}\frac{\theta_i}{n}\right) \\
        \le& \exp\left(-\frac{\varepsilon'^2 x_i}{2+\frac{2}{3}\varepsilon'}\frac{1}{n}\frac{n\cdot\left(2+\frac{2}{3}\varepsilon'\right)\cdot\left(\ln\cN(\cP,\frac{\varepsilon/3}{L_2} x_i)+\ell \ln n+\ln 2 + \ln\log_2{(n+\lambda k)}\right)}{\varepsilon'^2 x_i}\right) \\
        =& \frac{1}{2n^{\ell}\cdot \cN(\cP,\frac{\varepsilon/3}{L_2} x_i)\cdot \log_2{(n+\lambda k)}}.
    \end{align*}
    Then note that there are at most $\cN(\cP,\frac{\varepsilon}{L_2} x_i)$ possibilities for $\pi_i(\by_i)$, so applying the union bound, we have
    \[\Pr\left\{(\hat g_{\cR_i}+s)(\pi_i(\by_i)) \ge (1+\varepsilon') x_i\right\}\le \frac{1}{2n^{\ell}\log_2(n+\lambda k)}.\]
    Then we prove (b'). If $x_i = \frac{n+k}{2^i} \le \text{OPT}_{g+s}$, then for any fixed $y\in\cP$, we have
    \begin{align*}
        &\Pr\left\{(\hat g_{\cR_i}+s)(\by) \ge (1+\varepsilon') \text{OPT}_{g+s}\right\}\\
        =& \Pr\left\{\frac{n}{\theta}\sum_{j=1}^{\theta_i}X_j^{\cR_i}(\by) + s(\by) \ge (1+\varepsilon') \text{OPT}_{g+s}\right\} \\
        =& \Pr\left\{\frac{n}{\theta}\sum_{j=1}^{\theta_i}X_j^{\cR_i}(\by)-g(\by)\ge (1+\varepsilon') \text{OPT}_{g+s}-s(\by)-g(\by)\right\}\\
        \le& \Pr\left\{\sum_{j=1}^{\theta_i}X_j^{\cR_i}(\by)-\frac{\theta_i}{n}g(\by)\ge \frac{\theta_i}{n}\varepsilon' \text{OPT}_{g+s}\right\}\\
        \le& \exp\left(-\frac{\left(\frac{\varepsilon' \text{OPT}_{g+s}}{g(\by)}\right)^2}{2+\frac{2}{3}\frac{\varepsilon' \text{OPT}_{g+s}}{g(\by)}}\frac{\theta_i}{n}g(\by)\right) \\
        =& \exp\left(-\frac{\left(\varepsilon' \text{OPT}_{g+s}\right)^2}{2g(\by)+\frac{2}{3}\varepsilon' \text{OPT}_{g+s}}\frac{\theta_i}{n}\right) \\
        \le& \exp\left(-\frac{\varepsilon'^2 \text{OPT}_{g+s}}{2+\frac{2}{3}\varepsilon'}\frac{\theta_i}{n}\right) \\
        \le& \exp\left(-\frac{\varepsilon'^2 \text{OPT}_{g+s}}{2+\frac{2}{3}\varepsilon'}\frac{1}{n}\frac{n\cdot\left(2+\frac{2}{3}\varepsilon'\right)\cdot\left(\ln\cN(\cP,\frac{\varepsilon/3}{L_2} x_i)+\ell \ln n+\ln 2 + \ln\log_2{(n+\lambda k)}\right)}{\varepsilon'^2 x_i}\right) \\
        \le& \frac{1}{2n^{\ell}\cdot \cN(\cP,\frac{\varepsilon/3}{L_2} x_i)\cdot \log_2{(n+\lambda k)}}.
    \end{align*}
    Then similar to the proof of (a'), applying the union bound will conclude the proof.
\end{proof}

Then with the help of Lemma \ref{lem-high-probability-lb}, we can prove Theorem \ref{thm-combining-gradient-rrset}.

{\thmcombining*}

\begin{proof}[Proof of Theorem \ref{thm-combining-gradient-rrset}]
    First we show that with probability at least $1-\frac{1}{2n^{\ell}}$, the output lower bound $\text{LB} \le \text{OPT}_{g+s}$. We first prove the case when $\text{OPT}_{g+s} \ge x_{\lfloor\log_2(n+\lambda k)\rfloor-1}$. Let $k$ denote the smallest index such that $\text{OPT}_{g+s} \ge x_k$. Then for any $i\le k-1$, we have $\text{OPT}_{g+s} < x_i$ and for any $j\ge k$, we have $\text{OPT}_{g+s} \ge x_j$. Then from Lemma \ref{lem-high-probability-lb} and union bound, we know that with probability at least $\frac{1}{2n^{\ell}}$, for every $i\le k-1$, $(\hat g_{\cR_i}+s)(\pi_i(\by_i)) \ge (1+\varepsilon') x_i$, and for every $j\ge k$, we have $(\hat g_{\cR_j}+s)(\by_j) \ge (1+\varepsilon'+\varepsilon/3) \text{OPT}_{g+s}$. Then from the definition of the algorithm, we know that $\text{LB}\le\text{OPT}_{g+s}$.
    
    Then for the case when $\text{OPT}_{g+s} < x_{\lfloor\log_2(n+\lambda k)\rfloor-1}$, from the union bound, we know that with probability at least $1-\frac{1}{2n^{\ell}}$ the `break' statement will not be executed. So $\text{LB} = 1 \le \text{OPT}_{g+s}$.
    
    Then we bound the probability that the algorithm does not return an $\alpha-\varepsilon$ approximation. Let $\cA$ denote the event that the algorithm does not return an $\alpha-\varepsilon$ approximation, and $\cB$ denote the event that the output of the Sampling procedure $\text{LB} > \text{OPT}_{g+s}$. We want to show that $\Pr\{\cA\}\le\frac{1}{n^{\ell}}$. We have
    \begin{align*}
        \Pr\{\cA\} =& \Pr\{\cA\land\cB\} +\Pr\{\cA\land\lnot\cB\} \\
        \le& \Pr\{\cB\} + \Pr\{\cA|\lnot\cB\}.
    \end{align*}
    From Lemma \ref{lem-high-probability-lb}, we have $\Pr\{\cB\} \le \frac{1}{2n^{\ell}}$. Since we generate new RR-sets before using the oracle to get the solution, so $\text{LB}$ can be viewed as fixed, and from Lemma \ref{lem-approx-ratio-with-parameter}, we know that $\Pr\{\cA|\lnot\cB\}\le\frac{1}{2n^{\ell}}$. Combined them together, we complete the proof.
\end{proof}

\section{Omitted Proofs in Subsection \ref{sec-gradient-method-algorithm}}
\subsection{Proof of Theorem \ref{thm-proximal-grad}}
In this section, we give the formal proof of Theorem \ref{thm-proximal-grad}. First we slightly review the iteration step and the notations. For the proximal gradient descent, the iteration is shown as follows:
\begin{equation}\label{equ-proximal-gradient2}
\left\{
\begin{array}{l}
    \bx^{(t+1)} = \text{prox}_{-\eta_t f_2}(\bx^{(t)} + \eta_t \bv^{(t)}), 
    \text{where } \E[\bv^{(t)}] = \nabla f_1(\bx^{(t)}),\\
    \text{prox}_\phi(\bx) := \argmin_{\by\in\cP} (\phi(\by) + \frac{1}{2}||\bx-\by||^2_2),
    \text{for any convex function } \phi,
\end{array}
\right.
\end{equation}
where $\eta_t \le \frac{1}{\beta}$ is the step size and $\bv^{(t)}$ is the stochastic gradient at $\bx^{(t)}$ .
Note that without loss of generality, we can assume that $f_2(\bx)=-\infty$ for all $\bx\notin\cP$, and we have
$\argmin_{\by\in\cP} (\phi(\by) + \frac{1}{2}||\bx-\by||^2_2) = 
\argmin_{\by} (\phi(\by) + \frac{1}{2}||\bx-\by||^2_2)$.
Let $G_{\eta}(\bx) = \frac{1}{\eta}(\bx - \text{prox}_{-\eta f_2}(\bx + \eta \bv))$ where $\bv$ is the stochastic gradient of $f_1$ at point $\bx$, then we have
$\bx^{(t+1)} = \bx^{(t)} - \eta_t G_{\eta_t}(\bx^{(t)})$.

To analyze the convergence of the proximal gradient descent, we have the following proposition from \cite{chekuri2014submodular}.
\begin{proposition}\label{prop-submodular-gradient}
    If the function $f(\bx)$ is DR-submodular and monotone, we have
    \[\langle\nabla f(\bx), \bx-\by\rangle \le 2f(\bx) - f(\bx\land \by) -f(\bx\lor \by).\]
\end{proposition}

\begin{lemma}\label{lem-subgradient}
    Let $\bu = G_{\eta}(\bx) + \bv$, then $\bu\in -\partial f_2(\bx-\eta G_{\eta}(\bx))$, i.e. $\bu$ is the subgradient of $-f_2$ at point $\bx-\eta G_{\eta}(\bx)$. Here $\bv$ is the stochastic gradient of $f_1$ at point $\bx$, and $G_{\eta}(\bx) = \frac{1}{\eta}(\bx - \text{prox}_{-\eta f_2}(\bx + \eta \bv))$.
\end{lemma}

\begin{proof}
    First it is easy to show that if $g$ is convex and $\bu = \text{prox}_{g}(\bx)$, then we have $\bx-\bu\in\partial g(\bu)$. This is due to the fact that $\bu$ minimize the function $g_1(\by) = g(\by) + \frac{1}{2}||\by-\bx||_2^2$, and $g_1(\by)$ is convex in $\by$. Then we have
    \[0\in\partial g(\bu) + \bu-\bx \Rightarrow \bx-\bu \in \partial g(\bu).\]
    Then note that $\bx-\eta G_{\eta}(\bx) = \text{prox}_{-\eta f_2}(\bx+\eta \bv)$, then we have
    \[\bx+\eta \bv - \text{prox}_{-\eta f_2}(\bx+\eta\nabla f(\bx)) \in -\eta\partial f_2(\bx-\eta G_{\eta}(\bx)),\]
    and rearranging the terms we have
    \[\eta(\bv + G_{\eta}(\bx)) \in -\eta\partial f_2(\bx-\eta G_{\eta}(\bx)),\]
    which concludes the proof.
\end{proof}

\begin{lemma}
    Suppose $f_1(\bx)$ is a monotone DR-submodular function on convex set $\cP$ and $f_2(\bx)$ is concave on set $\cP$. Note that we assume $f_2(\bx) = -\infty$ for all $x\notin \cP$. If $f_1(\bx)$ is $\beta$-smooth, and $\eta \le \frac{1}{\beta}$ is the step size, then for any $\bx,\bz\in\cP$, we have
    \begin{align*}(f_1+f_2)(\bx-\eta G_{\eta}(\bx)) + (f_1+f_2)(\bx) \ge& (f_1+f_2)(\bz) + \frac{\eta}{2}||G_{\eta}(\bx)||_2^2 - G_{\eta}(\bx)^T(\bx-\bz)\\
    &\quad +f_2(\bx) - (\bv - \nabla f_1(\bx))^T(\bx-\eta G_{\eta}(\bx)-\bz).
    \end{align*}
\end{lemma}

\begin{proof}
    First note that $x-\eta G_{\eta}(\bx)\in\cP$, since $x-\eta G_{\eta}(\bx) = \text{prox}_{-\eta f_2}(\bx^{(k)} + \eta \bv^{(k)})$ where $\bv^{(k)}$ is the stochastic gradient of $f_1$ at point $\bx^{(k)}$, and we know that $-\eta f_2(\bx) = +\infty$ for all $\bx\notin\cP$.
    
    From the smoothness of function $f_1$ and the convexity of $-f_2$ and the previous lemma(Lemma \ref{lem-subgradient}), we have
    \begin{align*}
        &-(f_1+f_2)(\bx-\eta G_{\eta}(\bx))\\
        =& -f_1(\bx-\eta G_{\eta}(\bx)) -f_2(\bx-\eta G_{\eta}(\bx)) \\
        \le& -f_1(\bx) + \langle -\nabla f_1(\bx), -\eta G_{\eta}(\bx)\rangle + \frac{\beta}{2}||\eta G_{\eta}(\bx)||_2^2 -f_2(\bx-\eta G_{\eta}(\bx)) \\
        \le& -f_1(\bx) + \eta \nabla f_1(\bx)^T G_{\eta}(\bx) + \frac{\beta}{2}||\eta G_{\eta}(\bx)||_2^2\\
        &\quad - f_2(\bz) + (\bv + G_{\eta}(\bx))^T(\bx-\eta G_{\eta}(\bx)-\bz) \\
        =& -f_1(\bx) + \eta \nabla f_1(\bx)^T G_{\eta}(\bx) + \frac{\beta}{2}||\eta G_{\eta}(\bx)||_2^2\\
        &\quad - f_2(\bz) + (\nabla f_1(\bx) + G_{\eta}(\bx))^T(\bx-\eta G_{\eta}(\bx)-\bz) + (\bv - \nabla f_1(\bx))^T(\bx-\eta G_{\eta}(\bx)-\bz) \\
        =& -f_1(\bx) + \nabla f_1(\bx)^T(\bx-\bz) - \frac{\beta}{2}||\eta G_{\eta}(\bx)||_2^2 -f_2(\bz) + G_{\eta}(\bx)^T(\bx-\bz)\\
        &\quad + (\bv - \nabla f_1(\bx))^T(\bx-\eta G_{\eta}(\bx)-\bz).
    \end{align*}
    Then from the proposition(Proposition \ref{prop-submodular-gradient})
    \[\langle\nabla f_1(\bx), \bx-\bz\rangle \le 2f_1(\bx) - f_1(\bx\land \bz) -f_1(\bx\lor \bz) \le 2f_1(\bx) - f_1(\bz),\]
    we have
    \begin{align*}
        -(f_1+f_2)(\bx-\eta G_{\eta}(\bx)) \le& -f_1(\bx) + \nabla f_1(\bx)^T(\bx-\bz) - \frac{\beta}{2}||\eta G_{\eta}(\bx)||_2^2 -f_2(\bz)  \\
        &\quad + G_{\eta}(\bx)^T(\bx-\bz)+ (\bv - \nabla f_1(\bx))^T(\bx-\eta G_{\eta}(\bx)-\bz)\\
        \le& -f_1(\bx) + 2f_1(\bx) - f_1(\bz) - \frac{\eta}{2}||G_{\eta}(\bx)||_2^2 -f_2(\bz) \\
        &\quad + G_{\eta}(\bx)^T(\bx-\bz)+ (\bv - \nabla f_1(\bx))^T(\bx-\eta G_{\eta}(\bx)-\bz)\\
        =& f_1(\bx) - (f_1+f_2)(\bz) + G_{\eta}(\bx)^T(\bx-\bz)- \frac{\eta}{2}||G_{\eta}(\bx)||_2^2\\
        &\quad + (\bv - \nabla f_1(\bx))^T(\bx-\eta G_{\eta}(\bx)-\bz).\\
    \end{align*}
    Rearranging the terms, we have
    \begin{align*}(f_1+f_2)(\bx-\eta G_{\eta}(\bx)) + (f_1+f_2)(\bx) \ge& (f_1+f_2)(\bz) + \frac{\eta}{2}||G_{\eta}(\bx)||_2^2 - G_{\eta}(\bx)^T(\bx-\bz)\\
    &\quad +f_2(\bx) - (\bv - \nabla f_1(\bx))^T(\bx-\eta G_{\eta}(\bx)-\bz).
    \end{align*}
\end{proof}

{\thmproximalconverge*}

\begin{proof}[Proof of Theorem \ref{thm-proximal-grad}]
    Note that we assume the function $f_2$ to be the concave extension of the original function $d$, i.e. $f_2(\bx) = -\infty$ for all $x\notin\cP$. Then we know that
    \[\bx^{(t+1)} = \text{prox}_{-\eta f_2}(\bx^{(t)} + \eta_t\bv^{(t)}) \in \cP,\]
    which means that $f_2(\bx^{(t+1)}) \ge 0$, for all $t=1,2,\dots,T$. Then in the previous lemma, let $\bz = \bx^*$ where $\bx^*$ maximize $f_1+f_2$ in the set $\cP$. 
    We first consider the case when we have exact gradient $\bv^{(t)} = \nabla f_1(\bx^{(t)})$ and we set $\eta_t = \eta = \frac{1}{\beta}$, and we have
    \begin{align*}
        &(f_1+f_2)(\bx^{(t+1)}) + (f_1+f_2)(\bx^{(t)})\\
        \ge& (f_1+f_2)(\bx^*) + \frac{\eta}{2}||G_{\eta}(\bx^{(t)})||_2^2 - G_{\eta}(\bx^{(t)})^T(\bx^{(t)}-\bx^*) +f_2(\bx^{(t)}) \\
        \ge& (f_1+f_2)(\bx^*) + \frac{\eta}{2}||G_{\eta}(\bx^{(t)})||_2^2 - G_{\eta}(\bx^{(t)})^T(\bx^{(t)}-\bx^*) \\
        =& (f_1+f_2)(\bx^*) + \frac{1}{2\eta}(\eta G_{\eta}(\bx^{(t)}))^T(\eta G_{\eta}(\eta \bx^{(t)}) - 2\bx^{(t)}+2\bx^*) \\
        =& (f_1+f_2)(\bx^{*}) + \frac{1}{2\eta}\left(||\eta G_{\eta}(\bx^{(t)}) - \bx^{(t)} + \bx^*||_2^2 - ||\bx^{(t)} - \bx^*||_2^2\right) \\
        =& (f_1+f_2)(\bx^{*}) + \frac{1}{2\eta}\left(||\bx^{(t+1)} - \bx^*||_2^2 - ||\bx^{(t)} - \bx^*||_2^2\right).
    \end{align*}
    Then we sum up the above inequalities and we get
    \begin{align*}
        &\sum_{t=0}^{T-1}((f_1+f_2)(\bx^{(t+1)}) + (f_1+f_2)(\bx^{(t)}))\\
        \ge& \sum_{t=0}^{T-1}\left((f_1+f_2)(\bx^{*}) + \frac{1}{2\eta}\left(||\bx^{(t+1)} - \bx^*||_2^2 - ||\bx^{(t)} - \bx^*||_2^2\right)\right) \\
        =& T(f_1+f_2)(\bx^*) + \frac{1}{2\eta}||\bx^{(T)} - \bx^*||_2^2 - \frac{1}{2\eta}||\bx^{(0)} - \bx^*||_2^2 \\
        \ge& T(f_1+f_2)(\bx^*) - \frac{1}{2\eta}||\bx^{(0)} - \bx^*||_2^2\\
        \ge& T(f_1+f_2)(\bx^*) - \frac{1}{2\eta}\Delta^2.
    \end{align*}
    Then, we complete the proof by the fact that
    \[\max_{t=0,1,2,\dots,T} h(x^{(t)}) \ge \frac{1}{2T}\sum_{t=0}^{T-1}(h(x^{(t+1)}) + h(x^{(t)})),\]
    and we have
    \[\max_{t=0,1,2,\dots,T} h(x^{(t)}) \ge \frac{1}{2}(f_1+f_2)(\bx^*) - \frac{\beta \Delta^2}{4T}.\]
    Then we prove the stochastic gradient case. We first have the following property \cite{combettes2005signal}: For convex function $\phi$ and any $\bx,\by$, we have
    \[||\text{prox}_{\phi}(\bx) - \text{prox}_{\phi}(\by)||_2 \le ||\bx-\by||_2.\]
    From the previous lemma and the previous property, we take the expectation of $\bv^{(t)}$ and we can get
    \begin{align*}
        &\E_{\bv^{(t)}}(f_1+f_2)(\bx^{(t+1)}) + (f_1+f_2)(\bx^{(t)})\\
        \ge& (f_1+f_2)(\bx^*) + \frac{\eta}{2}\E_{\bv^{(t)}}||G_{\eta}(\bx^{(t)})||_2^2 - \E_{\bv^{(t)}} G_{\eta}(\bx^{(t)})^T(\bx^{(t)}-\bx^*) + f_2(\bx^{(t)}) \\
        &\quad -\E_{\bv^{(t)}}(\bv^{(t)} - \nabla f_1(\bx^{(t)}))^T(\bx^{(t)}-\eta G_{\eta}(\bx^{(t)})-\bx^*)\\
        =& (f_1+f_2)(\bx^*) + \frac{\eta}{2}\E_{\bv^{(t)}}||G_{\eta}(\bx^{(t)})||_2^2 - \E_{\bv^{(t)}} G_{\eta}(\bx^{(t)})^T(\bx^{(t)}-\bx^*) + f_2(\bx^{(t)}) \\
        &\quad -\E_{\bv}(\bv^{(t)} - \nabla f_1(\bx^{(t)}))^T(\text{prox}_{-\eta f_2}(\bx^{(t)} + \eta \bv^{(t)})-\text{prox}_{-\eta f_2}(\bx^{(t)} + \eta \nabla f_1(\bx^{(t)})))\\
        \ge& (f_1+f_2)(\bx^*) + \frac{\eta}{2}\E_{\bv^{(t)}}||G_{\eta}(\bx^{(t)})||_2^2 - \E_{\bv^{(t)}} G_{\eta}(\bx^{(t)})^T(\bx^{(t)}-\bx^*) + f_2(\bx^{(t)}) \\
        &\quad -\eta \E_{\bv}||\bv^{(t)} - \nabla f_1(\bx^{(t)})||_2^2\\
        \ge& (f_1+f_2)(\bx^*) + \frac{\eta}{2}\E_{\bv^{(t)}}||G_{\eta}(\bx^{(t)})||_2^2 - \E_{\bv^{(t)}}G_{\eta}(\bx^{(t)})^T(\bx^{(t)}-\bx^*) -\sigma^2\eta \\
        =& (f_1+f_2)(\bx^*) + \frac{1}{2\eta}\E_{\bv^{(t)}}(\eta G_{\eta}(\bx^{(t)}))^T(\eta G_{\eta}(\eta \bx^{(t)}) - 2\bx^{(t)}+2\bx^*)-\sigma^2\eta \\
        =& (f_1+f_2)(\bx^{*}) + \frac{1}{2\eta}\E_{\bv^{(t)}}\left(||\eta G_{\eta}(\bx^{(t)}) - \bx^{(t)} + \bx^*||_2^2 - ||\bx^{(t)} - \bx^*||_2^2\right)-\sigma^2\eta \\
        =& (f_1+f_2)(\bx^{*}) + \frac{1}{2\eta}\E_{\bv^{(t)}}\left(||\bx^{(t+1)} - \bx^*||_2^2 - ||\bx^{(t)} - \bx^*||_2^2\right)-\sigma^2\eta.
    \end{align*}
    Then we take expectation through all the randomness and sum them up, we have
    \begin{align*}
        &\sum_{t=0}^{T-1}(\E(f_1+f_2)(\bx^{(t+1)}) + \E(f_1+f_2)(\bx^{(t)}))\\
        \ge& \sum_{t=0}^{T-1}\E\left((f_1+f_2)(\bx^{*}) + \frac{1}{2\eta}\left(||\bx^{(t+1)} - \bx^*||_2^2 - ||\bx^{(t)} - \bx^*||_2^2\right)\right)-T\sigma^2\eta \\
        =& T(f_1+f_2)(\bx^*) + \frac{1}{2\eta}\E||\bx^{(T)} - \bx^*||_2^2 - \frac{1}{2\eta}||\bx^{(0)} - \bx^*||_2^2-T\sigma^2\eta \\
        \ge& T(f_1+f_2)(\bx^*) - \frac{1}{2\eta}||\bx^{(0)} - \bx^*||_2^2-T\sigma^2\eta\\
        \ge& T(f_1+f_2)(\bx^*) - \frac{1}{2\eta}\Delta^2-T\sigma^2\eta.
    \end{align*}
    Then we plug in $\eta = 1/(\beta + \frac{\sigma}{\Delta}\sqrt{2T})$, we have
    \begin{align*}
        &\E\max_{t=0,1,2,\dots,T} (f_1+f_2)(\bx^{(t)})\\
        \ge& \max_{t=0,1,2,\dots,T} \E(f_1+f_2)(\bx^{(t)})\\
        \ge& \frac{1}{2T}\sum_{t=0}^{T-1}(\E(f_1+f_2)(\bx^{(t+1)}) + \E(f_1+f_2)(\bx^{(t)}))\\
        \ge& \frac{1}{2}(f_1+f_2)(\bx^*) - \frac{1}{2T}\left(\frac{1}{2\eta}\Delta^2+T\sigma^2\eta\right)\\
        \ge& \frac{1}{2}(f_1+f_2)(\bx^*) - \frac{\beta \Delta^2}{4T} - \frac{\sigma \Delta}{\sqrt{2T}}.
    \end{align*}
\end{proof}

\subsection{Proofs of Lemma \ref{lem:generalGradient}}
The following lemma is the more detailed version of Lemma~\ref{lem:generalGradient}.
\begin{lemma}[Detailed version of Lemma~\ref{lem:generalGradient}]\label{lem-lipschitz-smoothness-general}
If functions $h_v(\bx)$'s are $L_h$-Lipschitz, then function $\hat g_{\cR}(\bx)$ 
is $(\nu^{(1)}(\cR)nL_h)$-Lipschitz, and function $g(\bx)$ is $(n^2L_h)$-Lipschitz. 
If functions $h_v(\bx)$'s are $\beta_h$-smooth, then function $\hat g_{\cR}(\bx)$ is
$(\nu^{(1)}(\cR) n \beta_h + \nu^{(2)}(\cR) n L_h^2)$-smooth.
The gradient of function $\hat g_{\cR}(\bx)$ is
%
%
    \[\nabla \hat g_{\cR}(\bx) = \frac{n}{\theta}\sum_{R\in\cR}\sum_{v'\in R}\nabla h_{v'}(\bx)\left(\prod_{v\in R,v\neq v'}(1-h_v(\bx))\right),\]
    and can be computed in time $O(\sum_{R\in\cR}|R|(1+T_h))$ if we assume that the gradient of $h_v(\bx)$ can be generated in time $O(T_h)$. 
\end{lemma}

\begin{proof}[Proof of Lemma \ref{lem-lipschitz-smoothness-general}]
    First we have the following formula for the gradient of $\hat g_{\cR}(\bx)$.
    \begin{align*}
        \nabla \hat g_{\cR}(\bx) =& \nabla\frac{n}{\theta}\sum_{R\in\cR}\left(1- \prod_{v\in R}(1-h_v(\bx))\right) \\
        =&\frac{n}{\theta}\sum_{R\in\cR}\nabla\left(1-\prod_{v\in R}(1-h_v(\bx))\right) \\
        =&\frac{n}{\theta}\sum_{R\in\cR}\sum_{v'\in R}\nabla h_{v'}(\bx)\left(\prod_{v\in R,v\neq v'}(1-h_v(\bx))\right).
    \end{align*}
    Next we show that we can generate the exact gradient of $\hat g_{\cR}(\bx)$ in $O(\sum_{R\in\cR}|R|)$ time(assuming that generating the gradient of $h_v$ needs $O(1)$ time). First without loss of generality, we can assume that $(1-h_v(\bx) \neq 0$. Otherwise, if $1-h_u(\bx) = 0$ and $u\in R$, then $\prod_{v\in R,v\neq u}(1-h_u(\bx)) = 0$, and the problem is simpler. Then we can compute $\prod_{v\in R}(1-h_v(\bx))$ in $O(|R|)$ time and then compute $\prod_{v\in R,v\neq v'}(1-h_v(\bx))$ in $O(1)$ time. Then computing $\sum_{v'\in R}\nabla h_{v'}(\bx)\left(\prod_{v\in R,v\neq v'}(1-h_v(\bx))\right)$ needs another $O(|R| T_h)$ time, and the total time complexity to compute the gradient $\nabla \hat g_{\cR}(\bx)$ is $O(\sum_{R\in\cR}|R|(1+T_h))$. 
    Then we compute the gradient of $g(\bx)$.
    \begin{align*}
    \nabla g(\bx) =& \sum_{S\subseteq V}\sigma(S)\left(\sum_{u'\in S}\nabla_x h_{u'}(\bx)\left(\prod_{u\in S,u\neq u'}h_u(\bx)\right)\left(\prod_{v\notin S}(1-h_v(\bx))\right)\right. \\
    &\quad\quad\quad\quad - \left.\sum_{v'\notin S}\nabla_x h_{v'}(\bx)\left(\prod_{u\in S}h_u(\bx)\right)\left(\prod_{v\notin S,v\neq v'}(1-h_v(\bx))\right) \right) \\
    =& \sum_{u'\in V}\left[\sum_{S:u'\in S}\sigma(S)\left(\prod_{u\in S,u\neq u'}h_u(\bx)\right)\left(\prod_{v\notin S}(1-h_v(\bx))\right)\cdot\nabla h_{u'}(\bx)\right. \\
    &\quad\quad\quad\quad - \left.\sum_{T:u'\notin T}\sigma(T)\left(\prod_{u\in T}h_u(\bx)\right)\left(\prod_{v\notin T,v\neq u'}(1-h_v(\bx))\right)\cdot\nabla h_{u'}(\bx)\right] \\
    =& \sum_{u'\in V}f_{u'}(\bx) \nabla h_{u'}(\bx),
\end{align*}
where $f_{u'}(\bx)$ is defined as
\begin{align*}
    f_{u'}(\bx) :=& \sum_{S:u'\in S}\sigma(S)\left(\prod_{u\in S,u\neq u'}h_u(\bx)\right)\left(\prod_{v\notin S}(1-h_v(\bx))\right) \\
    &\quad - \sum_{T:u'\notin T}\sigma(T)\left(\prod_{u\in T}h_u(\bx)\right)\left(\prod_{v\notin T,v\neq u'}(1-h_v(\bx))\right).
\end{align*}
    We know that
    \begin{align*}
        &|f_{u'}(\bx)|\\
        =& \biggr|\sum_{S:u'\in S}\sigma(S)\left(\prod_{u\in S,u\neq u'}h_u(\bx)\right)\left(\prod_{v\notin S}(1-h_v(\bx))\right)  - \sum_{T:u'\notin T}\sigma(T)\left(\prod_{u\in T}h_u(\bx)\right)\left(\prod_{v\notin T,v\neq u'}(1-h_v(\bx))\right)\biggr| \\
    \le& \biggr|\sum_{S:u'\in S}\sigma(S)\left(\prod_{u\in S,u\neq u'}h_u(\bx)\right)\left(\prod_{v\notin S}(1-h_v(\bx))\right)\biggr|  + \biggr| \sum_{T:u'\notin T}\sigma(T)\left(\prod_{u\in T}h_u(\bx)\right)\left(\prod_{v\notin T,v\neq u'}(1-h_v(\bx))\right)\biggr| \\
    =& \sum_{S:u'\in S}\sigma(S)\left(\prod_{u\in S,u\neq u'}h_u(\bx)\right)\left(\prod_{v\notin S}(1-h_v(\bx))\right)  + \sum_{T:u'\notin T}\sigma(T) \left(\prod_{u\in T}h_u(\bx)\right)\left(\prod_{v\notin T,v\neq u'}(1-h_v(\bx))\right) \\
    \le& n\sum_{S:u'\in S}\left(\prod_{u\in S,u\neq u'}h_u(\bx)\right)\left(\prod_{v\notin S}(1-h_v(\bx))\right) + n\sum_{T:u'\notin T}\left(\prod_{u\in T}h_u(\bx)\right)\left(\prod_{v\notin T,v\neq u'}(1-h_v(\bx))\right) \\
    =& n h_{u'}(\bx) + n(1-h_{u'}(\bx))\\
    =& n.
    \end{align*}
    Then we have
    \begin{align*}
        ||\nabla g(\bx)||_2 =& ||\sum_{u'\in V}f_{u'}(\bx) \nabla h_{u'}(\bx)||_2 \\
        \le& \sum_{u'\in V}|f_{u'}(\bx)|\cdot ||\nabla h_{u'}(\bx)||_2 \\
        \le& \sum_{u'\in V}n L_h \\
        \le& n^2 L_h.
    \end{align*}
    Then we know that $g(\bx)$ is $2n^2L_h$-Lipschitz. We also have
    \begin{align*}
        ||\nabla\hat g_{\cR}(\bx)||_2 =& \biggr|\biggr| \frac{n}{\theta}\sum_{R\in\cR}\sum_{v'\in R}\nabla h_{v'}(\bx)\left(\prod_{v\in R,v\neq v'}(1-h_v(\bx))\right)\biggr|\biggr|_2 \\
        \le& \frac{n}{\theta}\sum_{R\in\cR}\sum_{v'\in R}||\nabla h_{v'}(\bx)||_2\left(\prod_{v\in R,v\neq v'}(1-h_v(\bx))\right) \\
        \le& \nu^{(1)}(\cR) n L_h.
    \end{align*}
    So the function $\hat g_{\cR}(\bx)$ is also $n^2L_h$-Lipschitz.
    Then we show the smoothness of the function $\hat g_{\cR}(\bx)$. We have
    \begin{align*}
        &||\nabla \hat g_{\cR}(\bx) - \nabla \hat g_{\cR}(\by)||_2\\
        =& \biggr|\biggr|\frac{n}{\theta}\sum_{R\in\cR}\sum_{v'\in R}\nabla h_{v'}(\bx)\left(\prod_{v\in R,v\neq v'}(1-h_v(\bx))\right) - \frac{n}{\theta}\sum_{R\in\cR}\sum_{v'\in R}\nabla h_{v'}(\by)\left(\prod_{v\in R,v\neq v'}(1-h_v(\by))\right)\biggr|\biggr|_2 \\
        \le& \frac{n}{\theta}\sum_{R\in\cR}\sum_{v'\in R}\biggr|\biggr|\nabla h_{v'}(\bx)\left(\prod_{v\in R,v\neq v'}(1-h_v(\bx))\right)- \nabla h_{v'}(\by)\left(\prod_{v\in R,v\neq v'}(1-h_v(\by))\right)\biggr|\biggr|_2\\
        =& \frac{n}{\theta}\sum_{R\in\cR}\sum_{v'\in R}\biggr|\biggr|\nabla h_{v'}(\bx)\left(\prod_{v\in R,v\neq v'}(1-h_v(\bx))\right) - \nabla h_{v'}(\by)\left(\prod_{v\in R,v\neq v'}(1-h_v(\bx))\right) \\
        &\quad + \nabla h_{v'}(\by)\left(\prod_{v\in R,v\neq v'}(1-h_v(\bx))\right)- \nabla h_{v'}(\by)\left(\prod_{v\in R,v\neq v'}(1-h_v(\by))\right)\biggr|\biggr|_2\\
        \le&  \underbrace{\frac{n}{\theta}\sum_{R\in\cR}\sum_{v'\in R}\biggr|\biggr|\nabla h_{v'}(\bx)\left(\prod_{v\in R,v\neq v'}(1-h_v(\bx)) \right) - \nabla h_{v'}(\by)\left(\prod_{v\in R,v\neq v'}(1-h_v(\bx))\right)\biggr|\biggr|_2}_{\mathbb A}\\
        &\quad + \underbrace{\frac{n}{\theta}\sum_{R\in\cR}\sum_{v'\in R}\biggr|\biggr|\nabla h_{v'}(\by)\left(\prod_{v\in R,v\neq v'}(1-h_v(\bx))\right) - \nabla h_{v'}(\by)\left(\prod_{v\in R,v\neq v'}(1-h_v(\by))\right)\biggr|\biggr|_2}_{\mathbb B}.
    \end{align*}
    For term $\mathbb A$, we have
    \begin{align*}
        \mathbb A =& \frac{n}{\theta}\sum_{R\in\cR}\sum_{v'\in R}\biggr|\prod_{v\in R,v\neq v'}(1-h_v(\bx))\biggr|\cdot || \nabla h_{v'}(\bx) - \nabla h_{v'}(\by)||_2 \\
        \le& \frac{n}{\theta}\sum_{R\in\cR}\sum_{v'\in R} \cdot 1\cdot \beta_h ||\bx - \by||_2 \\
        \le& \beta_h \nu^{(1)}(\cR)  n ||\bx - \by||_2,
    \end{align*}
    where we use the assumption that $h_v(\bx)$ is $\beta_h$-smooth. As for the term $\mathbb B$, we first have
    \begin{align*}
        &\biggr|\prod_{v\in R\setminus\{v'\}=\{v_1,\dots,v_{|R|-1}\}}(1-h_v(\bx)) - \prod_{v\in R\setminus\{v'\}}(1-h_v(\by))\biggr| \\
        =& \biggr|\prod_{v\in R\setminus\{v'\}}(1-h_v(\bx)) - \prod_{i=1}^{|R|-2}(1-h_{v_i}(\bx))\cdot (1-h_{v_{|R|-1}}(\by)\\
        &\quad +  \prod_{i=1}^{|R|-2}(1-h_{v_i}(\bx))\cdot (1-h_{v_{|R|-1}}(\by) - \prod_{i=1}^{|R|-3}(1-h_{v_i}(\bx))\cdot \prod_{j=|R|-2}^{|R|-1}(1-h_{v_j}(\by))\\
        &\quad +\cdots+(1-h_{v_1}(\bx))\cdot\prod_{j=2}^{|R|-1}(1-h_{v_j}(\by))- \prod_{v\in R\setminus\{v'\}}(1-h_v(\by))\biggr|\\
        \le& \sum_{i=1}^{|R|-1}|1-h_{v_i}(\bx)-1+h_{v_i}(\by)| \\
        \le& |R|L_h ||\bx-\by||_2,
    \end{align*}
    where the last inequality comes from the $L_h$-lipschitz property of the function $h_v(\bx)$. Then we have
    \begin{align*}
        \mathbb B \le& \frac{n}{\theta}\sum_{R\in\cR}\sum_{v'\in R}\biggr|\prod_{v\in R\setminus\{v'\}=\{v_1,\dots,v_{|R|-1}\}}(1-h_v(\bx)) - \prod_{v\in R\setminus\{v'\}}(1-h_v(\by))\biggr|\cdot ||\nabla h_{v'}(\by)||_2\\
        \le& \frac{n}{\theta}\sum_{R\in\cR}\sum_{v'\in R} |R|L_h ||\bx-\by||_2\cdot ||\nabla h_{v'}(\by)||_2\\
        \le& \nu^{(2)}(\cR) n L_h^2 ||\bx-\by||_2.
    \end{align*}
    Then we know that the function is $(\nu^{(1)}(\cR) n \beta_h + \nu^{(2)}(\cR) n L_h^2)$-smooth.
\end{proof}

\subsection{Proof of Lemma \ref{lem-proximal-onenorm-twonorm}}

{\lemproximalonenorm*}

\begin{proof}[Proof of Lemma \ref{lem-proximal-onenorm-twonorm}]
    First, it is easy to know that if $c(\bx) = ||\bx||_2$ and $\cD = \R_{+}^d$, then the proximal step can be finished in time $O(d)$. In this case, the set $\cP$ is defined as $\cP = \{\bx|\ ||\bx||_2\le k,\bx\succeq 0\}$, and proximal step is defined as
    \[\text{prox}_{-\eta s}(\bx) := \argmin_{\by\in\cP}-\eta s(\by) + \frac{1}{2}||\bx-\by||^2 = \argmin_{\by\in\cP}\eta\lambda ||\by||_2 + \frac{1}{2}||\bx-\by||^2.\]
    It is obvious that $\by$ should lies in the line generated by $0$ and $\bx$, and we can solve for $\by$ by using the basic technique for solving optimal value of a uni-variate quadratic function.
    
    Then we show that how to do the proximal step when $c(\bx) = ||\bx||_1$ and $\cD = \R_{+}^d$. In this case, we know that $\cP = \{\bx|\ ||\bx||_1\le k,\bx\succeq 0\}$, and we want to find $\by\in\cP$ to minimize
    \[\text{prox}_{-\eta s}(\bx) := \argmin_{\by\in\cP}-\eta s(\by) + \frac{1}{2}||\bx-\by||^2 = \argmin_{\by\in\cP}\eta\lambda ||\by||_1 + \frac{1}{2}||\bx-\by||^2.\]
    We use $\eta' = \eta\cdot\lambda$ for convenience. The proximal step can also be written as minimizing
    \[\min_{y_i \ge 0,\sum_i y_i \le k}\eta' \sum_{i}y_i + \frac{1}{2}\sum_{i}(x_i-y_i)^2.\]
    First, we have $y_i \le \max\{x_i-\eta',0\}$ for all $i$. If not, suppose $y_i > \max\{x_i-\eta',0\}$, then we let $y_i' = \max\{x_i-\eta',0\}$. The new solution has smaller value and is also feasible.
    
    Then we also have the following property: if $\by$ is optimal, then for any $i\neq j$, if $y_i,y_j\neq 0$, then we have $x_i - y_i = x_j-y_j$. Suppose that $x_i - y_i > x_j-y_j$, then let $\varepsilon$ be sufficiently small. Let $y_i' = y_i+\varepsilon$ and $y_j' = y_j-\varepsilon$, the new solution also lies in set $\cP$, but the function value of $\eta' \sum_{i}y_i + \frac{1}{2}\sum_{i}(x_i-y_i)^2$ become smaller.
    
    We also have: suppose $\by$ is optimal, then if $x_i \ge x_j$. If $y_i = 0$, then $y_j = 0$. Otherwise, suppose $y_i = 0$ but $y_j > 0$. We can pick sufficiently small $\varepsilon$ and let $y_i' = y_i+\varepsilon$ and $y_j' = y_j-\varepsilon$. The function value will decrease but the new solution is also feasible.
    
    We also have: suppose $\by$ is optimal, then if $y_i = 0$, then for any $j$ such that $y_j\neq 0$, $x_j - y_j \ge x_i - y_i$. Otherwise, we can find sufficiently small $\varepsilon$ and let $y_i' = y_i+\varepsilon$, $y_j' = y_j-\varepsilon$. The function value will decrease and the new solution is feasible.
    
    Given the previous properties, we have the following structure of optimal value $\by$. Suppose that $\{(i)\}$ is a permutation of $[d]$ such that $x_{(1)} \le x_{(2)} \le \cdots \le x_{(d)}$, then there exists $i_0\in [d]$ such that $y_{(i)} = 0$ for all $i\le i_0$, and $0 < y_{(i)} \le x_{(i)}-\lambda'$ for all $i > i_0$. Besides, for all $i,j > i_0$, we have $x_{(i)} - y_{(i)} =  x_{(j)} - y_{(j)} > x_{(i_0)} - y_{(i_0)}$. We also have one of the following: For all $i$, $y_i = \max\{x_i-\eta',0\}$; otherwise, $\sum_i y_i = k$. Then, there is only one $i_0$ that satisfy this the previous structure, and we show that we can use $O(d\log d)$ time to find $i_0$. We first use $O(d\log d)$ time to sort $x_i$ and get $x_{(i)}$. We first let $y_i = \max\{x_i-\eta',0\}$ and we test if $\sum_i y_i \le k$. If yes, then $\by$ is the optimal solution. Otherwise, we know that $\sum_i y_i = k$. We binary search for $(i_0)$, and we use $i_1$ to denote the binary search variable. Each time we set $y_{(j)} = 0$ for all $j \le i_1$, and we set $y_{(j)}$ such that $x_{(j)}-y_{(j)}$ are the same for all $j > i_1$ and $\sum_i y_i = k$. Then we test if $x_{(d)}-y_{(d)} \ge x_{(i_1)}-y_{(i_1)}$, and $y_{(j)} \ge 0$ for all $j > i_1$. If both are yes, then $i_0$ is $i_1$ and $\by$ is the optimal solution. If there exists $j>i_1$ such that $y_{(j)} < 0$, then $i_1$ should be larger. If $x_{(d)}-y_{(d)} < x_{(i_1)}-y_{(i_1)}$, then $i_1$ should be smaller. Each test needs time $O(d)$, and there are at most $O(\log d)$ binary search step, so the total complexity is $O(d \log d)$.
\end{proof}

\subsection{Proof of Lemma \ref{lem-concave-upperbound}}\label{sec-proof-concave-upperbound}

{\lemconcaveupperbound*}

\begin{proof}[Proof of Lemma \ref{lem-concave-upperbound}]
    First, we optimize the function $(\bar g_{\cR}+s)(\bx)$ in the set $\cP$. First, it is easy to know that the function $(\bar g_{\cR}+s)(\bx)$ is concave, since first we know that $s(\bx)$ is concave, $q_{v,j}(\bx)$ are concave for all $v,j$, the constant $1$ is also concave. Then because addition of 2 concave function is also concave, and the point-wise minimum of 2 concave function is also concave, so the function $(\bar g_{\cR}+s)(\bx)$ is concave. Then since we assume that the function $(\bar g_{\cR}+s)(\bx)$ is $L_{\bar g+s}$-Lipschitz, by the projected subgradient descent, in $O\left(\frac{L_{\bar g+s}}{\varepsilon^2}\right)$ iteration, we can get a solution $\by$ such that $(\bar g_{\cR}+s)(\by) \ge \max_{\bx\in\cP}(\bar g_{\cR}+s)(\bx) - \varepsilon$. Then we show that
    \[(\hat g_{\cR}+s)(\by) \ge \left(1-\frac{1}{e}\right)\max_{\bx\in\cP}(\hat g_{\cR}+s)(\bx) - \varepsilon.\]
    From Proposition \ref{prop-concave-upperbound}, we can know that
    \[\left(1-\frac{1}{e}\right)\bar g_{\cR}(\bx) \le \hat g_{\cR}(\bx) \le \bar g_{\cR}(\bx),\forall \bx\in\cP,\]
    and because $s(\bx)$ is non-negative on set $\cP$, we have
    \[\left(1-\frac{1}{e}\right)(\bar g_{\cR}+s)(\by) \le (\hat g_{\cR}+s)(\by) \le (\bar g_{\cR}+s)(\by).\]
    Then we have
    \begin{align*} 
    (\hat g_{\cR}+s)(\by) \ge& \left(1-\frac{1}{e}\right)(\bar g_{\cR}+s)(\by)\\
    \ge& \left(1-\frac{1}{e}\right)\left(\max_{\bx\in\cP}(\bar g_{\cR}+s)(\bx) - \varepsilon\right)\\
    \ge& \left(1-\frac{1}{e}\right)\max_{\bx\in\cP}(\hat g_{\cR}+s)(\bx) - \varepsilon.
    \end{align*}
\end{proof}

\subsection{Proof of Lemma \ref{lem:independentGradient}}

The following lemma is a more detailed version of Lemma~\ref{lem:independentGradient}.
\begin{lemma}[Detailed version of Lemma~\ref{lem:independentGradient}]\label{lem-lipschitz-independent}
	Suppose that functions $q_{v,j}(x_j)$'s are $L_q$-Lipschitz, then 
function $g(\bx)$ is $n^2\sqrt{d}L_q$-Lipschitz, and 
functions $\hat g_{\cR}(\bx)$ and $\bar g_{\cR}(\bx)$ are
$\nu^{(1)}(\cR) n \sqrt{d}L_q$-Lipschitz.
The subgradient of the function $\bar g_{\cR}(\bx)$ is
    \begin{align*}
        \partial \bar g_{\cR}(\bx) &= \frac{n}{\theta}\sum_{R\in \cR}\partial \min\left\{1,\sum_{j\in [d], v\in                                        R}q_{v,j}(x_j)\right\}\\
                         &=\frac{n}{\theta}\sum_{R\in \cR}\left\{
                            \begin{aligned}
                            0 & \text{, if }\sum_{j\in [d],v\in R} q_{v,j}(x_j)\ge 1 \\
                            \sum_{v\in R, j\in [d]}\nabla q_{v,j}(x_j) & \text{, if }\sum_{j\in [d],v\in R} q_{v,j}(x_j)<1
                            \end{aligned}
                            \right.\\
\end{align*}
    and can be computed in time $O(T_q)$ if we assume that the gradient and function value of $q_{v,j}(\bx)$ can be generated in time $O(\sum_{R\in\cR}|R|(1+T_q))$.
\end{lemma}

\begin{proof}[Proof of Lemma \ref{lem-lipschitz-independent}]
    First recall that $h_v(\bx) = 1 - \prod_{j\in [d]} (1-q_{v,j}(x_j))$. Then since $q_{v,j}(\bx)$ is $L_q$-Lipschitz, it can be easily shown that $h_v(\bx)$ is $L_q\sqrt{d}$-Lipschitz, since
    \begin{align*}
        |h_v(\bx) - h_v(\by)| =& \bigg|1 - \prod_{j\in [d]} (1-q_{v,j}(x_j)) - 1 + \prod_{j\in [d]} (1-q_{v,j}(x_j))\bigg|\\
        \le& \bigg|\prod_{j=1}^d (1-q_{v,j}(x_j)) - (1-q_{v,j}(y_1))\prod_{j=2}^d (1-q_{v,j}(x_j))\bigg|\\
        &\quad +\bigg|(1-q_{v,j}(y_1))\prod_{j=2}^d (1-q_{v,j}(y_j)) - \prod_{j=1}^2(1-q_{v,j}(y_j))\prod_{j=3}^d (1-q_{v,j}(x_j))\bigg|\\
        &\quad +\cdots +\bigg|\prod_{j=1}^{d-1}(1-q_{v,j}(y_j))\cdot (1-q_{v,j}(x_d)) - \prod_{j=1}^d(1-q_{v,j}(y_j))\bigg|\\
        \le& \sum_{j=1}^d|q_{v,j}(x_j)-q_{v,j}(y_j)|\\
        \le& L_q\cdot ||\bx-\by||_1\\
        \le& L_q\sqrt{d}||\bx-\by||_2.
        \end{align*}
        Then from the previous lemma (Lemma \ref{lem-lipschitz-smoothness-general}), we 
        can see that $g(\bx)$ is $n^2\sqrt{d}L_q$-Lipschitz and
        $\hat g_{\cR}(\bx)$ is all $\nu^{(1)}(\cR) n \sqrt{d}L_q$-Lipschitz. We also have
        \begin{align*}
            |\bar g_{\cR}(\bx) - \bar g_{\cR}(\by)| \le& \frac{n}{\theta}\sum_{R\in \cR}\bigg|\min\left\{1,\sum_{j\in [d], v\in                                        R}q_{v,j}(x_j)\right\}-\min\left\{1,\sum_{j\in [d], v\in R}q_{v,j}(y_j)\right\}\bigg| \\
            \le& \frac{n}{\theta}\sum_{R\in \cR}\bigg|\sum_{j\in [d], v\in R}q_{v,j}(x_j)-\sum_{j\in [d], v\in R}q_{v,j}(y_j)\bigg|\\
            \le& \frac{n}{\theta}\sum_{R\in \cR}\sum_{j\in [d], v\in R}|q_{v,j}(x_j)-q_{v,j}(y_j)|\\
            \le& \frac{n}{\theta}\sum_{R\in \cR}\sum_{j\in [d], v\in R} L_q|x_j-y_j| \\
            \le& \frac{n}{\theta}\sum_{R\in \cR}\sum_{v\in R}L_q||\bx-\by||_1 \\
            \le& \nu^{(1)}(\cR) n L_q\sqrt{d}||\bx-\by||_2.
        \end{align*}
        As for the subgradient of $\bar g_{\cR}(\bx)$ and the time complexity to generate the subgradient, it is trivial. Note that in the time complexity $\sum_{R\in\cR}|R|(1+T_q)$, the constant $1$ is used for basic operations.
\end{proof}

\section{Omitted Proofs for Time Complexity (Theorem \ref{thm:timesimple})} \label{app:timeComplexity}
In this section, we present our time complexity results for $\ProxGradRIS$ and $\UpperGradRIS$. We divide this section into 2 parts. In the first part, we show the time complexity of $\ProxGradRIS$ and $\UpperGradRIS$ in a general form (Theorem \ref{thm-proximal-time} and \ref{thm-concave-upperbound-time}), and then Theorem \ref{thm:timesimple} will become a corollary. However, the time complexity bound in Theorem \ref{thm-proximal-time} and \ref{thm-concave-upperbound-time} is too conservative and cannot reflect the empirical running time in experiments. In order to close this gap, we give time complexity bounds based on the moments of RR-set size in the second part. Then we will give some statistics of the moments of RR-set size in the section describing our experiments.

\subsection{Proof of Theorem \ref{thm:timesimple}}
In this subsection, we prove Theorem \ref{thm:timesimple}. 
We actually prove the full version of the running times for the two algorithms
 in Theorems~\ref{thm-proximal-time} and~\ref{thm-concave-upperbound-time}.
Our proof follows from the original proof of the time complexity of IMM algorithm. First, we have to show a lemma(Lemma \ref{lem-high-prob-lb-and-opt}), which states that in Algorithm \ref{alg-sampling}, with high probability, we output $\text{LB}$ does not differ so much from the optimal value $\text{OPT}_{g+s}$. For convenience, all the notations follow from Algorithm \ref{alg-sampling}. We use $\text{OPT}_{g+s} = (g+s)(\bx_{g+s}^*)$ to denote the maximum value of $(g+s)$ in set $\cP$.

\begin{lemma}\label{lem-high-prob-lb-and-opt}
    For every $i = 1,2,\dots,\lfloor\log_2(n+\lambda k)\rfloor-1$, if $\text{OPT}_{g+s} \ge (1+\varepsilon/3+\varepsilon')^2 \cdot x_i / (\alpha - \varepsilon/3)$, then with probability at least $1-\frac{1}{2n^{\ell}\cdot \cN(\cP,\frac{\varepsilon/3}{L_2} x_i)\cdot \log_2{(n+\lambda k)}}$, $(\hat g_{\cR_i}+s)(\by_i) \ge (\alpha-\varepsilon/3)\text{OPT}_{g+s} / (1+\varepsilon/3+\varepsilon')$ and $(\hat g_{\cR_i}+s)(\by_i) \ge (1+\varepsilon/3+\varepsilon')x_i$.
\end{lemma}

\begin{proof}
    First we know that $(\hat g_{\cR_i}+s)(\by_i) \ge (\alpha - \varepsilon/3)\max_{\by\in\cP}(\hat g_{\cR_i}+s)(\by)$. For any $\cR$, let $X_i^{\cR}(\bx) = 1-\prod_{v\in R_i}(1-h_v(\bx))$ where $\cR = \{R_1,R_2,\dots,R_{\theta}\}$, then $X_i^{\cR}(\bx) \in [0,1]$ and $\hat g_{\cR}(\bx) = \frac{n}{\theta}\sum_{i=1}^{\theta}X_i^{\cR}(\bx)$. We also know that $X_i^{\cR}(\bx)$ are independent. By Chernoff Bound(Proposition \ref{prop-chernoff}), we have
    \begin{align*}
        &\Pr\{(\hat g_{\cR_i}+s)(\by_i) \le (\alpha-\varepsilon/3)\text{OPT}_{g+s} / (1+\varepsilon/3+\varepsilon')\}\\
        \le&\Pr\{(\alpha - \varepsilon/3)\max_{\by\in\cP}(\hat g_{\cR_i}+s)(\by) \le (\alpha-\varepsilon/3)\text{OPT}_{g+s} / (1+\varepsilon/3+\varepsilon')\}\\
        =& \Pr\{\max_{\by\in\cP}(\hat g_{\cR_i}+s)(\by) \le \text{OPT}_{g+s} / (1+\varepsilon/3+\varepsilon')\}\\
        \le& \Pr\{(\hat g_{\cR_i}+s)(\bx_{g+s}^*) \le \text{OPT}_{g+s} / (1+\varepsilon/3+\varepsilon')\}\\
        =& \Pr\{\hat g_{\cR_i}(\bx_{g+s}^*) \le \text{OPT}_{g+s} / (1+\varepsilon/3+\varepsilon')-s(\bx_{g+s}^*)\}\\
        =& \Pr\left\{\sum_{j=1}^{\theta_i}X_j^{\cR_i}(\bx_{g+s}^*)\frac{n}{\theta_i} \le \text{OPT}_{g+s} / (1+\varepsilon/3+\varepsilon')-s(\bx_{g+s}^*)\right\}\\
        =& \Pr\left\{\sum_{j=1}^{\theta_i}X_j^{\cR_i}(\bx_{g+s}^*) \le \frac{\theta_i}{n}\text{OPT}_{g+s} / (1+\varepsilon/3+\varepsilon')-\frac{\theta_i}{n}s(\bx_{g+s}^*)\right\}\\
        =& \Pr\left\{\sum_{j=1}^{\theta_i}X_j^{\cR_i}(\bx_{g+s}^*)-\frac{\theta_i}{n}g(\bx_{g+s}^*) \le \frac{\theta_i}{n}\text{OPT}_{g+s} / (1+\varepsilon/3+\varepsilon')-\frac{\theta_i}{n}g(\bx_{g+s}^*)-\frac{\theta_i}{n}s(\bx_{g+s}^*)\right\}\\
        \le& \exp\left(-\frac{\varepsilon'^2}{2(1+\varepsilon/3+\varepsilon')^2}\frac{\theta_i}{n}\text{OPT}_{g+s}\right)\\
        \le& \exp\left(- \frac{(1+\varepsilon/3+\varepsilon')^2 \cdot x_i\varepsilon'^2n\cdot\left(2+\frac{2}{3}\varepsilon'\right)\cdot\left(\ln\cN(\cP,\frac{\varepsilon/3}{L_2} x_i)+\ell \ln n+\ln 2 + \ln\log_2{(n+\lambda k)}\right)}{2(\alpha - \varepsilon/3)(1+\varepsilon/3+\varepsilon')^2n\varepsilon'^2 x_i}\right)\\
        \le& \frac{1}{2n^{\ell}\cdot \cN(\cP,\frac{\varepsilon/3}{L_2} x_i)\cdot \log_2{(n+\lambda k)}}.
    \end{align*}
    Combined with the assumption that $\text{OPT}_{g+s} \ge (1+\varepsilon/3+\varepsilon')^2 \cdot x_i / (\alpha - \varepsilon/3)$, we have: if $(\hat g_{\cR_i}+s)(\by_i) \ge (\alpha-\varepsilon/3)\text{OPT}_{g+s} / (1+\varepsilon/3+\varepsilon')$, then
    \[(\hat g_{\cR_i}+s)(\by_i) \ge (\alpha-\varepsilon/3)\text{OPT}_{g+s} / (1+\varepsilon/3+\varepsilon') \ge \frac{\alpha-\varepsilon/3}{1+\varepsilon/3+\varepsilon'}(1+\varepsilon/3+\varepsilon')^2 \cdot x_i / (\alpha - \varepsilon/3) = (1+\varepsilon/3+\varepsilon')x_i.\]
\end{proof}

Then, with the previous high probability lemma, we can upper bound $\E[\theta^{(1)}+\theta^{(2)}+\theta_{i_{ret}}]$, where $i_{ret}$ denote the index of Algorithm \ref{alg-sampling} that break from the for-loop. We use $L_1,L_2$ to denote the Lipschitz constant of the function $(g+s)$ and the function $(\hat g_{\cR_i}+s)$.

\begin{lemma}\label{lem-expectation-upperbound}
    Let $i_{ret}$ denote the index of Algorithm \ref{alg-sampling} that break from the for-loop. If we have $n+\lambda k = O(n^{\ell})$, then
    \[\E[\theta^{(1)}+\theta^{(2)}+\theta_{i_{ret}}] = O\left(\frac{n\cdot\ln\left(n^{\ell}\cN(\cP,\frac{\varepsilon}{L_1+L_2}\text{LB})\right)}{\varepsilon^2\text{OPT}_{g+s}}\right).\]
\end{lemma}

\begin{proof}
    Let $x_{ret}$ denote the value $x_{i_{ret}}$ when Algorithm \ref{alg-sampling} break from the for-loop. We first prove that,
    \[\E\left[\frac{1}{x_{ret}}\right] = O\left(\frac{1+\frac{n+\lambda k}{n^{\ell}}}{\text{OPT}_{g+s}}\right).\]
    Let $i$ denote the smallest index such that $\text{OPT}_{g+s} \ge (1+\varepsilon/3+\varepsilon')^2 \cdot x_i / (\alpha - \varepsilon)$. From the previous lemma(Lemma \ref{lem-high-prob-lb-and-opt}), we know that with probability at least $1-\frac{1}{2n^{\ell}\cdot \cN(\cP,\frac{\varepsilon}{L_2} x_i)\cdot \log_2{(n+\lambda k)}}$, we have 
    \[(\hat g_{\cR_i}+s)(\by_i) \ge (1+\varepsilon/3+\varepsilon')x_i,\]
    so the algorithm will break from for-loop and $x_{ret}\ge x_i$. Let $\cA$ denote the event that $(\hat g_{\cR_i}+s)(\by_i) \ge (1+\varepsilon/3+\varepsilon')x_i$. Let $X=\frac{1}{x_{ret}}$ and we have
    \[\E[X] = \E[X|\cA]\Pr[\cA]+\E[X|\lnot \cA]\Pr[\lnot\cA].\]
    We know that $\E[X|\cA] = O\left(\frac{1}{\text{OPT}_{g+s}}\right)$ and $\Pr[\cA]\le 1$. We also know that $\E[X|\lnot \cA] \le 1$ since $x_j \ge 1$. Then 
    \[\Pr[\lnot\cA] \le \frac{1}{2n^{\ell}\cdot \cN(\cP,\frac{\varepsilon/3}{L_2} x_i)\cdot \log_2{(n+\lambda k)}} = O\left(\frac{n+\lambda k}{n^{\ell} \text{OPT}_{g+s}}\right),\]
    where we use the fact that $g(\bx)\le n$ and $s(\bx) = \lambda(k-c(\bx))\le \lambda k$. Then we have
    \[\E\left[\frac{1}{x_{ret}}\right] = O\left(\frac{1+\frac{n+\lambda k}{n^{\ell}}}{\text{OPT}_{g+s}}\right) = O\left(\frac{1}{OPT_{g+s}}\right).\]
    Recall that
    \[\theta_i = \left\lceil\frac{n\cdot\left(2+\frac{2}{3}\varepsilon'\right)\cdot\left(\ln\cN(\cP,\frac{\varepsilon/3}{L_2} x_i)+\ell \ln n+\ln 2 + \ln\log_2{(n+\lambda k)}\right)}{\varepsilon'^2 x_i}\right\rceil,\]
    and
    \[\theta^{(1)} = \frac{8n\cdot \ln\left(4n^{\ell}\right)}{\text{LB}\cdot (\alpha-\varepsilon/3)^2\varepsilon^2/9}, \quad \theta^{(2)} = \frac{2\alpha'\cdot n\cdot\ln\left(4n^{\ell}\cN(\cP,\frac{\varepsilon/3}{L_1+L_2}\text{LB})\right)}{(\varepsilon/3 - \frac{1}{4}(\alpha-\varepsilon/3)^2 \varepsilon/3)^2\text{LB}}.\]
    We know that $\text{LB} \ge x_{ret}$, so we have
    \[\E[\theta^{(1)}+\theta^{(2)}+\theta_{i_{ret}}] = O\left(\frac{n\cdot\ln\left(n^{\ell}\cN(\cP,\frac{\varepsilon/3}{L_1+L_2}\text{LB})\right)}{\varepsilon^2/9\text{OPT}_{g+s}}\right) = O\left(\frac{n\cdot\ln\left(n^{\ell}\cN(\cP,\frac{\varepsilon}{L_1+L_2}\text{LB})\right)}{\varepsilon^2\text{OPT}_{g+s}}\right).\]
\end{proof}

The above lemma is the main lemma for the time complexity of Algorithm \ref{alg-meta}. Next, we show some existing propositions and lemmas, which will help to prove the time complexity. The following lemmas and assumptions comes from \cite{tang15}, and we use the lemmas and assumptions to prove Theorem \ref{thm-proximal-time} and \ref{thm-concave-upperbound-time}.

\begin{definition}[Martingale]
    A series of random variables $X_1,X_2,\dots$ is a martingale, if for all $i\ge 1$, $\E[|X_i|] < +\infty$ and $\E[X_{i+1}|X_1,\dots,X_i] = X_i$.
\end{definition}

\begin{lemma}[Sufficient Condition for Martingale]\label{lem-sufficent-martingle}
    Suppose $X_1,X_2,\dots, X_t$ are $t$ random variables on $[0,1]$, which satisfy that there exists a constant $\mu$, $\E[X_{i}|X_1,\dots,X_{i-1}] = \mu$ for all $i\in [t]$. Let $Z_i = \sum_{j=1}^i(X_j-\mu)$, then $Z_1,Z_2,\dots,Z_t$ is a martingale.
\end{lemma}

\begin{proposition}[Stopping Time of Martingale]\label{prop-stopping-time}
    Let random variable $\tau$ denote the stopping time of a martingale $\{X_i\}_{i\ge 1}$. If there is a constant $c$ that is independent to $\{X_i\}_{i\ge 1}$ and $\tau \le c$, then $\E[X_{\tau}] = \E[X_1]$.
\end{proposition}

Besides, we have the following assumptions.

\begin{restatable}{assumption}{assumptriggeringset}\label{assump-triggering-set}
	In the triggering model, the time complexity to sample a triggering set $T_v$ for any $v$ is $O(|N^{-}(v)|)$, where $N^{-}(v)$ is the set of in-neighbors of node $v$.
\end{restatable}

\begin{restatable}{assumption}{assumpoptimalvalue}\label{assump-optimal-value}
	We assume that $\max_{\bx\in\cP} g(\bx) \ge \max_{v\in V}\sigma(v)$,
	where $\sigma(v)$ denote the influence spread of node $v$. Besides, we have $\text{OPT}_{g+s} \ge \max_{v\in V}\sigma(v)$.
\end{restatable}

Given a set $R\subseteq V$, let $\omega(R)$ denote the sum of in-degree of nodes in $R$. Based on Assumption \ref{assump-triggering-set}, we know that generating a RR-set needs time $O(\omega(R)+1)$. Next, we use $\text{EPT} = \E[\omega(R)]$ to denote the expectation of $\omega(R)$, and we have the following lemma, which also comes from \cite{tang15}.

\begin{lemma}\label{lem-ept}
    \[\text{EPT} = \E[\omega(R)] = \frac{m}{n}\cdot \E_{\tilde v}[\sigma(\tilde v)],\]
    where $\sigma(\tilde v)$ denote the influence spread of node $\tilde v$.
\end{lemma}

With the help of the previous lemma and proposition, we can prove Theorem \ref{thm-proximal-time} and Theorem \ref{thm-concave-upperbound-time}, and their corresponding corollaries.

\begin{restatable}{theorem}{thmproximaltime}\label{thm-proximal-time}
    Under Assumption \ref{assump-triggering-set} and Assumption \ref{assump-optimal-value}. Suppose that the proximal step can be finished in time $T_{prox}$,. Besides, suppose that $c(\bx)$ is $L_c$-Lipschitz and $\beta_c$-smooth, $h_v(\bx)$ are $L_h$-Lipschitz and $\beta_h$-smooth for all $v$, and the gradient of $h_v(\bx)$ and $c(\bx)$ can be generated in time $T_h$ and $T_c$. We also assume that the balance variable $\lambda$ is also a constant and $n+\lambda k\le n^l$. Then the expected running time of Algorithm \ref{alg-meta} with proximal gradient descent oracle is bounded by
    \begin{align*}
        O\left(\frac{\beta_h n^2 + L_h^2 n^3}{\varepsilon}\cdot \left((1+T_h)\frac{(m+n)\cdot\ln\left(n^{\ell}\cN(\cP,\frac{\varepsilon}{2n^2L_h+2\lambda L_c})\right)}{\varepsilon^2}+\log_2(n+\lambda k)T_c\right)\right.\\
        +\left.\frac{\log_2(n+\lambda k)(\beta_h n^2 + L_h^2 n^3 )T_{prox}}{\varepsilon}\right).
    \end{align*}
\end{restatable}

\begin{proof}[Proof of Theorem \ref{thm-proximal-time}]
    Let $i_{ret}$ denote the index where Algorithm \ref{alg-sampling} break from the for-loop. First, note that in the sampling procedure, Algorithm \ref{alg-sampling} generate at most $2\theta_{i_{ret}}+\tilde\theta$ number of RR-sets, which is bounded by $O(\theta^{(1)}+\theta^{(2)}+\theta_{i_{ret}})$. We use $\tau = 2\theta_{i_{ret}}+\tilde\theta$ to denote the number of RR-sets. By Assumption \ref{assump-triggering-set}, generating such number of RR-sets needs time $O(\sum_{j=1}^{\tau}(\omega(R_j)+1))$. Let $W_i = \sum_{j=1}^{i}(\omega(R_j)-\text{EPT})$ for $i=1,\dots,\tau$. Because the procedure to generate $R_j$ is independent to the procedure that generates $R_1,\dots,R_{j-1}$, we have $\E[\omega(R_j)|\omega(R_1),\dots,\omega(R_{j-1})] = \E[\omega(R_j)] = \text{EPT}$. From Lemma \ref{lem-sufficent-martingle}, we know that $\{W_i\}_{i\le\tau}$ is a martingale, and $\tau$ is a stopping time. Obviously, $\tau$ has an upper bound, which can be derived by setting $x_{ret} =  1$ and $LB = 1$. Then by the stopping time theorem(Proposition \ref{prop-stopping-time}), the time complexity for generating RR-sets is $O(\E[\sum_{j=1}^{\tau}(\omega(R_j)+1)]) = O(\E[W_{\tau}]+\E[\tau\cdot(\text{EPT}+1)]) = O(\E[\theta^{(1)}+\theta^{(2)}+\theta_{i_{ret}}]\cdot (\text{EPT}+1)) = O(\E[\tau]\cdot(\text{EPT}+1))$.
    
    Then, we count the total time complexity for calling the proximal gradient oracle. Each time we use the proximal gradient to optimize $\hat g_{\cR_i}+s$ or $\hat g_{\cR}+s$, we will use $O\left(\frac{\beta}{\varepsilon x_i}\right)$ or $O\left(\frac{\beta}{\varepsilon \text{LB}}\right)$ number of iterations, where $\beta$ is the smoothness constant for $\hat g_{\cR_i}$ and $\hat g_{\cR}$. We can lower bound $x_i$ and $\text{OB}$ by $1$, and by Lemma \ref{lem-lipschitz-smoothness-general}, we have $\hat g_{\cR_i}$ and $\hat g_{\cR}$ are all $(\beta_h n^2 + L_h^2 n^3 )$-smooth. So each time the number of iteration is upper bounded by $N = O\left(\frac{\beta_h n^2 + L_h^2 n^3 }{\varepsilon}\right)$. Each time we need a proximal step, so the total proximal step time complexity is bounded by $O\left(\frac{\log_2(n+\lambda k)(\beta_h n^2 + L_h^2 n^3 )T_{prox}}{\varepsilon}\right)$. Then we consider the time complexity that generate the gradient. By Lemma \ref{lem-lipschitz-smoothness-general}, the time complexity of generating gradient is $O(\sum_{R\in\cR_i}|R|(1+T_h)+T_c)$ or $O(\sum_{R\in\cR}|R|(1+T_h)+T_c)$, then the total time complexity for generating gradient is bounded by
    \[O\left(N\cdot \left(2\sum_{R\in\cR_{i_{ret}}}|R|(1+T_h)+\sum_{R\in\cR}|R|(1+T_h)+\log_2(n+\lambda k)T_c\right)\right).\]
    By the martingale stopping time theorem(Proposition \ref{prop-stopping-time}), we know that
    \[\E[2\sum_{R\in\cR_{i_{ret}}}|R|+\sum_{R\in\cR}|R|] = \E[\tau]\cdot \E[|R_1|],\]
    and $|R_1|\le\omega(R_1)+1$ since each RR-set is weak connected. Then we have
    \[\E[2\sum_{R\in\cR_{i_{ret}}}|R|+\sum_{R\in\cR}|R|] \le \E[\tau]\cdot (\text{EPT}+1),\]
    and the total expected time complexity of generating gradient is bounded by
    \[O\left(N\cdot \left((1+T_h)(\text{EPT}+1)\E[\tau]+\log_2(n+\lambda k)T_c\right)\right).\]
    Then the total expected time complexity is bounded by(the time to generate RR-set is far less than the time to generate all of the gradient)
    \begin{align*}
        O\left(\frac{\beta_h n^2 + L_h^2 n^3 }{\varepsilon}\cdot \left((1+T_h)(\text{EPT}+1)\E[\tau]+\log_2(n+\lambda k)T_c\right)\right.\\
        +\left.\frac{\log_2(n+\lambda k)(\beta_h n^2 + L_h^2 n^3 )T_{prox}}{\varepsilon}\right).
    \end{align*}
    By Assumption \ref{assump-optimal-value}, Lemma \ref{lem-lipschitz-smoothness-general}, Lemma \ref{lem-expectation-upperbound} and Lemma \ref{lem-ept}, we have
    \begin{align*}
        (\text{EPT}+1)\E[\tau] =& O\left(\frac{n\cdot\ln\left(n^{\ell}\cN(\cP,\frac{\varepsilon}{2n^2L_h+2\lambda L_c})\text{LB}\right)}{\varepsilon^2\text{OPT}_{g+s}}\left(\frac{m}{n}\cdot \E_{\tilde v}[\sigma(\tilde v)]+1\right)\right) \\
        \le& O\left(\frac{(m+n)\cdot\ln\left(n^{\ell}\cN(\cP,\frac{\varepsilon}{2n^2L_h+2\lambda L_c})\right)}{\varepsilon^2}\right).
    \end{align*}
    The total time complexity is bounded by(let $T_h$ and $T_c$ be constants)
    \begin{align*}
        O\left(\frac{\beta_h n^2 + L_h^2 n^3 }{\varepsilon}\cdot \left((1+T_h)\frac{(m+n)\cdot\ln\left(n^{\ell}\cN(\cP,\frac{\varepsilon}{2n^2L_h+2\lambda L_c})\right)}{\varepsilon^2}+\log_2(n+\lambda k)T_c\right)\right.\\
        +\left.\frac{\log_2(n+\lambda k)(\beta_h n^2 + L_h^2 n^3 )T_{prox}}{\varepsilon}\right).
    \end{align*}
\end{proof}

\begin{restatable}{theorem}{thmconcaveupperboundtime}\label{thm-concave-upperbound-time}
    Under Assumption \ref{assump-triggering-set} and Assumption \ref{assump-optimal-value}. Suppose that the projection step can be finished in time $T_{proj}$,. Besides, suppose that $c(\bx)$ is $L_c$-Lipschitz and $\beta_c$-smooth, $q_{v,j}(x_j)$ are $L_q$-Lipschitz for all $v,j$, and the gradient and function value of $q_{v,l}(x_j)$ and $c(\bx)$ can be generated in time $T_q$ and $T_c$($L_c,L_q,\beta_c,T_q,T_c$ are all constants). We also assume that the balance variable $\lambda$ is also a constant and $n+\lambda k\le n^l$. Then the expected running time of Algorithm \ref{alg-meta} with proximal gradient descent oracle is bounded by
    \begin{align*}
        O\left(\frac{(n^2\sqrt{d}L_q+\lambda L_c)^2}{\varepsilon^2}\cdot \left((1+T_h)\frac{(m+n)\cdot\ln\left(n^{\ell}\cN(\cP,\frac{\varepsilon}{2n^2\sqrt{d}L_q+\lambda 2L_c})\right)}{\varepsilon^2}+\log_2(n+\lambda k)T_c\right)\right.\\
        +\left.\frac{(n^2\sqrt{d}L_q+\lambda L_c)^2\log_2(n+\lambda k)T_{proj}}{\varepsilon^2}\right).
    \end{align*}
\end{restatable}

\begin{proof}[Proof of Theorem \ref{thm-concave-upperbound-time}]
    The proof of this theorem is almost the same as the previous one. We only have to change the iteration step from $O\left(\frac{\beta_h n^2 + L_h^2 n^3 }{\varepsilon}\right)$ to $O\left(\frac{(n^2\sqrt{d}L_q+\lambda L_c)^2}{\varepsilon^2}\right)$ and the Lipschitz constant for $\hat g_{\cR}$ and $g$ from $n^2L_h$ to $n^2\sqrt{d}L_q$.
\end{proof}

The above 2 theorems are the main theorem of the the time complexity. First, we have the fact that when $c(\bx) = ||\bx||_1$ or $c(\bx) = ||\bx||_2$, then proximal step and the projection step can all be finished in time $\tilde O(d)$, and the time to generate the function value and gradient of $c(\bx)$ is also $O(d)$. Besides, as shown in \cite{van2014probability}, 
	we know that the covering number 
	$\cN(\cP,\varepsilon) \le \cN(\mathbb B_1(k),\varepsilon) \le \cN(\mathbb B_2(k),\varepsilon) \le \left(3k/\varepsilon\right)^d$. Then since $||\bx||_1 \ge ||\bx_2||$, we have $\mathbb B_1(k) \subseteq\mathbb B_2(k)$, and $\cN(\mathbb B_1(k),\varepsilon) \le \left(3k/\varepsilon\right)^d$. We can directly get the time complexity bounds in Theorem \ref{thm:timesimple}.
	$\cN(\cP,\varepsilon) \le \cN(\mathbb B_1(k),\varepsilon) \le \cN(\mathbb B_2(k),\varepsilon) \le \left(3k/\varepsilon\right)^d$. Then since $||\bx||_1 \ge ||\bx_2||$, we have $\mathbb B_1(k) \subseteq\mathbb B_2(k)$, and $\cN(\mathbb B_1(k),\varepsilon) \le \left(3k/\varepsilon\right)^d$. We can directly get the time complexity bounds in Theorem \ref{thm:timesimple}.
	
	\subsection{Time Complexity Bound Based on the Moments of RR Set Size}
	\label{app:timeMoments}
	In this subsection, we give the proof of our time complexity bounds based on the moments of the size of the RR-sets and the optimal value $\text{OPT}_{g+s}$ for $\ProxGradRIS$ and $\UpperGradRIS$. We use $\nu^{(1)}$, $\nu^{(2)}$ and $\nu^{(3)}$ to denote the first, second and third moments of the mean size of a random generated RR-set. Formally, we have
	\[\nu^{(1)}= \E_{R}[|R|],\quad\nu^{(2)}=E_{R}[|R|^2], \quad\nu^{(3)}=E_{R}[|R|^3].\]
	Given the time complexity theorem in Theorem \ref{thm-proximal-time-heu} and \ref{thm-concave-upperbound-time-heu} and some statistics for the variable $\nu^{(1)},\nu^{(2)},\nu^{(3)}$, we can know why in experiments, $\ProxGradRIS$ and $\UpperGradRIS$ do not make so much difference with the heuristic greedy algorithm in terms of the running time.
	
	To derive the time complexity bounds based on the moments of the size of RR-sets, we need to slightly revise our sampling algorithm. In Algorithm \ref{alg-sampling}, we use the previous generated RR-sets and only generate $\theta_i - \theta_{i-1}$ RR-sets in round $i$ (line~\ref{line:generateRRset1} of Algorithm~\ref{alg-sampling}). 
	However, in this subsection, we assume that we generate $\theta_i$ RR-sets in round $i$ and we do not use the previous generated RR-sets. 
	This is because we need to use the martingale stopping time for each round, and to construct a martingale, we need the empirical moments $\nu^{(1)}(\cR_i),\nu^{(2)}(\cR_i), \nu^{(3)}(\cR_i)$ to be independent to $\cR_{j}$ for all $j < i$.
	The following theorem summarizes the time complexity of $\ProxGradRIS$ 
	with the above resampling of RR sets adjustment.
	
	\begin{restatable}{theorem}{thmproximaltimeheu}\label{thm-proximal-time-heu}
    Under Assumption \ref{assump-triggering-set} and Assumption \ref{assump-optimal-value}. Suppose that the proximal step can be finished in time $T_{prox}$,. Besides, suppose that $c(\bx)$ is $L_c$-Lipschitz and $\beta_c$-smooth, $h_v(\bx)$ are $L_h$-Lipschitz and $\beta_h$-smooth for all $v$, and the gradient of $h_v(\bx)$ and $c(\bx)$ can be generated in time $T_h$ and $T_c$. We also assume that the balance variable $\lambda$ is also a constant and $n+\lambda k\le n^l$. Then the expected running time of
    $\ProxGradRIS$ (revised by resampling of RR sets in each sampling iteration)
    is bounded by
    \begin{align*}
       & O\left(\frac{n^2(\nu^{(2)}\beta_h + \nu^{(3)} L_h^2)\ln\left(n^{\ell}\cN(\cP,\frac{\varepsilon}{2n^2\sqrt{d}L_q+\lambda 2L_c})\right)}{\varepsilon^3\text{OPT}_{g+s}} (1+T_h)\right.\\
       & \left.+ \frac{\nu^{(1)} n \beta_h + \nu^{(2)} n L_h^2}{\varepsilon}\log_2(n+\lambda k)(T_c+T_{prox})
        +\frac{(m+n)\cdot\ln\left(n^{\ell}\cN(\cP,\frac{\varepsilon}{2n^2\sqrt{d}L_q+\lambda 2L_c})\right)}{\varepsilon^2}\right).
    \end{align*}
    \end{restatable}
	
	\begin{proof}[Proof of Theorem \ref{thm-proximal-time-heu}]
    The first part is to compute the expected time to generate the RR-sets. This part is similar to the same part in the proof of Theorem \ref{thm-proximal-time} and \ref{thm-concave-upperbound-time}. Let $i_{ret}$ denote the index where Algorithm \ref{alg-sampling} break from the for-loop, and we know that Algorithm \ref{alg-sampling} generate at most $2\theta_{i_{ret}}+\tilde\theta$ number of RR-sets, which is bounded by $O(\theta^{(1)}+\theta^{(2)}+\theta_{i_{ret}})$. We use $\tau = 2\theta_{i_{ret}}+\tilde\theta$ to denote the number of RR-sets. By Assumption \ref{assump-triggering-set}, generating such number of RR-sets needs time $O(\sum_{j=1}^{\tau}(\omega(R_j)+1))$. Let $W_i = \sum_{j=1}^{i}(\omega(R_j)-\text{EPT})$ for $i=1,\dots,\tau$. Because the procedure to generate $R_j$ is independent to the procedure that generates $R_1,\dots,R_{j-1}$, we have $\E[\omega(R_j)|\omega(R_1),\dots,\omega(R_{j-1})] = \E[\omega(R_j)] = \text{EPT}$. From Lemma \ref{lem-sufficent-martingle}, we know that $\{W_i\}_{i\le\tau}$ is a martingale, and $\tau$ is a stopping time. Obviously, $\tau$ has an upper bound, which can be derived by setting $x_{ret} =  1$ and $LB = 1$. Then by the stopping time theorem(Proposition \ref{prop-stopping-time}), the expected time complexity for generating RR-sets is $O(\E[\sum_{j=1}^{\tau}(\omega(R_j)+1)]) = O(\E[W_{\tau}]+\E[\tau\cdot(\text{EPT}+1)]) = O(\E[\theta^{(1)}+\theta^{(2)}+\theta_{i_{ret}}]\cdot (\text{EPT}+1)) = O(\E[\tau]\cdot(\text{EPT}+1))$. Based on Assumeption \ref{assump-triggering-set} and \ref{assump-optimal-value} and Lemma \ref{lem-expectation-upperbound} and \ref{lem-ept}, we know that
    \begin{align*}
        O(\E[\tau]\cdot(\text{EPT}+1)) =& O\left(\frac{n\cdot\ln\left(n^{\ell}\cN(\cP,\frac{\varepsilon}{L_1+L_2}\text{LB})\right)}{\varepsilon^2\text{OPT}_{g+s}}\cdot\left(\frac{m}{n}\cdot \E_{\tilde v}[\sigma(\tilde v)]+1\right)\right)\\
        =&O\left(\frac{(m+n)\cdot\ln\left(n^{\ell}\cN(\cP,\frac{\varepsilon}{2n^2\sqrt{d}L_q+\lambda 2L_c})\right)}{\varepsilon^2}\right).
    \end{align*}
    Then we count the time to generate the gradients and the proximal steps. In round $i$, Algorithm \ref{alg-sampling} will use the RR-sets $\cR_i$ with $|\cR_i| = \theta_i$. In $\ProxGradRIS$, in the $i$-th round, the algorithm iterates for $O(\frac{\nu^{(1)}(\cR_i) n \beta_h + \nu^{(2)}(\cR_i) n L_h^2}{\varepsilon})$ times. Then, the total time to generate the gradient and to complete the proximal steps in round $i$ are bounded by
    \begin{align*}
        &O\left(\frac{\nu^{(1)}(\cR_i) n \beta_h + \nu^{(2)}(\cR_i) n L_h^2}{\varepsilon}\left(\sum_{R\in\cR_i}|R|(1+T_h)+T_c\right)\right)\\
        =& O\left(\frac{\nu^{(1)}(\cR_i) n \beta_h + \nu^{(2)}(\cR_i) n L_h^2}{\varepsilon}\left(\theta_i \nu^{(1)}(\cR_i)(1+T_h)+T_c\right)\right)\\
        \le& O\left(\frac{\nu^{(2)}(\cR_i) n \beta_h + \nu^{(3)}(\cR_i) n L_h^2}{\varepsilon}\theta_i (1+T_h)+\frac{\nu^{(1)}(\cR_i) n \beta_h + \nu^{(2)}(\cR_i) n L_h^2}{\varepsilon}T_c\right),
    \end{align*}
    where the last line comes from the fact that $\nu^{(1)}(\cR_i)^2 \le \nu^{(2)}(\cR_i)$ and $\nu^{(1)}(\cR_i) \nu^{(2)}(\cR_i)\le \nu^{(3)}(\cR_i)$. $\nu^{(1)}(\cR_i)^2 \le \nu^{(2)}(\cR_i)$ comes directly from the Cauchy-Schwatz inequality, and $\nu^{(1)}(\cR_i) \nu^{(2)}(\cR_i)\le \nu^{(3)}(\cR_i)$ comes from the fact that for any $a,b\ge 0$, we have $a^3 + b^3 \ge ab^2 + a^2b$. Then we sum up all of them from $i=1$ to $i=i_{ret}$ and bound the expectation. We want to bound
    \[O\left(\E\left[\sum_{i=1}^{i_{ret}}\frac{\nu^{(2)}(\cR_i) n \beta_h + \nu^{(3)}(\cR_i) n L_h^2}{\varepsilon}\theta_i (1+T_h)+\frac{\nu^{(1)}(\cR_i) n \beta_h + \nu^{(2)}(\cR_i) n L_h^2}{\varepsilon}T_c\right]\right).\]
    First note that the second part is easy to compute, and it is just bounded by
    \[O\left(\frac{\nu^{(1)} n \beta_h + \nu^{(2)} n L_h^2}{\varepsilon}\log_2(n+\lambda k)T_c\right),\]
    since we know that $i_{ret} \le \log_2(n+\lambda k)$. Then we bound the first term. For convenience, we use the following notations
    \[Y_i = \frac{\nu^{(2)}(\cR_i) n \beta_h + \nu^{(3)}(\cR_i) n L_h^2}{\varepsilon}\theta_i (1+T_h), \quad Z_i = \frac{\nu^{(2)} n \beta_h + \nu^{(3)} n L_h^2}{\varepsilon}\theta_i (1+T_h).\]
    
    It is easy to show that $\E Y_i = Z_i$, and let $W_j = \sum_{i=1}^j(Y_i - Z_i)$. We will show that $W_j$ is a martingale. This is due to the fact that since we generate new RR-sets in each round, and each RR-set is independent to others, then $Y_i - Z_i$ is independent to all of the information before round $i$, then $W_j$ is a martingale. Also notice that $i_{ret}$ is a stopping time, since we only decide if the sampling procedure will stop based on the previous information. By the Martingale Stopping Time theorem(Proposition \ref{prop-stopping-time}), we have
    \[\E \sum_{i=1}^{i_{ret}}(Y_i - Z_i) = \E[Y_1-Z_1]\E[i_{ret}] = 0.\]
    Then we only have to bound $O(\E \sum_{i=1}^{i_{ret}}Z_i)$. We have the following
    \[\sum_{i=1}^{i_{ret}}Z_i = \sum_{i=1}^{i_{ret}}\frac{\nu^{(2)} n \beta_h + \nu^{(3)} n L_h^2}{\varepsilon}\theta_i (1+T_h) \le 2\frac{\nu^{(2)} n \beta_h + \nu^{(3)} n L_h^2}{\varepsilon}\theta_{i_{ret}} (1+T_h).\]
    
    Similarly, the total time to generate the gradient and to complete the proximal steps which are not counted in the sampling procedure(Algorithm \ref{alg-sampling}) are bounded by
    \[O\left(\frac{\nu^{(2)}(\cR) n \beta_h + \nu^{(3)}(\cR) n L_h^2}{\varepsilon}\tilde\theta (1+T_h)+\frac{\nu^{(1)}(\cR) n \beta_h + \nu^{(2)}(\cR) n L_h^2}{\varepsilon}T_c\right).\]
    Taking the expectation, because we generate the RR-sets independently and thus $\nu^{(2)}(\cR),\nu^{(3)}(\cR)$ are independent to $\tilde\theta$, we have the following bound for the expected time
    \[O\left(\frac{\nu^{(2)} n \beta_h + \nu^{(3)} n L_h^2}{\varepsilon}\E[\tilde\theta] (1+T_h)+\frac{\nu^{(1)} n \beta_h + \nu^{(2)} n L_h^2}{\varepsilon}T_c\right).\]
    Then from Lemma \ref{lem-expectation-upperbound}, we have
    \[\frac{\nu^{(2)} n \beta_h + \nu^{(3)} n L_h^2}{\varepsilon}\E[2\theta_{i_{ret}}+\tilde\theta] (1+T_h) \le \frac{n^2(\nu^{(2)}\beta_h + \nu^{(3)} L_h^2)\ln\left(n^{\ell}\cN(\cP,\frac{\varepsilon}{2n^2\sqrt{d}L_q+\lambda 2L_c})\right)}{\varepsilon^3\text{OPT}_{g+s}} (1+T_h).\]
    Also, it is easy to show that the expected time for proximal steps is bounded by
    \[\frac{\nu^{(1)} n \beta_h + \nu^{(2)} n L_h^2}{\varepsilon}\log_2(n+\lambda k)T_{prox},\]
    since the algorithm applies the proximal steps in each iteration, which is the same as calling the gradient of function $c(\bx)$.
    
    Combining all of them together(generating gradient, proximal steps, generating RR-sets), we know that the expected time complexity is
    \begin{align*}
        &O\left(\frac{n^2(\nu^{(2)}\beta_h + \nu^{(3)} L_h^2)\ln\left(n^{\ell}\cN(\cP,\frac{\varepsilon}{2n^2\sqrt{d}L_q+\lambda 2L_c})\right)}{\varepsilon^3\text{OPT}_{g+s}} (1+T_h)\right.\\
       & \left.+ \frac{\nu^{(1)} n \beta_h + \nu^{(2)} n L_h^2}{\varepsilon}\log_2(n+\lambda k)(T_c+T_{prox})
        +\frac{(m+n)\cdot\ln\left(n^{\ell}\cN(\cP,\frac{\varepsilon}{2n^2\sqrt{d}L_q+\lambda 2L_c})\right)}{\varepsilon^2}\right).
    \end{align*}
    \end{proof}
    
We remark that, when comparing 
Theorem~\ref{thm-proximal-time-heu} with Theorem~\ref{thm-proximal-time}, 
if we do the following relaxations for the corresponding terms
in the bound of Theorem~\ref{thm-proximal-time-heu}: $\nu^{(2)}\le \nu^{(1)} \cdot n$, 
$\nu^{(3)}\le \nu^{(1)} \cdot n^2$, $\nu^{(1)} \le n$, and $\nu^{(2)} \le n^2$, 
together with $\nu{(1)} /  \text{OPT}_{g+s} \le m/n$ (Lemma~\ref{lem-ept}), 
then we will have the bound given in Theorem~\ref{thm-proximal-time}.
This indicates how loose is the bounds we give in Theorem~\ref{thm:timesimple} in the main text.
In our experiments (Section~\ref{sec:experiment}), we will demonstrate how much this relaxation is 
	numerically in our dataset.

The following theorem summarizes our result for the $\UpperGradRIS$ algorithm,
	revised with the resampling of RR sets in each iteration step as described before.
    \begin{restatable}{theorem}{thmconcaveupperboundtimeheu}\label{thm-concave-upperbound-time-heu}
    Under Assumption \ref{assump-triggering-set} and Assumption \ref{assump-optimal-value}. Suppose that the projection step can be finished in time $T_{proj}$,. Besides, suppose that $c(\bx)$ is $L_c$-Lipschitz and $\beta_c$-smooth, $q_{v,j}(x_j)$ are $L_q$-Lipschitz for all $v,j$, and the gradient and function value of $q_{v,l}(x_j)$ and $c(\bx)$ can be generated in time $T_q$ and $T_c$($L_c,L_q,\beta_c,T_q,T_c$ are all constants). We also assume that the balance variable $\lambda$ is also a constant and $n+\lambda k\le n^l$. Then the expected running time of $\UpperGradRIS$ 
    (revised by resampling of RR sets in each sampling iteration)
     is bounded by
    \begin{align*}
        O\left(\frac{n(n^2d\nu^{(3)}L_q^2 + 2\lambda n\sqrt{d}\nu^{(2)}L_qL_c + \lambda^2L_c^2)\ln\left(n^{\ell}\cN(\cP,\frac{\varepsilon}{2n^2\sqrt{d}L_q+\lambda 2L_c})\right)}{\varepsilon^4\text{OPT}_{g+s}} (1+T_h)\right.\\
        + \frac{(n^2d\nu^{(3)}L_q^2 + 2\lambda n\sqrt{d}\nu^{(2)}L_qL_c + \lambda^2L_c^2)}{\varepsilon^2}\log_2(n+\lambda k)(T_c+T_{proj})\\
        \left.+\frac{(m+n)\cdot\ln\left(n^{\ell}\cN(\cP,\frac{\varepsilon}{2n^2\sqrt{d}L_q+\lambda 2L_c})\right)}{\varepsilon^2}\right).
    \end{align*}
    \end{restatable}
\begin{proof}[Proof of Theorem \ref{thm-concave-upperbound-time-heu}]
    The proof of this theorem is almost the same as the previous one. We only have to change the iteration step from $O\left(\frac{\nu^{(1)}(\cR) n \beta_h + \nu^{(2)}(\cR) n L_h^2}{\varepsilon}\right)$ to $O\left(\frac{(\nu^{(1)}(\cR) n \sqrt{d}L_q+\lambda L_c)^2}{\varepsilon^2}\right)$ and the Lipschitz constant for $\hat g_{\cR}$ and $g$ from $n^2L_h$ to $n^2\sqrt{d}L_q$.
\end{proof}

\section{Properties of the Original Function $g(\bx)$} \label{app:original}
In this section, we discuss the properties of the original function $g(\bx)$. We will compute the gradient of $g(\bx)$ and show how to compute the stochastic gradient of $g(\bx)$. Then we give upper bounds for the smoothness constant of the function $g(\bx)$ and
	the variance of 
	the stochastic gradient estimator $\widehat{\nabla g}(\bx)$
	in terms of its L2-norm difference with the
	true gradient, defined as $\text{Var} = \E[||\widehat{\nabla g}(\bx)-\nabla g(\bx)||_2^2]$.
These bounds would lead to our settings of the step size and number of iterations for the
	stochastic gradient method on the original objective function $g(\bx)$.


The following lemma provides the exact gradient formula for $g(\bx)$.
\begin{lemma} \label{lem:gradientOrg}
The gradient of function $g(\bx)$ can be written as:
\begin{align*}
\nabla g(\bx) =& \sum_{u'\in V}\left[\sum_{S:u'\notin S}\left(\sigma(S\cup\{u'\})-\sigma(S)\right)\left(\prod_{u\in S}h_u(\bx)\right)\left(\prod_{v\notin S \cup \{u'\}}(1-h_v(\bx))\right)\cdot\nabla h_{u'}(\bx)\right]. \\
\end{align*}
\end{lemma}

\begin{proof}
We first recall the expression of the function $g(\bx)$,

\[g(\bx) = \E_S[\sigma(S)] = \sum_{S\subseteq V}\left(\sigma(S)\left(\prod_{u\in S}h_u(\bx)\right)\left(\prod_{v\notin S}(1-h_v(\bx))\right)\right).\]

The gradient of $g(\bx)$ is given as follow
\begin{align*}
    \nabla g(\bx) =& \nabla_{\bx}\left(\sum_{S\subseteq V}\left(\sigma(S)\left(\prod_{u\in S}h_u(\bx)\right)\left(\prod_{v\notin S}(1-h_v(\bx))\right)\right)\right) \\
    =& \sum_{S\subseteq V}\sigma(S)\nabla_{\bx}\left(\left(\prod_{u\in S}h_u(\bx)\right)\left(\prod_{v\notin S}(1-h_v(\bx))\right)\right) \\
    =& \sum_{S\subseteq V}\sigma(S)\left(\sum_{u'\in S}\nabla_{\bx} h_{u'}(\bx)\left(\prod_{u\in S,u\neq u'}h_u(\bx)\right)\left(\prod_{v\notin S}(1-h_v(\bx))\right)\right. \\
    &\quad\quad\quad\quad - \left.\sum_{v'\notin S}\nabla_{\bx} h_{v'}(\bx)\left(\prod_{u\in S}h_u(\bx)\right)\left(\prod_{v\notin S,v\neq v'}(1-h_v(\bx))\right) \right).
\end{align*}

We then rearrange the gradient term, we have
\begin{align*}
    \nabla g(\bx)
    =& \sum_{u'\in V}\left[\sum_{S:u'\in S}\sigma(S)\left(\prod_{u\in S,u\neq u'}h_u(\bx)\right)\left(\prod_{v\notin S}(1-h_v(\bx))\right)\cdot\nabla h_{u'}(\bx)\right. \\
    &\quad\quad\quad\quad - \left.\sum_{T:u'\notin T}\sigma(T)\left(\prod_{u\in T}h_u(\bx)\right)\left(\prod_{v\notin T,v\neq u'}(1-h_v(\bx))\right)\cdot\nabla h_{u'}(\bx)\right] \\
    =& \sum_{u'\in V}f_{u'}(\bx) \nabla h_{u'}(\bx),
\end{align*}
where $f_{u'}(\bx)$ is defined as
\begin{align}
\label{eq:fuprime}
    f_{u'}(\bx) :=
    \sum_{S:u'\notin S}\left(\sigma(S\cup\{u'\})-\sigma(S)\right)\left(\prod_{u\in S}h_u(\bx)\right)\left(\prod_{v\notin S \cup \{u'\}}(1-h_v(\bx))\right).
\end{align}
\end{proof}

 Lemma~\ref{lem:gradientOrg} provides the exact gradient formula for $g(\bx)$, but
 	it involves an exponentially large number of summation terms and cannot be efficiently
 	computed.
Instead, we use Lemma~\ref{lem:gradientOrg} to define a simple stochastic gradient estimator 
	$\widehat{\nabla g}(\bx)$ as the unbiased estimator of $\nabla g(\bx)$.
\begin{definition}[Stochastic Gradient Estimator]
\label{def:stochastic_estimator}
For any vector $x \in \cD$,  we construct a stochastic gradient estimator 
	$\widehat{\nabla g}(\bx)$ as follows:
First, we same a node $u'\in V$ uniformly at random from $V$.
Then, we sample a subset $S \subseteq V \setminus \{u'\}$ 
	according to $h_u(\bx)$ for $u\in V\setminus \{u'\}$, more specifically, 
	for each $u\in V\setminus \{u'\}$, we include $u$ in $S$ with probability $h_u(\bx)$
	and exclude $u$ from $S$ with probability $1-h_u(\bx)$, and different $u$'s are sampled
	independently.
Next we sample a live-edge graph $L$ based on the triggering model, and compute the
	marginal gain of $u'$ on $S$ in graph $L$, denoted $\sigma_L(u'|S)$, which is
	the number of nodes in $L$ that can be reached from $u'$ but not from $S$.
Finally, we set $\widehat{\nabla g}(\bx) = n \cdot \sigma_L(u'|S) \nabla h_{u'}(\bx)$.
\end{definition}

It is straightforward to see that  $\widehat{\nabla g}(\bx)$ is an unbiased estimator
	of $\nabla g(\bx)$, i.e. $\E[\widehat{\nabla g}(\bx)] = \nabla g(\bx)$.

The next lemma provides the bounds on the smoothness of $g(\bx)$ and the variance of
	its stochastic gradient estimator defined above.
\begin{lemma} \label{lem:boundOriginal}
Assuming that the function $h_v(\bx)$ is $L_h$-Lipschitz and $\beta_h$-smooth, then we have the following bound for the smoothness constant $\beta_g$ of $g(\bx)$ and the variance of stochastic gradient estimator
	$\text{Var} = \E[||\widehat{\nabla g}(\bx)-\nabla g(\bx)||_2^2]$.
\begin{align}
    \beta_g \le& \beta_h n^2+2L_h^2 n^3,\\
    \text{Var} \le& 4L_h^2n^4.
\end{align}
\end{lemma}

\begin{proof}
In this proof, we use the same $f_{u'}(\bx)$ definition as given in
	Eq.~\eqref{eq:fuprime}.
we have
\begin{align*}
    ||\nabla g(\bx) - \nabla g(\by)||_2 =& ||\sum_{u'\in V}f_{u'}(\bx) \nabla h_{u'}(\bx) - \sum_{u'\in V}f_{u'}(\by) \nabla h_{u'}(\by)||_2 \\
    =& ||\sum_{u'\in V}f_{u'}(\bx) \nabla h_{u'}(\bx) - \sum_{u'\in V}f_{u'}(\bx) \nabla h_{u'}(\by)\\
    &\quad + \sum_{u'\in V}f_{u'}(\bx) \nabla h_{u'}(\by) - \sum_{u'\in V}f_{u'}(\by) \nabla h_{u'}(\by)||_2 \\
    \le& ||\sum_{u'\in V}f_{u'}(\bx) \nabla h_{u'}(\bx) - \sum_{u'\in V}f_{u'}(\bx) \nabla h_{u'}(\by)||_2\\
    &\quad + ||\sum_{u'\in V}f_{u'}(\bx) \nabla h_{u'}(\by) - \sum_{u'\in V}f_{u'}(\by) \nabla h_{u'}(\by)||_2 \\
    \le& \sum_{u'\in V}|f_{u'}(\bx)|\cdot ||\nabla h_{u'}(\bx) - \nabla h_{u'}(\by)||_2\\
    &\quad + \sum_{u'\in V}|f_{u'}(\bx) - f_{u'}(\by)|\cdot ||\nabla h_{u'}(\by)||_2 \\
    \le & \sum_{u'\in V}|f_{u'}(\bx)|\cdot \beta_h ||\bx-\by||_2 + \sum_{u'\in V}|f_{u'}(\bx) - f_{u'}(\by)|\cdot L_h,
\end{align*}
where the last inequality comes from the assumptions we made before.

Then, it is easy to see that
\begin{align*}
    f_{u'}(\bx) \le 
\sum_{S:u'\notin S} n \cdot \left(\prod_{u\in S}h_u(\bx)\right)\left(\prod_{v\notin S \cup \{u'\}}(1-h_v(\bx))\right) = n.
\end{align*}
Then we bound $|f_{u'}(\bx) - f_{u'}(\by)|$. First, we define
\[
   r_{u'}(\bx) = 
\sum_{S:u'\notin S}\left(\sigma(S\cup\{u'\})-\sigma(S)\right)\left(\prod_{u\in S}x_u\right)\left(\prod_{v\notin S \cup \{u'\}}(1-x_v)\right).
\]
%
where the input for $r_{u'}(\cdot)$ is defined as $\bx = (x_v)_{v\in V} \in \mathbb R^n$. Then we show that
\[|r_{u'}(\bx) - r_{u'}(\by)| \le 4n\sum_{v\in V}|x_v - y_v|.\]
We first show that 
\[|r_{u'}(\bx_{-w},x_w) - r_{u'}(\bx_{-w},x'_w)| \le 4n|x_w - x'_w|.\]
First, if $w = u'$, then it is obvious that $r_{u'}(x_{-w},x_w) - r_{u'}(x_{-w},x'_w) = 0$, since the variable corresponding to $w$ does not appear in the formula. Then assume that $u' \neq w$, we can get the following equation by rearranging the terms.

\begin{align*}
    &r_{u'}(\bx_{-w},x_w)\\
    =& x_w\sum_{S: u'\notin S, w\in S} \left(\sigma\left(S\cup \{u'\}\right)-\sigma(S)\right)\prod_{u\in S, u\neq w}x_u\prod_{v\notin S, v\neq u'}(1-x_v)\\
    &\quad +(1-x_w)\sum_{S:u',w\notin S}\left(\sigma(S\cup \{u'\})-\sigma(S)\right)\prod_{u\in S}x_u\prod_{v\notin S, v\neq u',w}(1-x_v)\\
\end{align*}
Then we have the following inequalities
\begin{align*}
   &|r_u'(\bx_{-w},x_w)-r_u'(\bx_{-w},x'_w)|\\
   =&\Bigg|(x_w-x'_w)\sum_{S: u'\notin S, w\in S} \left(\sigma\left(S\cup \{u'\}\right)-\sigma(S)\right)\prod_{u\in S, u\neq w}x_u\prod_{v\notin S, v\neq u'}(1-x_v)\\
   &\quad + (x'_w-x_w)\sum_{S:u',w\notin S}\left(\sigma(S\cup \{u'\})-\sigma(S)\right)\prod_{u\in S}x_u\prod_{v\notin S, v\neq u',w}(1-x_v)\Bigg|\\
   \le& 2n |x_w-x'_w|\\
\end{align*}
The last inequality is due to the two sum terms are the expectations of marginal influence. The marginal influence is smaller than $n$ due to submodularity.
 
Then given two vectors $\bx,\by\in \mathbb R^n$, we define $\bz_0 = x$, and $\bz_i = (y_1,\dots, y_i,x_{i+1},\dots,x_n)$ for $i=1,2,\dots,n-1$ and $\bz_n = \by$. Then we have

\begin{align*}
    |r_{u'}(\bx) - r_{u'}(\by)| =& |\sum_{i=0}^{n-1}(r_{u'}(\bz_i) - r_{u'}(\bz_{i+1})| \\
    \le& \sum_{i=0}^{n-1}|(r_{u'}(\bz_i) - r_{u'}(\bz_{i+1})| \\
    \le& \sum_{i=0}^{n-1}2n|x_{i+1} - y_{i+1}| \\
    =& 2n\sum_{v\in V}|x_v - y_v|.
\end{align*}

Then, let $h(\bx) := (h_u(\bx))_{u\in V}$ to be the vector for the activation probabilities, then we have $f_{u'}(\bx) = r_{u'}(h(\bx))$, and we have
\begin{align*}
    |f_{u'}(\bx) - f_{u'}(\by)| =& |r_{u'}(h(\bx)) - r_{u'}(h(\by))| \\
    \le& 2n\sum_{v\in V}|h_v(\bx) - h_v(\by)| \\
    \le& 2n\sum_{v\in V}L_h||\bx-\by||_2 \\
    =& 2L_hn^2\cdot||\bx-\by||_2,
\end{align*}
where the last inequality comes from the assumption that $h_v(\bx)$ is $L_h$-Lipschitz. Plug in all the terms, we have shown that

\begin{align*}
    ||\nabla g(\bx) - \nabla g(\by)||_2 \le & \sum_{u'\in V}|f_{u'}(\bx)|\cdot \beta_h ||\bx-\by||_2 + \sum_{u'\in V}|f_{u'}(\bx) - f_{u'}(\by)|\cdot L_h \\
    \le& \sum_{u'\in V}n\cdot \beta_h ||\bx-\by||_2 + \sum_{u'\in V} L_h\cdot wL_hn^2\cdot||\bx-\by||_2 \\
    =& \left(\beta_h n^2 + 2n^3L_h^2\right)||\bx-\by||_2.
\end{align*}

As shown above, the original function $g$ is $(\beta_h n^2+2n^3L_h^2)$-smooth. 

Next we bound $\E[||\widehat{\nabla g}(\bx)-\nabla g(\bx)||_2^2]$.
First, it is easy to show that $|f_{u'}(\bx)|$ is bounded by $n$, so we have
\[||\widehat{\nabla g}(\bx)||_2 \le \sum_{u'}|f_{u'}(\bx)|\cdot ||\nabla h_{u'}(\bx)||_2 \le n^2 L_h.\]
Then similar to the argument above, we also have
\[||\nabla g(\bx)||_2 \le  n^2 L_h.\]
Then
\[\E[||\widehat{\nabla g}(\bx)-\nabla g(\bx)||_2^2] \le \E[(2 n^2 L_h)^2] = 4L_h^2n^4 = O(n^4).\]
\end{proof}

The following corollary summarizes our stochastic gradient algorithm for the original
	objective function, which will be used in the experiment section.
It is a direct consequence of Theorem~\ref{thm-proximal-grad} and Lemma~\ref{lem:boundOriginal},
	and the proof is omitted.
\begin{corollary}
Assuming that the function $h_v(\bx)$ is $L_h$-Lipschitz and $\beta_h$-smooth. By choosing the unbiased stochastic gradient estimator shown in Definition~\ref{def:stochastic_estimator}, we run stochastic gradient algorithm which comes from \cite{HassaniSK17}. The algorithm runs $T$ rounds with gradient step size 
$\mu_t=\frac{1}{\beta_h n^2+2L_h^2 n^3+\frac{2\sqrt{2}L_h n^2}{\Delta}\sqrt{t}}$, 
where $\Delta = sup_{\bx, \by\in \cP }||\bx-\by||_2$. 
Then
\[\E\left[\max_{t=0,1,2,\dots,T} g(\bx_{t}) \right]\ge \frac{\text{OPT}}{2}-\left(\frac{\Delta^2(\beta_h n^2+2L_h^2 n^3)}{4T}+\frac{\sqrt{2} \Delta L_h n^2}{\sqrt{T}}\right).\]
\end{corollary}

The above theorem shows that with the number of iterations 
$T=O\left(\frac{d^2(\beta_h n^2+L_h^2 n^3)}{\varepsilon}+\frac{dL_h n^2}{\varepsilon^2}\right)$, 
we obtain a solution whose objective value is at least $\left(\frac{\text{OPT}}{2}-\varepsilon\right)$.

}

\end{document}